\documentclass[11pt]{article}
\usepackage{amsfonts}
\usepackage[dcucite]{harvard}
\usepackage{a4wide}
\usepackage[dvips]{graphicx}

\newcommand{\mR}{\ensuremath{{\mathbb R}}}
\newcommand{\mZ}{\ensuremath{{\mathbb Z}}}

\newcommand{\mN}{\ensuremath{{\mathbb N}}}
\newcommand{\mE}{\ensuremath{{\mathbb E}}}
\newcommand{\mP}{\ensuremath{{\mathbb P}}}
\newcommand{\ve}{\ensuremath{\varepsilon}}
\newcommand{\eqref}[1]{(\ref{#1})}
\newcommand{\Of}{\ensuremath{O}}

\newcommand{\T}{\ensuremath{\mathcal T}}

\newcommand{\hOf}{\ensuremath{\hat{\Of}}}
\renewcommand{\O}{\ensuremath{\mathcal{O}}}
\newcommand{\tO}{\ensuremath{\tilde O}}
\newcommand{\bOf}{\ensuremath{\check{\Of}}}
\newcommand{\Kp}{\ensuremath{G}}
\newcommand{\hKp}{\ensuremath{\hat{\Kp}}}
\newcommand{\Wf}{\ensuremath{\Xi_+}}

\newcommand{\hWf}{\ensuremath{\tilde{\Wf}}}
\newcommand{\hWff}{\ensuremath{\tilde{\Wf}}}

\newcommand{\Wp}{\ensuremath{\Xi_p^-}}
\newcommand{\hWp}{\ensuremath{\tilde{\Wp}}}

\newcommand{\la}{\ensuremath{\langle}}
\newcommand{\ra}{\ensuremath{\rangle}}

\newcommand{\hD}{\ensuremath{\hat \Delta}}
\newcommand{\D}{\ensuremath{\Delta}}
\newcommand{\hO}{\ensuremath{\hat \Omega_{\hat u,\Delta z}\hat \Omega_{\Delta z,\Delta z}^{-1}}}
\newcommand{\YY}{\ensuremath{\Sigma_{y2,y2}}}
\newcommand{\hYY}{\ensuremath{\hat \Sigma_{y2,y2}}}
\newcommand{\hYZ}{\ensuremath{\hat \Sigma_{y,z}}}

\newcommand{\lT}{\ensuremath{\log{T}}}
\newcommand{\llT}{\ensuremath{\log \log{T}}}

\newtheorem{theorem}{Theorem}[section]
\newtheorem{lemma}{Lemma}[section]

\newenvironment{assume}[1]{{\bf Assumption #1:}}{$\Box$}
\newenvironment{proof}{{\sc Proof:}}{$\Box$\\}

\begin{document}
\setlength{\baselineskip}{0.7cm}

\title{Asymptotic distribution of estimators in reduced rank regression settings 
when the regressors are integrated}
\author{Dietmar Bauer \thanks{Correspondence to: tel.: ++43 +50 550 6669, e-mail: Dietmar.Bauer@ait.ac.at.} \\
{\small Austrian Institute of Technology, Dynamic Transportation Systems} \\
{\small  Giefingg. 2, A-1210 Wien  }}

\maketitle
\pagestyle{empty}

\begin{abstract}
 In this paper the asymptotic distribution of estimators 
 is derived in a general regression setting where rank restrictions on a submatrix of the coefficient matrix 
are imposed and the regressors can include stationary or I(1) processes. 
 Such a setting occurs e.g. in factor models. 
 Rates of convergence are derived and the asymptotic distribution is given for least squares estimators 
 as well as fully-modified estimators. The gains in imposing the rank restrictions are investigated.
 A number of  special cases are discussed including the Johansen results in the case 
 of cointegrated VAR(p) processes.
 \end{abstract}


\section{Introduction}
In this paper a multivariable time series 
$(y_t)_{t \in \mZ}, y_t \in \mR^s,$ is modeled as a linear 
function of two processes $(z_t^r)_{t \in \mZ}, z_t^r \in \mR^{m_r}$ and
$(z_t^u)_{t \in \mZ}, z_t^u \in \mR^{m_u}$ (where 'r' stands for restricted
and 'u' for unrestricted) using the following model:
\begin{equation} \label{equ:RRR}
y_t = b_r z_t^r + b_u z_t^u + u_t, t=1,\dots,T
\end{equation}
where $b_r= O\Gamma'$ is of rank $n< \min(s, m_r)$. 
Such a situation can occur e.g. for panel data sets where both $s$ and $m_r$ are large.
Throughout all variables will be assumed to be either stationary or (co-)integrated. 
Details on the assumptions for the processes are given below. 
For the moment assume that $(u_t)_{t \in \mZ}$ is an independent identically distributed (iid) process.

In this situation the asymptotics for the OLS estimators \cite{ParkPhillips,ParkPhillips89} and fully modified \cite{Phillips95} estimators neglecting the rank restriction are well documented in the literature.   
However, neglecting the rank restriction in the case that $m_r$ and $s$ are large, the number of parameters to be estimated equals $(m_r+m_u)s$ which might require excessively large
samples in order to allow for reasonable accuracy. As an alternative then rank restricted regression (RRR) can be used in order to reduce the number of parameters greatly. 

The RRR framework of equation~\eqref{equ:RRR} is also of importance for the estimation involved in subspace methods,  \citeaffixed{L83,BauWag2001}{see e.g.}. 
In these methods a RRR of the type~\eqref{equ:RRR} is the central step in the estimation. Thus the understanding of the asymptotic properties of the corresponding estimators needs a thorough understanding 
of the asymptotic properties of estimators for \eqref{equ:RRR}.

If all involved processes are stationary the asymptotic theory of RRR estimators based on OLS is presented in \cite{Reinsel}. 
There consistency and asymptotic normality of the estimated 
coefficient matrices is stated for the (generic) special case that all singular values of $b_r$ are distinct. Further expressions for the asymptotic variance matrix are provided 
using implicitly defined quantities which, hence, are not easy to interpret or implement.  
 
For a cointegrated process $X_t$ letting 
$$
y_t = \Delta X_t = X_t-X_{t-1}, \quad z_t^r = X_{t-1}, \quad z_t^u = [\Delta X_{t-1}',\dots,\Delta X_{t-p}']'
$$
equation~\eqref{equ:RRR} corresponds to the Johansen framework \cite{Johansen}. 
Also in this case the asymptotics of quasi maximum likelihood estimators are well known. 
Although the original material focuses on the estimation of the cointegrating relations $\Gamma$ 
extracted as the right factor in the product $b_r = O \Gamma'$, the asymptotics for the full estimator $\hat b_r$ can be derived based on these results, see the evaluations in \cite{Johansen}.
The arguments given there rely on stationarity of $y_t$ and $\Gamma' X_{t-1}$
as well as on the fact that the rank restriction only restricts the coefficients corresponding to the nonstationary components of $X_{t-1}$ as will be demonstrated below.  

Equation~\eqref{equ:RRR} extends this framework by allowing for more general processes $z_t^r$ and $z_t^u$. 
It will be shown below (see Theorem~\ref{thm:repr}) that there exist nonsingular transformations $\T_y, \T_r$ such that 
$$
\tilde b_r := \T_y b_r {\mathcal T}_r^{-1} = \left[ \begin{array}{ccc} I_{c_y} & 0_{c_y \times (c_r-c_y)} & 0_{c_y\times (m_r-c_r)} \\ 0_{(s-c_y) \times c_y} & 0_{(s-c_y) \times (c_r-c_y)} & \tilde b_{23,r} \end{array} \right]  
= \underbrace{\left[ \begin{array}{cc} I_{c_y} & 0  \\ 0 & \tilde O_2 \end{array} \right]}_{\tilde O} 
\underbrace{\left[ \begin{array}{ccc} I_{c_y} & 0 & 0 \\ 0 & 0 & \tilde \Gamma_{32}' \end{array} \right]}_{\tilde \Gamma'}
$$
and in $\tilde z_t^r = \T_r z_t^r$ the first $c_r$ coordinates are integrated, the remaining ones being stationary. In the Johansen framework $c_y=0$ holds while in this paper $0 \le c_y \le n \le s$ is allowed for. 
Also in $\T_y (y_t - b_uz_t^u)$ the first $c_y$ components are integrated the remaining ones being stationary. 

In this extended situation the asymptotics of \cite{Johansen} do not apply as can be seen from the following arguments:
Using the notation $\la a_t , b_t \ra = T^{-1}\sum_{t=1}^T a_t b_t'$ for processes $(a_t)_{t \in \mZ}, (b_t)_{t \in\mZ}$ the consistency proof in Lemma 13.1. of \cite{Johansen} 
relies on solving the generalized eigenvalue problem 
$$
\lambda \la \tilde z_{t}^{r,\pi}, \tilde z_{t}^{r,\pi} \ra v - \la \tilde z_{t}^{r,\pi}, y_t^\pi \ra \la y_t^\pi , y_t^\pi\ra^{-1} \la y_t^\pi , 
\tilde z_{t}^{r,\pi}  \ra v = 0.
$$
Here $a_t^\pi$ denotes the residuals of a regression onto $z_t^u$. Consistency is shown by transforming the problem using the matrix $A_T = [\tilde \Gamma_\perp T^{-1/2},\tilde \Gamma]$ 
(changing the order of the block column to correspond to our ordering as used below) where the columns of 
the matrix $\tilde \Gamma_\perp, (\tilde \Gamma_\perp)' \tilde \Gamma_\perp = I$ span the orthogonal complement of the space spanned by the columns of $\tilde \Gamma$. 
Correspondingly for $c_y=0$ in $A_T' z_{t}^{r,\pi}$ the first components are 
nonstationary but scaled by $T^{-1/2}$ and the remaining 
ones stationary. In the transformed problem 
$$
\lambda A_T'\la \tilde z_{t}^{r,\pi}, \tilde z_{t}^{r,\pi} \ra A_T w- A_T'\la \tilde z_{t}^{r,\pi}, y_t^\pi \ra \la y_t^\pi , y_t^\pi\ra^{-1} \la y_t^\pi , \tilde z_{t}^{r,\pi} \ra A_T w= 0 
$$
all matrices converge to block diagonal matrices. Thus the corresponding eigenvalues and matrix of eigenvectors $V_T$ (with a suitable choice of the basis) converge. 
The eigenvectors of the transformed problem corresponding to the nonzero eigenvalues 
are related via $V_T = A_T^{-1}W_T
= [\tilde \Gamma,\tilde \Gamma_\perp T^{1/2}]' W_T$ (assuming without restriction of generality $\tilde \Gamma' \tilde \Gamma = I_n$) implying that $T^{1/2} \tilde \Gamma_\perp' W_T$ converges to zero in probability. No almost sure (a.s.) results and no sharper bounds on the order of convergence are provided in \cite{Johansen}.
\\
For $c_y>0$, however,  $\tilde \Gamma' \tilde z_{t}^r$ is nonstationary and 
$A_T'\la \tilde z_{t}^{r,\pi}, \tilde z_{t}^{r,\pi} \ra A_T$ does not converge. 
Using instead $\tilde A_T$ as 
$$
\tilde A_T' = \left[ \begin{array}{ccc} T^{-1/2}I_{c_y} & 0 & 0 \\ 
0 & T^{-1/2}I_{c_r-c_y} & 0 \\
 0 & 0 & \tilde \Gamma_{32}' \\ 
0 & 0 & \tilde \Gamma_{32,\perp}' \end{array} \right]
$$
where $\tilde \Gamma_{32,\perp}'\tilde \Gamma_{32}=0, \tilde \Gamma_{32,\perp}'\tilde \Gamma_{32,\perp}=I$  
leads to convergence for the generalized eigenvalue problem. Thus again $V_T = \tilde A_T^{-1}W_T$ converges, where the first $c_y$ columns corresponding to the eigenvalue $\lambda=1$
converge to $[I_{c_y},0]'$.  
Consequently also $W_T = \tilde A_TV_T$ converges. 
However, the heading $c_y \times c_y$ subblock of this matrix equals $T^{-1/2}I_{c_y}$ and hence converges to zero as does the whole block column.  
Multiplying the corresponding block column with $T^{1/2}$ the heading subblock equals the identity matrix as required, but the orders of convergence for the remaining blocks are reduced by this order and 
hence the remaining arguments in the proof of Lemma 13.1 of \cite{Johansen} do no longer apply. 
Therefore this approach cannot be used in order to show consistency for the estimator of $\tilde \Gamma$ and thus also not of $\tilde O$. 
Due to this complication \cite{BauWag2001} were led to provide an adapted estimator by setting the remaining block rows of the first
block column of $V_T$
equal to zero. 
In this paper a different route in the proof for consistency of the estimator for $b_r$ is provided showing that the adaptation is not needed. 

In addition to the changes in the consistency proofs also the derivation of 
the asymptotic distribution of the estimators $\check O$ and $\check \Gamma$ of $\tilde O$ and $\tilde \Gamma$ as provided in Lemma 13.2. of \cite{Johansen} for the case $c_y=0$
cannot be used in the case $c_y>0$ as can be seen from these arguments:  
In the last equation on p. 182 the last term $(\check O - \tilde O)\tilde \Gamma' \la \tilde z_t^{r,\pi} , \tilde z_t^{r,\pi} \ra$ is shown to tend to zero
using consistency for $\check O$ and stationarity of $\tilde \Gamma' \tilde z_t^{r,\pi}$ for $c_y=0$. For $c_y>0$, however, $\tilde \Gamma' \tilde z_t^{r,\pi}$ contains nonstationary 
components such that $\la \tilde \Gamma' \tilde z_t^{r,\pi},  \tilde z_t^{r,\pi} \ra = O_P(T)$ and moreover converges in distribution to a nondegenerate distribution when divided by $T$. 
Hence even if $(\check O - \tilde O)$ is estimated superconsistently such that $T(\check O - \tilde O)$ converges in distribution, the term 
$(\check O - \tilde O)\tilde \Gamma' \la \tilde z_t^{r,\pi} , \tilde z_t^{r,\pi} \ra$ in the last equation on p. 182 does not vanish. 
Thus also for the asymptotic distribution the proof in \cite{Johansen} does not apply for the case $c_y>0$ and a more detailed analysis is needed.
It is the main goal of the paper to close this gap in the literature.



In this paper two different estimators 
are considered: RRR estimator based on the unrestricted OLS estimator as well as based on   
the fully modified unrestricted estimator of \cite{Phillips95}.
The main contributions of the paper are: 

\begin{itemize}
	\item A full discussion of the asymptotic properties of the RRR estimators including 
conditions for 
consistency, derivation of  the asymptotic distribution of the estimators under the condition of known rank $n$ is provided. 
	\item For the RRR estimator based on OLS almost sure (a.s.) rates of convergence are provided, improving the results in the 
literature which provide only in probability convergence. 
	\item Furthermore in all cases the asymptotic distribution will be given explicitly and a detailed comparison of the relative advantages in
a number of  special cases is 
provided.  
\end{itemize}

The organization of this paper is the following: 
The next section presents the various estimation algorithms while their corresponding asymptotic properties are discussed 
in section~\ref{sec:results}. 
Section~\ref{sec:specialcases} 
illustrates the results using a number of special cases. Finally section~\ref{sec:concl} summarizes the paper. 
All results are proved in Appendix~\ref{app:A}. 
A summary of the notation is contained in Appendix~\ref{sec:notation}.

\section{Estimation Algorithms} \label{sec:algo}
In this paper four different estimators for the coefficient matrices $b_r, b_u$ in equation~\eqref{equ:RRR} based on 
observations for time instants $t=1,\dots,T$ are considered. 
Throughout as above the notation $\la a_t, b_t \ra := T^{-1}\sum_{t=1}^T a_t b_t'$ will be used (somewhat sloppily using $a_t, b_t$ for 
the processes $(a_t)_{t \in \mZ}$ and $(b_t)_{t \in \mZ}$ and for the variables $a_t, b_t$ for given time instant $t$ respectively). 
\\
Using this notation the ordinary least squares (OLS) estimator (that ignores the knowledge on the rank constraint $\mbox{rank}(b_r)=n$) 
can be written as
$$
\hat \beta_{OLS} = \la y_t , z_t \ra \la z_t, z_t \ra^{-1}, \quad \hat \beta_{OLS} = [\hat \beta_{OLS,r},\hat \beta_{OLS,u}].  
$$

If 
$$
\hWf  (\hWf)':= \la y_t^\pi,  y_t^\pi \ra^{-1}, \quad 
y_t^\pi = y_t - \la y_t, z_{t}^{u} \ra \la z_{t}^{u},z_{t}^{u} \ra^{-1} z_{t}^{u},
$$
the rank restricted estimator maximizing the quasi maximum likelihood based on the assumption of iid Gaussian residuals can be defined as 
$$
\hat \beta_{RRR} = \mbox{arg}\min_{\beta=[\beta_r,\beta_u]  \in \mR^{s \times (m_r+m_u)}, \mbox{rank}(\beta_r)=n} 
\mbox{tr}\left[ \hWf \sum_{t=1}^T (y_t - \beta_r z_t^r - \beta_u z_t^u) (y_t - \beta_r z_t^r - \beta_u z_t^u)' \hWf' \right]
$$
and is given by

\begin{equation} \label{def:RRRest}
\hOf = (\hWf)^{-1} \hat U_n \hat S_n, \hKp' = \hat V_n' (\hWp)^{-1}, \hat \beta_{RRR,r} = \hOf\hKp', 
\hat \beta_{RRR,u} = \la y_t - \hat \beta_{RRR,r} z_t^r , z_t^u \ra \la z_t^u, z_t^u\ra^{-1}
\end{equation}

using the SVD
$$
\hWf \hat \beta_{OLS,r} \hWp = \hat U_n \hat S_n \hat V_n' + \hat R_n
$$

where $\hat U_n$ denotes the matrix having as columns the singular vectors corresponding to the dominant
singular values $\hat \sigma_1\ge \hat \sigma_2 \ge \dots \ge \hat \sigma_n > 0$ contained as the diagonal in the diagonal matrix
$\hat S_n$. The corresponding right singular vectors are contained in $\hat V_n$. Finally $\hat R_n$ constitutes the 
approximation error. 
Here $\hWp = \la z_t^{\pi}, z_t^{\pi} \ra^{1/2}$ 
(where $X^{1/2}$ denotes the symmetric matrix square root of the square matrix $X$ and $z_t^{\pi}$ denote 
residuals from 
regression of $z_{t}^r$ onto $z_{t}^u$). 
Clearly the estimator $\hat \beta_{RRR,r}$ does not depend on the decomposition of $\hOf\hKp'$ into $\hOf$ and $\hKp'$.

Note that for this choice of $\hWf$ and $\hWp$ the columns of $\hKp'$ can also be interpreted as the eigenvectors to the generalized eigenvalue problem 
$$
\la z_{t}^{\pi} , z_{t}^{\pi} \ra \hKp \hat S_n^2 = \la  z_{t}^{\pi} ,  y_{t}^{\pi} \ra \la y_{t}^{\pi} ,  y_{t}^{\pi} \ra^{-1} \la  y_{t}^{\pi} , 
 z_{t}^{\pi} \ra \hKp .
$$

As can be verified straightforwardly the corresponding estimate $\hOf$ equals the coefficients for regressing $y_t^{\pi}$ onto $\hKp'  z_{t}^{\pi}$. 
Thus in the Johansen framework the Johansen estimators are obtained. 

\cite{Phillips95} discusses the fully-modified (FM) estimators as an alternative to least squares estimation. 
The fully modified OLS estimator (FM-OLS) of $\beta$
is defined as 
\begin{equation} \label{equ:defFMOLS}
\hat \beta_{OLS}^+ 
:=  \left(\la y_t, z_t\ra -  \hat \Delta_{\hat u,\Delta z} - \hat \Omega_{\tilde u,\Delta z} \hat \Omega_{\Delta z,\Delta z}^{-1} ( \la \Delta z_t, z_t \ra 
- \hat \Delta_{\Delta z,\Delta z} )
\right) \la z_t , z_t \ra^{-1}
\end{equation}
where  
for processes $(a_t)_{t \in \mZ}$ and $(b_t)_{t \in \mZ}$ the estimates
$$
\hat \Omega_{a,b} := \sum_{j=1-T}^{T-1} w(j/K)\hat \Gamma_{a,b}(j), \quad \hat \Delta_{a,b} := \sum_{j=0}^{T-1} w(j/K)
\hat \Gamma_{a,b} (j)
$$

are used. As usual $\Delta z_t := z_t - z_{t-1}$. Here $\hat \Gamma_{a,b} (j):=\la a_t, b_{t-j}\ra = T^{-1}\sum_{t=1}^{T} a_t b_{t-j}'$ 
denotes the estimated covariance sequence where 
observations outside of the observed sample are treated as zeros.
Further $\hat \Omega_{\tilde u,\Delta z}$ is estimated using the residuals 
$\hat u_t = y_t - \hat \beta_{OLS} z_t$. 
Throughout we will use the subscripts to indicate the processes involved. Additionally superscripts 
indicate components of the processes. 
A slight difference to the notation of e.g. \citeasnoun{Phillips95} is that for the integrated processes $z_t$, say,
we index by $\Delta z$ rather than only $z$.

Consequently for stationary processes $(a_t)_{t \in \mZ}$ and $(b_t)_{t \in \mZ}$ it follows that 
$\hat \Omega_{a,b}$ and $\hat \Delta_{a,b}$ are estimators of the long-run-covariance and the one-sided long run covariance matrices defined as
$$
\Omega_{a,b} = \sum_{j=-\infty}^{\infty} \mE a_j b_{0}', \quad \Delta_{a,b} = \sum_{j=0}^\infty \mE a_j b_{0}'.
$$

For the kernel function $w(\cdot)$ occurring in this definition 
we will use the standard assumptions 
\citeaffixed{Phillips95}{cf.}:
\\
\begin{assume}{K}
The kernel function $w(.): \mR \to [-1,1]$ is a twice continuously 
differentiable even function with
\begin{itemize}
	\item[(a)] $w(0)=1, w'(0)=0, w''(0)\ne 0$ 
	\item[(b)] $w(x)=0, |x| \ge 1$ with $\lim_{|x| \to 1} w(x)/(1-|x|)^2 =$ constant 
\end{itemize}
Further the bandwidth parameter $K$ in the kernel estimates
is chosen proportional to $c_TT^b$ for  some $b \in (1/4,2/3)$
where $c_T$ is slowly varying at infinity (i.e. $c_{Tx}/c_T \to 1, \forall x>0$).
\end{assume}

Analogously to the RRR estimator derived from the OLS estimator 
we derive  
the new fully modified RRR estimator (henceforth denoted as
{\tt FM-RRR}) from the FM-estimator using the SVD 
\begin{eqnarray}
& & \hWff \hat \beta_{OLS,r}^+
\la z_t , z_t \ra^{1/2}   =  
\hat U_n^+ \hat S_n^+ (\hat V_n^+)' + \hat R_n^+, \nonumber \\
\hWff & = & \left( \la y_t^\pi , y_t^\pi \ra - \hat \Omega_{\tilde u,\Delta z} \hat \Omega_{\Delta z,\Delta z}^{-1}
(\la \Delta z_t, y_t^\pi \ra - \hat \Delta_{\Delta z,\Delta y^\pi})-
(\la y_t^\pi , \Delta z_t \ra - \hat \Delta_{\Delta y^\pi,\Delta z})\hat \Omega_{\Delta z,\Delta z}^{-1}\hat \Omega_{\Delta z,\hat u} 
\right)^{-1/2}. \label{eq:defWff}
\end{eqnarray}
where as before $\hat U_n^+$ denotes the matrix
of left singular vectors, $\hat S_n^+=\mbox{diag}(\hat s_1^+,\hat s_2^+,\dots,\hat s_n^+)$ 
is the diagonal matrix containing the
dominant estimated singular values $\hat s_1^+ \ge \hat s_2^+ \ge \dots \ge \hat s_n^+ > 0$ 
decreasing in size and the columns of 
$\hat V_n^+$ contain the corresponding right singular vectors.

The estimator under the rank restriction $\mbox{rank}(\beta)=n$ 
then is defined as
\begin{equation} \label{equ:defFMRRR}
\hat \beta_{RRR,r}^+ = 
(\hWff)^{-1} \hat U_n^+ \hat S_n^+ (\hat V_n^+)'\la z_t, z_t \ra^{-1/2}, 
\hat \beta_{RRR,u}^+ = 
\hat \beta_{OLS,u}^+ - (\hat \beta_{OLS,r}^+ - \hat \beta_{RRR,r}^+) \la z_t^r, z_t^u \ra \la z_t^u, z_t^u \ra^{-1}.
\end{equation}

\section{Results} \label{sec:results}
In this paper the following assumptions on the data generating process (dgp) will be used:
\\
\begin{assume}{P}  \label{ass:procrrr}
The process $(y_t)_{t \in \mZ}$ is generated according to~\eqref{equ:RRR} with $u_t = \Lambda \ve_t$ ($\Lambda  \in \mR^{s \times k}$ of 
full row rank) where 
$(z_t^r)_{t \in \mZ}$ and $(z_t^u)_{t \in \mZ}$ are processes such that for some 
orthogonal matrices $H_r = [H_{r,\parallel},H_{r,\bot}], H_u = [H_{u,\parallel},H_{u,\bot}]$ 
($H_{r,\parallel} \in \mR^{m_r \times c_r}, H_{u,\parallel} \in \mR^{m_u \times c_u}$) we have
$$
\mbox{diag}(\Delta(L)I_{c_r},I_{m_r-c_r})H_{r}' z_{t}^r = v_t, \quad t \in \mZ, \quad
\mbox{diag}(\Delta(L)I_{c_u},I_{m_u-c_u})H_{u}' z_t^u = w_t, \quad t \in \mZ,
$$
where $\Delta (L) = 1-L$ denotes the difference operator ($L$ denoting the backward shift operator) 
and the joint vector $\nu_t := [v_t',w_t']'$
is a stationary process generated according to
$$
\nu_t  = \sum_{j=1}^\infty C_j \ve_{t-j} 
$$
where $\sum_{j=1}^\infty j^a \| C_j \| < \infty$ for some $a>3/2$ and where 
for the transfer function $c(z):= [c_v(z)',c_w(z)']'=\sum_{j=1}^\infty C_jz^j$ (with
$z$ denoting a complex variable) the matrix $c(1)$ is of full row rank. 
Additionally it is assumed that 
$$
\mE 
\left[ \begin{array}{c} H_{r,\bot}'z_t^r \\ H_{u,\bot}' z_t^u \end{array} \right] 
\left[ \begin{array}{c} H_{r,\bot}'z_t^r \\ H_{u,\bot}' z_t^u \end{array} \right]' >0.
$$
Here $(\ve_t)_{t \in \mZ}$ is an iid process with zero mean, nonsingular variance  $\Sigma$ and finite 
fourth moments. Finally $H_{r,\parallel}'z_0^r=0$ and $H_{u,\parallel}'z_0^u=0$. 
\end{assume}

Note that summation for $\nu_t$ starts at $j=1$. Thus uncorrelatedness of the regressors with the noise is built into the 
assumptions. 
The assumptions imply that $z_t^r$ and $z_t^u$ are I(1) processes such that the cointegrating rank of the joint process
equals the sum of the cointegrating ranks of the two processes. 

The assumption of zero initial conditions is not important and can be replaced with the assumption of deterministic initial conditions, i.e. assuming that modeling is performed 
conditional on initial conditions. 

The noise is assumed to constitute an iid sequence which is somewhat restrictive. 
Weaker assumptions are possible but make the asymptotic distributions more involved. 
Further note that the same noise $\ve_t$ is used to generate the regressors as well as the residuals in the estimation equation.
Consequently lagged $y_t$'s are admitted as regressors 
and some dynamics may be included in the model, alleviating the iid assumption. 

Furthermore these assumptions exclude deterministic terms such as the constant as regressors which are discussed separately below.

The assumptions on the data generating process lead to the following representation result:
\begin{theorem} \label{thm:repr}
Let Assumption P hold where $n=\mbox{rank}(b_r),
b = [b_r,b_u]$.  
\\
(I) Let 
$c_y\le n$ denote the
rank of $b_r H_{r,\parallel}$. Then the cointegrating rank of
$(z_t^r)_{t \in \mZ}$ is $m_r - c_r$ and the
cointegrating rank of $(y_t - b_u z_t^u)_{t \in \mZ}$ is
$s - c_y$. \\
(II)
There exist nonsingular
matrices ${\mathcal T}_y \in \mR^{s \times s}, {\mathcal T}_{z,r} \in \mR^{m_r \times m_r}$ and
${\mathcal T}_{z,u} \in \mR^{m_u \times m_u}$
such that
$$
\begin{array}{l}
\tilde y_t = \left[ \begin{array}{c} \tilde y_{t,1} \\ \tilde y_{t,2} \end{array}\right]
= {\mathcal T}_y (y_t - b_u z_t^u) =
\tilde b_r \tilde z_t + 
\tilde \ve_t =
\left[ \begin{array}{ccc} I_{c_y} & 0 & 0 \\ 0 & 0 & \tilde b_{2,3} \end{array} \right]
\left[\begin{array}{c} \tilde z_{t,1} \\ \tilde z_{t,2} \\ \tilde z_{t,3} \end{array}\right]
+ \left[\begin{array}{c} \tilde \ve_{t,1} \\ \tilde \ve_{t,2} \end{array}\right], \\
\tilde z_t = {\mathcal T}_{z,r} z_{t}^r = \left[ \begin{array}{c} \tilde z_{t,1} \\
\tilde z_{t,2} \\ \tilde z_{t,3} \end{array} \right],
\tilde z_t^u = {\mathcal T}_{z,u} z_t^u =
\left[\begin{array}{c} \tilde z_{t,1}^u \\ \tilde z_{t,2}^u \end{array} \right]
\end{array}
$$
where $\Delta(L) \tilde z_{t,1} = \tilde c_{z,1}(L)\ve_t, \Delta (L)\tilde z_{t,2} =
\tilde c_{z,2}(L)\ve_t,
\Delta(L) \tilde z_{t,1}^u = \tilde c_{z,u}(L)\ve_t, t \in \mZ$ ($L$ denoting the backward shift operator)
and the matrix $[\tilde c_{z,1}(1)',\tilde c_{z,2}(1)',\tilde c_{z,u}(1)']$
is of full column rank,
and $(\tilde z_{t,3})_{t \in \mZ}$
and $(\tilde z_{t,2}^u)_{t \in \mZ}$ are
stationary processes with nonsingular spectrum at $z=1$. 
\end{theorem}
The result is proved in Appendix~\ref{app:A}. 
It builds the main representation of the regression on which the asymptotic results are
based upon. Note that the matrices ${\mathcal T}_y, {\mathcal T}_{z,u}$ and ${\mathcal T}_{z,r}$ separating the non-stationary 
and stationary directions of the various processes 
are not unique and the theorem only ascertains the existence. The restrictions on the ranks of the various matrices ensures that
the various components are either stationary processes which are not over differenced or integrated processes which are not 
cointegrated. 

Under these assumptions it is well known that the OLS  estimators are weakly consistent \cite{ParkPhillips,ParkPhillips89}.
Furthermore almost sure consistency as well as the convergence rate $\hat \beta_{OLS}- b = O(\sqrt{\log \log T/T})$
(i.e. $\sqrt{T/\log \log T}(\hat \beta_{OLS} - b)$ is almost surely (a.s.) bounded) can be derived, see e.g. \cite{baunote}.  
Additionally their asymptotic distribution is also well documented:
Let ${\cal T}_z = \mbox{diag}({\mathcal T}_{z,r},{\mathcal T}_{z,u})$
and let $D_z = \mbox{diag}(D_{z,r},D_{z,u}), D_{z,r} = \mbox{diag}(T^{-1}I_{c_r},T^{-1/2}I_{m_r-c_r}),
D_{z,u} =\mbox{diag}(T^{-1}I_{c_u},T^{-1/2}I_{m_u-c_u})$.
Then one obtains
$$
{\cal T}_y(\hat \beta_{OLS}- b) {\cal T}_z^{-1}D_z^{-1} \stackrel{d}{\to} \left[\begin{array}{cccc} M_r & Z_r
& M_u - M_r N_r & Z_u - Z_r \mE  \tilde z_{t,3} (\tilde z_{t,2}^u)'
(\mE  \tilde z_{t,2}^u (\tilde z_{t,2}^u)')^{-1}
\end{array}\right]
$$
where (using the notation\footnote{Here and below $\int dE W'$ is the usual shorthand notation for $\int_{0}^1 dE(w) W(w)'$ for Brownian motions $E(w), W(w), w \in [0,1]$. Analogously $\int WW'$ is short for $\int_0^1 W(w)W(w)' dw$.} $f(E,W)= \int dE W' (\int WW')^{-1}$)  
\begin{eqnarray*}
M_r & = & f({\cal T}_y\Lambda W,W_z^\Pi),
\\
M_u & = & f({\cal T}_y\Lambda W,W_u), \\
N_r & = & \int W_zW_u' \left( \int W_uW_u'\right)^{-1}, \\
\mbox{vec}\left[T^{-1/2}\sum_{t=1}^T {\cal T}_y \Lambda \ve_t (\tilde z_{t,3}^\pi)'\la \tilde z_{t,3}^\pi, \tilde z_{t,3}^\pi\ra^{-1}\right] &  \stackrel{d}{\to} & \mbox{vec}(Z_r), \\
\mbox{vec}\left[T^{-1/2}\sum_{t=1}^T {\cal T}_y \Lambda \ve_t (\tilde z_{t,2}^u)'\la \tilde z_{t,2}^u, \tilde z_{t,2}^u \ra^{-1}\right]  & \stackrel{d}{\to}&  \mbox{vec}(Z_u),
\end{eqnarray*}
where $\tilde z_{t,3}^\pi := \tilde z_{t,3} - \la \tilde z_{t,3},
\tilde z_{t,2}^u \ra \la \tilde z_{t,2}^u, \tilde
z_{t,2}^u\ra^{-1} \tilde z_{t,2}^u$. $W$ denotes the Brownian
motion corresponding to $(\ve_t)_{t \in \mN}$ and $W_z = \tilde c_{z,1:2}(1)W, W_u = \tilde c_{z,u}(1)W,
W_z^\Pi = W_z  - \int
W_zW_{u}' (\int W_{u}W_{u}' )^{-1}W_{u}$.
Further $\mbox{vec}(Z_r)$ and
$\mbox{vec}(Z_u)$ are normally distributed
with mean zero (vec denotes columnwise vectorisation).
Finally $\tilde c_{z,1:2}(1):= [\tilde c_{z,1}(1)',\tilde c_{z,2}(1)']'$.

The next theorem, which is the main contribution of this paper, extends these results to the 
RRR estimators:

\begin{theorem} \label{thm:RRR}
(I) Let the assumptions of Theorem~\ref{thm:repr} hold. 
Then 
$\hat \beta_{RRR}- b = O((\lT)^6/\sqrt{T})$.
Furthermore let $\hat \beta_{RRR,r}$ and $\hat \beta_{OLS,r}$ denote the
coefficients corresponding to $z_{t}^r$. Then 
the asymptotic distribution of $\hat \beta_{RRR,r}$ can be found from

$$
[{\mathcal T}_y (\hat \beta_{RRR,r} - \hat \beta_{OLS,r} ){\mathcal T}_{z,r}^{-1}] D_{z,r}^{-1} \stackrel{d}{\to}
-\left[ \begin{array}{cc}
\left[ \begin{array}{c}
-\Xi  \\ I \end{array} \right]
(I-\tO_2\tO_2^\dagger)M_{r,2} \left[ \begin{array}{cc} -Y_{21}Y_{11}^{-1} &  I \end{array} \right]
&
\left[ \begin{array}{c}  \tilde R \\
(I-\tO_2\tO_2^\dagger) \tilde Z_{r,2} P \end{array} \right]  \end{array} \right]
$$
where
\begin{eqnarray*}
\Xi & = & - [I,0]{\mathcal T}_y \Lambda  \mE  \ve_{t}\tilde y_{t,2}'(\mE \tilde y_{t,2}^\Pi\tilde y_{t,2}')^{-1}, \\
M_{r,2}  & = &
f([0,I]{\mathcal T}_y \Lambda W,W_{z,2}^\Pi-Y_{21}Y_{11}^{-1}W_{z,1}^\Pi),\\
Y_{i1} & =& \int W_{z,i}^\Pi (W_{z,1}^\Pi)', i=1,2, \\ 
P & = & I - \mE \tilde z_{t,3}^\Pi(\tilde z_{t,3}^\Pi)' \Gamma_{3,2}\Gamma_{3,2}^\dagger,
\Gamma_{3,2}^\dagger = (\Gamma_{3,2}' \mE \tilde z_{t,3}^\Pi(\tilde z_{t,3}^\Pi)' \Gamma_{3,2})^{-1}\Gamma_{3,2}',\\
\tilde b_{2,3}  & = & \tO_2\Gamma_{3,2}', 
\tO_2^\dagger  =  (\tO_2'(\mE \tilde y_{t,2}^\Pi\tilde y_{t,2}')^{-1}\tO_2)^{-1}\tO_2'(\mE \tilde y_{t,2}^\Pi\tilde y_{t,2}')^{-1}.
\end{eqnarray*}
Here $\tilde Z_{r,2} =[0,I] Z_r$,
$W_z^\Pi = [(W_{z,1}^\Pi)',(W_{z,2}^\Pi)']'$
and
$\tilde z_{t,3}^\Pi = \tilde z_{t,3} - \mE \tilde z_{t,3} (\tilde z_{t,2}^u)' (\mE  \tilde z_{t,2}^u  (\tilde z_{t,2}^u)' )^{-1} \tilde z_{t,2}^u$
and $\tilde y_{t,2}^\Pi$ is defined analogously.
Finally $\tilde R$ is defined in Lemma~\ref{lem:asydist}.
Correspondingly letting $\tilde \beta_{RRR,u}$ and $\tilde \beta_{OLS,u}$ denote the coefficients
corresponding to $z_{t}^u$ then
$$
{\mathcal T}_y(\hat \beta_{RRR,u}-\hat \beta_{OLS,u}){\mathcal T}_{z,u}^{-1}D_{z,u}^{-1} = -{\mathcal T}_y(\hat \beta_{RRR,z}-\hat \beta_{OLS,z}){\mathcal T}_{z,r}^{-1}
D_{z,u}^{-1} \left[\begin{array}{cc} N_r & 0 \\ 0 & \mE  \tilde z_{t,3} (\tilde z_{t,2}^u)'
(\mE  \tilde z_{t,2}^u (\tilde z_{t,2}^u)')^{-1} \end{array} \right]+ o_P(1).
$$
(II) All results hold true in the situation that  
all observations are demeaned or detrended prior to estimation, if a.s. rates are replaced with in probability rates, the Brownian motions are replaced by 
their corresponding demeaned or detrended version  and if additionally to the assumptions above the 
condition $\sum_{j=1}^\infty j^a \| C_j \| < \infty$ holds for some $a>3$.
\end{theorem}

Note that the decomposition of $\tilde b_{2,3}$ is not specified. The asymptotic distribution does not depend on the actual choice. 

The theorem shows how the inclusion of the rank constraint
affects the estimation error which is given as a sum of the error
for the unrestricted estimate plus a correction term. All coefficients
corresponding to the nonstationary directions in $z_t$ are estimated $T$-consistent and
asymptotically the estimation errors have 'matrix unit root' distributions,
whereas for directions in which $(z_t)_{t \in \mN}$ is 
stationary the coefficients are only
$\sqrt{T}$ consistent and the errors are asymptotically normal. 
The proof of this theorem is given in  Appendix~\ref{app:A}.
Note that the larger bounds in the almost sure convergence rates for the restricted estimator reflects only the techniques of proof used and not
the accuracy of the estimators which is more appropriately represented in the distributional results. I.e. the larger bounds for the 
rank restricted estimator mirrors our inability to prove the tighter bounds rather than the relative accuracy of the estimators.

In the fully modified case conditions for consistency and the asymptotic distribution
of the unrestricted estimator 
is provided in \cite{Phillips95}:
Under the assumptions on the kernel provided in Assumptions~K 
one obtains:
$$
{\mathcal T_y}(\hat \beta_{OLS}^+ -  b) {\mathcal T_z}^{-1} D_z^{-1} \stackrel{d}{\to}
\left[\begin{array}{cccc} M_r^+ & Z_r
& M_u^+ - M_r^+ N_r & Z_u - Z_r \mE  \tilde z_{t,3} (\tilde z_{t,2}^u)'
(\mE  \tilde z_{t,2}^u (\tilde z_{t,2}^u)')^{-1}
\end{array}\right]
$$
where
\begin{eqnarray*}
B & = & \Lambda W - \Omega_{u,\Delta z}^{:,n} (\Omega_{\Delta z,\Delta z}^{n,n})^{-1}
\left[\begin{array}{c} \tilde c_{z,1:2}(1)W \\ \tilde c_{z,u}(1)W \end{array} \right],\\
M_r^+ & = & f(B,W_{z}^\Pi), \\
M_u^+ & = & f(B,\tilde c_{z,u}(1)W).
\end{eqnarray*}
Here the superscript $n$ refers to the nonstationary directions in $[\tilde z_t',(\tilde z_t^u)']'$
and the matrices $ \Omega_{u,\Delta z}^{:,n}$ and $\Omega_{\Delta z,\Delta z}^{n,n}$ are composed of the respective columns and rows corresponding to the
nonstationary components.
 
The next theorem discusses the properties of the corresponding
rank restricted estimator:
\begin{theorem} \label{thm:FM}
Let the assumptions of Theorem~\ref{thm:RRR} hold and additionally
assume that a kernel function fulfilling assumptions~K is used in the nonparametric estimation
of the long run variances.
Let $\tilde \beta_{RRR}^+$ denote the FM-RRR estimator (based on the fully modified estimator $\tilde \beta_{OLS}^+$)
defined in~\eqref{equ:defFMRRR}
for the weight $\hWf$ as defined in~\eqref{eq:defWff}. Then using the notation of Theorem~\ref{thm:RRR}
it holds that
$$
[{\mathcal T}_y (\hat \beta_{RRR,r}^+ - \hat \beta_{OLS,r}^+ ){\cal T}_{z,r}^{-1}] D_{z,r}^{-1} \stackrel{d}{\to}
-\left[ \begin{array}{cc}
\left[ \begin{array}{c}
\Xi  \\ I \end{array} \right]
(I-\tilde O_2\tilde O_2^\dagger)M_2^+ \left[ \begin{array}{cc} -Y_{21}Y_{11}^{-1} &  I \end{array} \right]
&
\left[ \begin{array}{c} \tilde R \\
(I-\tilde O_2\tilde O_2^\dagger)\tilde Z_{r,2}P \end{array} \right]  \end{array} \right].
$$
where $M_2^+ =
f([0,I]{\mathcal T}_y \Lambda B,W_{z,2}^\Pi-Y_{21}Y_{11}^{-1}W_{z,1}^\Pi)$.
\end{theorem}
Therefore the relation between the restricted and the unrestricted regressions
are identical for the conventional and the fully modified case. Also the
expressions for the two sets of estimators are identical except for the use
of $W$ in the conventional case which is replaced by $B$ in the fully modified
case.
Therefore it follows that the distribution in the direction of
(asymptotically) stationary components of $w_t$ is identical for both
estimators. Hence in the case that $c_z=0$ and therefore no integration
is present the conventional and the fully modified estimators
have the same asymptotic distribution. This is true for the restricted and the unrestricted
estimates. We refrain from a more complete discussion on the properties of the
fully modified estimator since for the unrestricted case these are
well documented in the literature.
Instead a number of special cases will be discussed below. 

\section{Special Cases} \label{sec:specialcases}
First consider the case where all included variables are stationary.
In that case $\T_y=I, \T_z = I$ can be used and the asymptotic distribution of the vectorizations of 
$\sqrt{T}(\hat \beta_{OLS} -  b)$ and 
$\sqrt{T}(\hat \beta_{OLS}^+ - b)$ are both normal with mean zero and variance $(\mE z_tz_t')^{-1} \otimes \Lambda \Sigma \Lambda'$
which equals the distribution of 

$$
\mbox{vec}\left(\left[Z_r,\quad Z_u - Z_r \mE \tilde z_{t,3} (\tilde z_{t,2}^u)'(\mE \tilde z_{t,2}^u (\tilde z_{t,2}^u)')^{-1} \right] \right)
$$ 

noting that in this case $\tilde z_t = \tilde z_{t,3} = z_t^r, \tilde z_{t,2}^u = z_t^u$ can be chosen. 
The correction due to the rank restriction for $\beta_r$ equals the vectorization of 
$$
-(I-\tilde O_2\tilde O_2^\dagger)Z_r(I- \mE \tilde z_{t,3} 
(\tilde z_{t,3})' \Gamma_{32}\Gamma_{32}^\dagger).
$$ 

The corresponding correction to $b_u$ follows. On total hence one obtains as the asymptotic distribution
of the RRR-estimator for $b_r$ the distribution of 
$
Z_r -(I-\tilde O_2\tilde O_2^\dagger)Z_r(I- \mE \tilde z_{t,3} 
(\tilde z_{t,3})' \Gamma_{32}\Gamma_{32}^\dagger). 
$

This asymptotic distribution (for a generic case) has previously been documented in \cite{Reinsel} on p. 45 (2.36) albeit in a different form which is less accessible.
On p. 46 a more explicit expression for the case $n=1$ is given. 
It is straightforward to see that the expressions in this special case are identical while the formula provided above also provide insights in the general case.
It must be noted, however, that these expressions are not new and have been used already e.g. in \cite{BauDeiSch96}.   

The consequences of the correction using the rank restriction are the following: 
Premultiplying the asymptotic distribution with $\tilde O_2^\dagger$ one notices that the rank restriction does not influence the distribution in these directions. 
In the orthogonal complement, however, the distribution is changed from $x' Z_r$ to $x' Z_r \mE \tilde z_{t,3} 
(\tilde z_{t,3})' \Gamma_{32}\Gamma_{32}^\dagger$ and hence projected onto the rows corresponding to the space spanned by the columns of $\Gamma_{32}$. 
The analogous statements hold for the postmultiplication with $\Gamma_{32}$. 
Note that these arguments also hold in the general case for the $(2,3)$ block of $\tilde b_r$.

As a second special case consider the VAR(1) I(1) model of \citeasnoun{Anderson02}. For simplicity of notation the transformed 
system will be used which in the notation of \citeasnoun{Anderson02} is stated as 
\begin{equation} \label{eq:dXX}
\Delta X_t = \Upsilon X_{t-1} + W_t = \left[ \begin{array}{cc} 0_{c_y \times c_y} & 0 \\ 0 & \Upsilon_{22} \end{array} \right] X_{t-1} + W_t, t \in \mN
\end{equation}
for $X_0=0$ where $\Upsilon_{22}$ is nonsingular. This defines an I(1) process $X_t \in \mR^s$ whose first component, $X_{t,1}\in \mR^{c_y}$ say, 
is integrated, the remaining, $X_{t,2} \in \mR^{s-c_y}$ say, being stationary for 
$|\lambda_{max}(I+\Upsilon_{22})|<1$ which is assumed in the following. The variance of the iid white noise $W_t$ is taken to 
be $[\Sigma_{ij}^W]_{i,j=1,2}$ which is partitioned according to the partitioning of $X_t$. 
In our notation no transformation matrices are needed since the system is already in the appropriate coordinate system. 
Hence $y_t = \tilde y_{t,2} = \Delta X_t$ is stationary, $\tilde z_{t,2} = X_{t-1,1},\tilde z_{t,3} = X_{t-1,2}, z_{t}^u$ does not 
occur. Consequently $\tilde b_r = \Upsilon = [\Upsilon_{:, 1},\Upsilon_{:, 2}]$ and $b_u$ does not occur. 

In this situation \cite{Anderson02} gives the asymptotic distribution of the unrestricted and the restricted estimates 
of $\Upsilon$. Consider first $\Upsilon_{:, 1}$, i.e. the first block column. Then Theorem 1 of \cite{Anderson02} 
states that $T\tilde \Upsilon_{:, 1,OLS} \stackrel{d}{\to} J_{:, 1} I_{11}^{-1}$ which in our notation equals $f(W,W_1)$ where 
$W$ is the Brownian motion according to $u_t = W_t$ and $W_1$ denotes the corresponding first block.  
For $\tilde \beta_{RRR,r}$ \cite{Anderson02} states the asymptotic distribution of the first block column 
as 
$$
T \tilde \Upsilon_{:, 1,RRR} \stackrel{d}{\to} \left[ \begin{array}{c} 0 \\ J_{2.1}I_{11}^{-1} \end{array} \right], J_{2.1} = [-\Sigma_{WW}^{21}(\Sigma_{WW}^{11})^{-1},I]J_{:, 1}.
$$
From Theorem~\ref{thm:RRR} we obtain 
$$
T\tilde \beta_{:, 1,RRR} = T\tilde \beta_{:, 1,OLS} + T(\tilde \beta_{:, 1,RRR}-\tilde \beta_{:, 1,OLS}) 
\stackrel{d}{\to} f(W,W_1) -(I-\tilde O_2\tilde O_2^\dagger)f(W,W_1) = \tilde O_2\tilde O_2^\dagger J_{:, 1}I_{11}^{-1}.
$$
The second block column of $\tilde b_r$ provides the decomposition $\tilde O_2 := [0,I]',\Gamma_{32}' := \Upsilon_{22}$.
Here $\tO_2\tO_2^\dagger = [0,I]'([0,I] \YY^{-1} [0,I]')^{-1} [0,I]\YY^{-1}$ where 
\begin{eqnarray*}
\YY^{-1} & = & 
\left[ \begin{array}{cc} \Sigma_{WW}^{11} & \Sigma_{WW}^{12} \\ \Sigma_{WW}^{21} & Q \end{array}\right]^{-1}
\\ 
& = &\left[ \begin{array}{cc} (\Sigma_{WW}^{11})^{-1} & 0 \\ 0 & 0 \end{array}\right]
+ \left[ \begin{array}{c} -(\Sigma_{WW}^{11})^{-1} \Sigma_{WW}^{12} \\ I \end{array} \right] 
\left( Q- \Sigma_{WW}^{21} (\Sigma_{WW}^{11})^{-1} \Sigma_{WW}^{12}\right)^{-1}
[-\Sigma_{WW}^{21}(\Sigma_{WW}^{11})^{-1},I]
\end{eqnarray*}
(for some matrix $Q$) 
according to the block matrix inversion. Thus $\tO_2\tO_2^\dagger = [0,I]'[-\Sigma_{WW}^{21}(\Sigma_{WW}^{11})^{-1},I]$ showing the 
identity of the expressions. 

For the FM estimator note that the involved long run covariances equal
$$
\Omega_{u,\Delta z}^{:,n} =  \Sigma_{WW}^{:, 1},\quad \Omega_{\Delta z, \Delta z}^{n,n} = \Sigma_{WW}^{11}
$$
with all other terms being zero. Correspondingly 
$$
\Omega_{u,\Delta z}^{:,n}(\Omega_{\Delta z, \Delta z}^{n,n})^{-1} = [I,(\Sigma_{WW}^{11})^{-1}\Sigma_{WW}^{12}]' \Rightarrow 
J_{2.1}I_{11}^{-1} = f(B,W). 
$$

Thus it follows that the coefficients to the nonstationary regressors for the FM-estimator have the same asymptotic distribution
as the RRR estimators. Adding the rank restriction in this case does not change the {\em asymptotic distribution} while it might well influence the finite sample properties. 
It is straightforward to show that in this case also the RRR-FM estimator has the same distribution for the columns corresponding to the integrated regressors. 

With respect to the stationary directions it is easy to see that $P = 0$ since $\Gamma_{32}=\Upsilon_{22}$ is invertible. Consequently the 
RRR estimator and the OLS estimator have the same asymptotic distribution in the columns corresponding to the 
stationary regressors. Since FM and OLS estimators have the same asymptotic behavior for stationary regressors all four 
estimators show the same asymptotic behavior in these columns. The underlying reason for this is that the rank restriction 
exclusively applies to the nonstationary restrictions where the corresponding coefficient is restricted to zero. For the stationary regressors there are no other rank restrictions in this case.  

Adding additional lagged first differences to \eqref{eq:dXX} the AR(p) setting with transformed coordinates is obtained. 
The additional coefficients are not restricted (except for the seldom imposed restriction of stability of the 
corresponding transfer function) and hence in this case $z_{t}^u = \tilde z_{t,2}^u = [\Delta X_{t-1}',\dots,\Delta 
X_{t-p+1}']'$, i.e. additional stationary regressors are present. It is well known that in this case (using the 
usual notation such that $\Upsilon = \alpha \beta'$ where $\alpha_\perp' \alpha = 0, \beta_\perp' \beta = 0$ for orthogonal matrices $\alpha_\perp, \beta_\perp$ of maximal dimension such that the columns span the 
orthogonal complement of $\alpha, \beta$ respectively) we have
\begin{eqnarray}
X_t & = &
\beta_\perp(\Gamma_J)^{-1}\left(\alpha_\perp'\sum_{j=1}^{t-1}\ve_{t-j}
+ X_{1}\right) + w_t,
\end{eqnarray}
for some stationary process $w_t$ and nonsingular matrix $\Gamma_J$ (expressions could be given but are not of importance in the following and hence omitted). 
In the example $\alpha' = [0,\Upsilon_{22}'], \beta' = [0,I]$ and thus $\alpha_\perp' = [I,0],
\beta_\perp = [I,0]'$.  

The changes in the asymptotic distribution are the following: $W$ is unchanged while $W_z = \Gamma_J^{-1}W_1$. 
The stationary components change accordingly. 
Imposing the rank restriction (as is done in the Johansen quasi-ML estimators)
does not change the asymptotic distribution of the coefficients corresponding to the stationary terms as in the AR(1) case
presented above since $\Gamma_{32}$ again is nonsingular and hence $P=0$. 
Thus we obtain the same asymptotic distribution as in the non restricted case. This asymptotic distribution is also given in Theorem 13.5. of \cite{Johansen}.

For the coefficients corresponding to nonstationary coordinates we obtain analogously to above
$$
T\tilde \beta_{:, 1,RRR} = T\tilde \beta_{:, 1,OLS} + T(\tilde \beta_{:, 1,RRR}-\tilde \beta_{:, 1,OLS}) 
\stackrel{d}{\to} \tO_2\tO_2^\dagger f(W,W_z) 
= 
\left[ \begin{array}{c} 0 \\ I\end{array} \right] [-\Sigma_{WW}^{21}(\Sigma_{WW}^{11})^{-1},I] f(W,W_z).
$$

For the FM estimator note that the involved long run covariances equal
$$
\Omega_{u,\Delta z}^{:,n} =  \Sigma_{WW}^{:, 1}(\Gamma_J^{-1})',\Omega_{\Delta z, \Delta z}^{n,n} = \Gamma_J^{-1}\Sigma_{WW}^{11}(\Gamma_J^{-1})'
$$
due to the change in the nonstationary directions. 
Then $\Omega_{u,\Delta z}^{:,n}(\Omega_{\Delta z, \Delta z}^{n,n})^{-1} = [I,(\Sigma_{WW}^{11})^{-1}\Sigma_{WW}^{12}]'(\Gamma_J^{-1})'$
as above 
implies that again the unrestricted FM estimator has the same asymptotic distribution as the RRR estimator. 
This is remarkable since the FM estimator does not require the specification of the rank restriction. 
This has already been observed in \cite{Phillips95} but apparently did not draw the attention of the community. 

\section{Conclusions}  \label{sec:concl}
In this paper the asymptotic properties for 
two estimators in a regression setting explicitly imposing a rank restriction are discussed.
Beside providing 
(almost sure) rates of convergence also explicit expressions for the asymptotic distribution of transformed estimators (such 
that stationary and nonstationary coordinates are separated) are provided. These expressions reveal the main characteristics 
of the estimators and allow insights into the relative merits of the various methods such as the gain in asymptotic accuracy 
obtained by imposing the rank restriction. 
In particular it is shown that the
fully modified estimators in many situations achieve the same asymptotic distribution as the rank restricted regression OLS 
estimators without imposing the rank restriction. This is an attractive feature in situations where the rank is not known. 

The  results contain a number of well known situations as special cases and even in some of these cases allow new insights as the 
previously published expressions for the asymptotic distribution are much more complicated to interpret. 

Finally it must be noted that the results in this paper are seen to be intermediate results that might in many cases not seem to be 
relevant as they relate to transformed estimators where the transformations are not known during the estimation. Nevertheless, the results
are important ingredients to explore the properties of procedures that use the RRR as an intermediate step. 
An important example are subspace methods in the case
of cointegrated processes. These results will be presented elsewhere.

\bibliographystyle{kluwer}


\appendix 

\section{Proofs} \label{app:A}

Throughout the appendix the following notation will be
heavily used:
For a sequence of random matrices $F_T$ with elements $F_{i,j,T}$ and
a sequence of scalars $g_T$ we will use
the notation $F_T = o(g_T)$ if $\lim \sup_{T\to \infty} \max_{i,j} | F_{i,j,T}/g_T | \to 0$ almost surely
(a.s.). Similarly $F_T = O(g_T)$ if there exists a constant $M$ such that
$\lim \sup_{T\to \infty} \max_{i,j} | F_{i,j,T}/g_T | \le M$ a.s.
The corresponding in probability versions are:
$F_T = o_P(g_T)$ if $\max_{i,j} | F_{i,j,T}/g_T | \to 0$ in probability
and $F_T = O_P(g_T)$ if for each $\ve>0$ there exists a constant $M(\ve)<\infty$ such that
$ \lim_{T\to \infty}\mP \{ \max_{i,j} | F_{i,j,T}/g_T | > M(\ve) \} < \ve$.
In all these statements $T$ denotes the sample size.
Therefore in particular convergence in distribution to a finite dimensional
 almost surely finite
random variable implies the rate $O_P(1)$. 
Throughout convergence in probability will be denoted as
$\stackrel{p}{\to}$ and
convergence in distribution as $\stackrel{d}{\to}$. Almost sure (a.s.) convergence is denoted as $\to$.
 $\|. \|$ denotes the Euclidean norm if not
stated explicitly otherwise. $\| . \|_{Fr}$ is used to denote the Frobenius norm.
As usual the integral
$\int W_1 W_2'$ is short notation for $\int_0^1 W_1(\omega)W_2(\omega)'d\omega$
and $\int dW_1 W_2'$ is short for $\int_0^1 dW_1(\omega)W_2(\omega)'$.
Here $W_1(\omega)$ and $W_2(\omega)$ are two Brownian motions on $[0,1]$.

\subsection{Preliminary lemmas}

\begin{lemma} \label{lem:bound}
(I) Let $(\ve_t)_{t \in \mZ}$ denote
a white noise sequence which fulfills the noise assumptions contained in
Assumption P.
Define $x_{t,1} := \sum_{j=1}^{t-1}\ve_j, t\ge2, x_1=0, v_t :=
c_v(L)\ve_{t}:= \sum_{i=1}^\infty C_{v,i}\ve_{t-i}, t\in \mN,$ for
some transfer function $c_v(z):=\sum_{j=1}^\infty C_{v,j}z^j$ where
$\sum_{j=1}^\infty \| C_{v,j} \| j^{a} < \infty$ for some $a>3/2$.
Furthermore $n_t :=\sum_{j=0}^{t-1} C_{n,j}x_{t-j,1}, t \in \mN,$ for a sequence 
$C_{n,j}$ such that $\sum_{j=0}^\infty \| C_{n,j} \| j^{a}< \infty, a>3/2$
and for $c_n(z)=\sum_{j=0}^\infty C_{n,j}z^j$ it holds that $\det c_n(1) \ne 0$.
Finally let
$Q_T := \sqrt{\log \log T/T}$.
Then
\begin{eqnarray*}
\| \la v_{t} , \ve_t \ra \| = O(Q_T) &,
& \| \la v_t , v_{t} \ra - \mE v_t v_{t}' \|=O(Q_T), \\
\la x_{t,1}, x_{t,1} \ra = O(T\log \log T) &, & \la x_{t,1}, \ve_t \ra = O(\log T), \\
\| \la x_{t,1}, v_{t} \ra \| = O(\log T) &, & \la x_{t,1} , x_{t,1} \ra^{-1} = O(Q_T^2).
\end{eqnarray*}
All expressions remain true if $x_{t,1}$ is replaced by $n_t$.

(II) Furthermore using  
$\Delta_{v,\Delta n} = \sum_{j=0}^\infty \mE v_j(\Delta n_0)'$ where 
$\Delta n_t = c_n(L)\ve_t, t \in \mZ$
we have:
\begin{eqnarray*}
\la v_t , n_t \ra  & \stackrel{d}{\to} & \int c_v(1) dW W'c_n(1)' + \Delta_{v,\Delta n}, \\
T^{-1}\la n_t, n_t \ra &  \stackrel{d}{\to} & c_n(1) \int WW' c_n(1)', \\
\mbox{vec}(T^{1/2} \la \ve_t, v_t \ra) & \stackrel{d}{\to} & {\cal N}(0,\mE v_t v_t' \otimes \mE \ve_t \ve_t')
\end{eqnarray*}
where $\mbox{vec}$ denotes column wise vectorization, $W(w)$ denotes the limiting
Brownian motion corresponding to $T^{-1/2}\sum_{j=1}^{\lfloor w T \rfloor} \ve_j$.
Finally ${\cal N}(0,V)$ denotes a Gaussian random variable with mean zero and variance $V$. 

(III) Let $x_t := [x_{t,\bullet}',x_{t,1}']'$ where $(x_{t,\bullet})_{t \in \mN}$ fulfills the same
restrictions as $(v_t)_{t \in \mN}$ under (I) and $(x_{t,1})_{t \in \mN}$ and $(n_t)_{t \in \mN}$
are integrated and of the same form as $(n_t)_{t \in \mN}$ under (I).
Further let $(v_{t})_{t \in \mN}$ and $(w_{t})_{t \in \mN}$ be two stationary processes
fulfilling the
assumption of $(v_t)_{t \in \mN}$ under (I). Let $(\ve_t)_{t \in \mZ}$ be as under (I). 
Let $^\pi$ denote the residuals of a regression onto $x_t$  and let $^\Pi$ denote the
corresponding limits (whenever the symbol is used the limit exists). 
Hence e.g. $v_t^\pi = v_t - \la v_t , x_t \ra \la x_t , x_t \ra^{-1} x_t$.
Then
\begin{eqnarray*}
\la \ve_t , v_t^\pi \ra  & = &  \la \ve_t , v_t^\Pi \ra + o(T^{-1/2}) = O(Q_T), \\ 
\la \ve_t , n_t^\pi \ra & = &  \la \ve_t , n_t \ra - \la \ve_t , x_{t,1} \ra \la x_{t,1}, x_{t,1} \ra^{-1} \la x_{t,1}, n_t \ra + o(1) = O(\log T(\log \log T)^2), \\
\la v_t^\pi , w_t^\pi \ra  & = &  \la v_t^\Pi , w_t^\Pi \ra + O(Q_T)= O(1)  ,  \\
\la v_t^\pi , n_t^\pi \ra  & = &  O(\log T(\log \log T)^2), \\
\la n_t^\pi , n_t^\pi \ra & = & \la n_t , n_t \ra   -   \la n_t , x_{t,1} \ra \la x_{t,1}, x_{t,1} \ra^{-1} \la x_{t,1}, n_t \ra + o(T) =O(T(\log \log T)).
\end{eqnarray*}

(IV) Let $d_t := [1,t]'$ and $D_d := \mbox{diag}(1,T^{-1})$.
For any process $(a_t)_{t \in \mN}$ let $\bar a_t$ denote the
detrended process $\bar a_t := a_t - \la a_t, d_t \ra \la d_t , d_t \ra^{-1}d_t$.
Let $v_t = \sum_{j=0}^\infty C_{v,j} \ve_{t-j}$ where $\sum_{j=0}^\infty j^2 \| C_{v,j} \| < \infty$.
Then $\| \la \bar v_t, \bar v_{t} \ra - \mE v_t v_{t}' \| = O_P(Q_T)$.
Further for $(x_t)_{t \in \mN}$ as in (I) it follows that $\la \bar \ve_t , \bar x_t \ra = O_P(1),
\la \bar v_{t},  \bar x_t \ra = O_P(1)$.
The same holds for replacing $x_t$ with $n_t := \sum_{j=0}^{t-1} C_{n,j} x_{t-j}$ if $\sum_{j=0}^\infty \| C_{n,j} \| j^3 < \infty$.

The following limit theorems hold:
\begin{eqnarray*}
\la \bar \ve_t, \bar x_t \ra  \stackrel{d}{\to}  \int d\bar W \bar W' & , &
T^{-1}\la \bar x_t , \bar x_t \ra  \stackrel{d}{\to}   \int \bar W \bar W'
\end{eqnarray*}
where $\bar W := W -  (\int W(s)ds)(4-6\omega) -  (\int sW(s) ds)(12\omega - 6)$
denotes the demeaned
and detrended Brownian motion associated with $(\ve_t)_{t \in \mN}$.\\
The analogous results also holds for
the demeaned series $\bar a_t := a_t - \la a_t, 1\ra$
where $\bar W := W -  (\int W(s)ds)$ appears in the asymptotic distributions.
\end{lemma}

\begin{proof} (I) $\la v_{t}, \ve_t \ra =O(Q_T)$ and $\la v_{t}, v_{t}\ra - \mE v_tv_{t}'=O(Q_T)$
follow from Theorem 7.4.3. of~\citeasnoun{HanDei}.
$\la x_t , x_t\ra = O(T\log \log T)$ follows from Theorem 3 of \citeasnoun{LaiWei83},
$\la x_t, \ve_t \ra = o(\log T)$ from Corollary 2 of ~\citeasnoun{LaiWei}.
Both results only deal with the univariate case but the extension to the multivariate
situation is obvious.
This result also implies $\la x_t, v_{t} \ra = O(\log T)$ (see \cite{baunote}, Lemma 4).
The same result applies for $n_t$ in place of $x_t$ by splitting
$n_t = c_n(1)x_t + n_t^*$ (Beveridge-Nelson decomposition, see e.g. \citename{PhillipsSolo}, \citeyear*{PhillipsSolo})
where
\begin{eqnarray*}
n_t^*= n_t - c_n(1)x_t
 & = & \sum_{j=0}^{t-1}C_{n,j}x_{t-j} - \sum_{j=0}^\infty C_{n,j} x_t
= (C_{n,0} - c_n(1))x_t + \sum_{j=1}^{t-1} C_{n,j} x_{t-j}\\
& = & (C_{n,0} - c_n(1))\ve_t + (C_{n,0} - c_n(1))x_{t-1} +
\sum_{j=1}^{t-1}C_{n,j}x_{t-j}  \\
& = & (C_{n,0} - c_n(1))\ve_t + (C_{n,0} + C_{n,1} - c_n(1))(\ve_{t-1} +x_{t-2}) + \sum_{j=2}^{t-1}C_{n,j}x_{t-j} \\
& = & \sum_{j=0}^{t-1} C_{n,j}^* \ve_{t-j}
\end{eqnarray*}
where $C_{n,i}^* := -\sum_{j=i+1}^{\infty} C_{n,j}$.
Due to the summability assumptions on $C_{n,j}$ the transfer function
$c_n^*(z):=\sum_{i=0}^\infty C_{n,i}^* z^i$ fulfills the properties
of Theorem~7.4.3. of \citeasnoun{HanDei}. The result then follows from the assumed non-singularity of 
$c_n(1)$.

The univariate version
of $\la x_t ,x_t \ra^{-1} = O(Q_T^2)$ is contained in ~\citeasnoun[p. 163]{LaiWei}.
The multivariate version is showed in \citeasnoun{baunote}. 

(II) Since $\sum_{t=1}^{\lfloor wT \rfloor} \ve_t /\sqrt{T} \Rightarrow W(w)$
 \cite[Theorem 27.17]{Davidson} the convergence of
$\la v_t, n_t \ra$ is e.g. given in ~\citeasnoun[Lemma 2.1. (e)]{ParkPhillips}.
The result for $T^{-1}\la n_t , n_t \ra$ is stated in part (c) of the same lemma.
The central limit theorem is standard \citeaffixed[Lemma 4.3.4.]{HanDei}{cf. e.g.}
since $(\ve_t)_{t \in \mN}$ is an ergodic square integrable
martingale difference sequence.

(III) The proof  is based on the block matrix inversion formula
\begin{equation} \label{equ:BMI}
\left[ \begin{array}{cc} A & B \\ C & D \end{array} \right]^{-1} =
\left[ \begin{array}{cc} A^{-1} & 0 \\ 0 & 0 \end{array} \right]
+ \left[ \begin{array}{c} - A^{-1} B \\ I \end{array} \right]
\left( D - CA^{-1}B \right)^{-1}
\left[ \begin{array}{cc} -CA^{-1}, I \end{array} \right]
\end{equation}
applied to $\la x_t, x_t \ra$. As an example consider
$$
\begin{array}{l}
\la \ve_t , v_t^\pi \ra = \la \ve_t, v_t \ra
- \la \ve_t , x_t \ra \la x_t, x_t \ra^{-1}\la x_t, v_t \ra \\
= \la \ve_t, v_t \ra
- \la \ve_t , x_{t,\bullet} \ra \la x_{t,\bullet}, x_{t,\bullet} \ra^{-1}\la x_{t,\bullet}, v_t \ra
- \la \ve_t , x_{t, \neg 1} \ra \la x_{t, \neg 1}, x_{t, \neg 1} \ra^{-1}\la x_{t, \neg 1}, v_t \ra
\end{array}
$$
where $x_{t, \neg 1} := x_{t,1} - \la x_{t,1} , x_{t,\bullet} \ra \la x_{t,\bullet}, x_{t,\bullet} \ra^{-1}x_{t,\bullet}$.
Therefore $\la x_{t, \neg 1} , x_{t, \neg 1}  \ra = \la x_{t,1}, x_{t,1} \ra - O(\log T)O(1)O(\log T)$
and $\la x_{t, \neg 1}, v_t \ra = \la x_{t,1} , v_t \ra - O(\log T) = O(\log T)$ by (I).
Thus
$$
\begin{array}{ccc}
\la \ve_t , v_t^\pi \ra & = & \la \ve_t, v_t \ra
- \la \ve_t , x_{t,\bullet} \ra \la x_{t,\bullet}, x_{t,\bullet} \ra^{-1}
\la x_{t,\bullet}, v_t \ra + o((\log T)^3/T) \\
& = & \la \ve_t, v_t \ra
- \la \ve_t , x_{t,\bullet} \ra
(\mE x_{t,\bullet} x_{t,\bullet}')^{-1} \mE x_{t,\bullet} v_t' + o(T^{-1/2})
\end{array}
$$
since $\la \ve_t , x_{t, \bullet} \ra = O(Q_T)$
and $(\mE x_{t,\bullet} x_{t,\bullet}')^{-1} \mE x_{t,\bullet} v_t'-\la x_{t,\bullet}, x_{t,\bullet} \ra^{-1}
\la x_{t,\bullet}, v_t \ra = O(Q_T)$
as required.
For $\la v_t^\pi, w_t^\pi \ra = \la v_t, w_t^\pi \ra$ the same arguments apply
with the exception that now $\la v_t, w_t \ra =O(1)$ rather than $O(Q_T)$ leading to the
second claim.
The other claims follow in a similar manner from the bounds achieved under (I). 

(IV)
Only the results for detrending are shown, the analogous statements
for the demeaned series are obvious from the given results.
The derivations here use Lemma 1, p. 121 of \citeasnoun{SimsStockWatson}.
Lemma 1 (g) shows that
$\la v_t, d_t \ra D_d = O_P(T^{-1/2})$ for stationary $(v_t)_{t \in \mN}$
and $\la x_t, d_t \ra D_d = O_P(T^{1/2})$ for integrated $(x_t)_{t \in \mN}$.
\\
For $\la \bar \ve_t , \bar x_t \ra = \la \ve_t, x_t \ra - \la \ve_t, d_t
\ra \la d_t, d_t \ra^{-1} \la d_t, x_t \ra$
note that the first term converges in distribution according to (II). For the second term
note that $\sqrt{T}\la \ve_t, d_t \ra D_d \stackrel{d}{\to} [ \int dW, \int \omega dW ],
D_d \la d_t, d_t \ra D_d$ converges to a constant nonsingular matrix and
$T^{-1/2}D_d \la d_t, x_t \ra \stackrel{d}{\to}
[ \int W, \int \omega W ]'$. The last two statements follow
from Lemma 1 (a) and (c) of \citeasnoun{SimsStockWatson}.
Therefore
$$
\la \bar \ve_t , \bar x_t \ra \stackrel{d}{\to} \int dW W' - \left[ \int dW, \int \omega dW \right]
\left[ \begin{array}{cc} \int 1 & \int \omega \\ \int \omega & \int \omega^2 \end{array} \right]^{-1}
\left[ \begin{array}{c} \int W' \\ \int \omega W' \end{array} \right].
$$
This shows that the Brownian motion in the limiting expression is demeaned and detrended.

If $\ve_t$ is replaced by $v_t$
convergence in distribution still holds, but the limits change.

The evaluations for $T^{-1}\la \bar x_t, \bar x_t \ra$ follow the same lines and are omitted.

Decomposing $n_t = c_n(1)x_t + n_t^*$ as above shows that in the above calculations $x_t$ can
be replaced with $n_t$ without changing the orders of convergence. 

Finally if the time trend is omitted and only demeaning is performed the results can be shown analogously using the arguments given above. 
\end{proof}

\begin{lemma} \label{lem:omegadelta}
Under the assumptions of  Theorem~\ref{thm:FM} the following holds true:
\begin{eqnarray*}
\hat \Omega_{\tilde u,\Delta z} \hat \Omega_{\Delta z, \Delta z}^{-1} & = &
\left[ \begin{array}{cc} \Omega_{u, \Delta z}^{:,n} (\Omega_{\Delta z, \Delta z}^{n,n})^{-1} + o_P(1) & O_P(1) \end{array} \right],\\
\la \Delta \check z_t, \check z_t \ra - \hat \Delta_{\Delta z, \Delta z} & = &
\left[ \begin{array}{cc}
\int dB [ \left[ \begin{array}{c} \tilde c_{z,1:2}(1) \\ \tilde c_{z,u}(1) \end{array} \right]      W]'+ o_P(1) & O_P(K^{-2}) + O_P(1/\sqrt{KT}) \\
T^{-1} \check z_{T,2}\check z_{T,1}' + O_P(K^{-2})+ O_P(1/\sqrt{TK}) & O_P(K^{-2}) \end{array} \right],\\
\hat \Delta_{\tilde u, \Delta z} & = &
\left[ \begin{array}{cc} \Delta_{u,\Delta z}^{:,n} + O_P((K/T)^{1/2}) & O_P(1/\sqrt{KT}) \end{array} \right],\\
(\la \Delta \check z_t, \tilde u_t \ra - \tilde \Delta_{\Delta z, \Delta u}) & = &
\left[ \begin{array}{c} O_P(K^{-2}) + O_P(1/\sqrt{KT}) \\ O_P(K^{-2}) \end{array}\right]
\end{eqnarray*}
where $B= $ 

Here the notation refers to the transformed vectors $\check z_t = [(\tilde z_{t,1})',(\tilde z_{t,2})',(\tilde z_{t,1}^u)',(\tilde z_{t,3})',(\tilde z_{t,2}^u)']'$ 
where the nonstationary components of both vectors (first block row) and the stationary components (second block row)
of these vectors are separated. $K$ denotes the kernel bandwidth parameter (see Assumptions K).
Furthermore $\Delta_{w,\Delta v} = \mE w_tv_t'$ for stationary processes $(v_t)_{t \in \mZ}$ and $(w_t)_{t \in \mZ}$.
\end{lemma}

The proof of all but the last of these facts can be found in \citeasnoun[Lemma 8.1]{Phillips95}.
The last fact can be easily derived from the infinite sum representation of $\Delta_{z,\Delta v}$.

\begin{lemma} \label{lem:dKp}
Let $b_{r} = \Of \Kp'$ where $\Kp'S_p = I_n$.
Here $S_p  \in \mR^{m_r \times n}$ denotes a selector matrix
(i.e. a matrix composed of $n$ columns of $I_{m_r}$).
Let $\tilde \beta_{r}$ denote an estimator of
$b_{r}$ such that $\| \tilde \beta_{r} - b_{r}\|_{Fr} = o(a_T)$.
\\
Assume that $\| \hWf - \Wf \|_{Fr} = o(T^{-\epsilon})$ and $\| \hWp - \Wp \|_{Fr} = o(T^{-\epsilon})$ ($\Wf$ and $\Wp$ being nonsingular) for some $\epsilon>0$ and let 
$\hWf \hOf \hKp \hWp$ be obtained as the best (in Frobenius norm) rank $n$ approximation of $\hWf \tilde \beta_r \hWp$.   
Further let $\Of^\dagger := (\Of' (\Wf)^2 \Of)^{-1} \Of'(\Wf)^2$ assuming that
$ \| (\Of'(\Wf)^2  \Of)^{-1} \| < \infty$.

Then for $T$ large enough $\hKp$ can a.s. be chosen such that $\hKp'S_p = I_n$.
Further
\begin{equation} \label{eq:DKP}
\hKp' - \Kp' = \Of^\dagger (\tilde \beta_{r} - b_{r}) (I_p - S_p\Kp') + o(a_T).
\end{equation}
\end{lemma}
\begin{proof}
Since $\| \tilde \beta_{r} - b_{r}\|_{Fr} = o(a_T)$ it follows from the
boundedness assumption on $\Of^\dagger$ that $\Of^\dagger \tilde \beta_{r}S_p \to \Of^\dagger \beta_{r}S_p=I_n$.
Since $\hOf \hat \Kp'$ is a best approximation to $b_{r}$ based
on $\tilde \beta_{r}$ in a weighted least squares sense it follows
that
$$
\| \hOf \hKp' - b_{r} \|_{Fr} \le
\| \hOf \hKp' - \tilde \beta_{r} \|_{Fr} + \| \tilde \beta_{r} - b_{r} \|_{Fr}
= O(\| \tilde \beta_{r} - b_{r} \|_{Fr})
=o(1).
$$

It follows that $\Of^\dagger \hOf \hKp' S_p \to \Of^\dagger b_{r}S_p =I_n$.
Therefore $\hKp$ subject to
the restriction $\hKp'S_p=I_n$ is well defined a.s. for $T$ large enough.
It follows that
$\| (\hOf\hKp' - b_{r})S_p \|_{Fr} =\| \hOf -\Of \|_{Fr} = o(1)$.
Letting $\hOf^\dagger := (\hOf' (\hWf)^2 \hOf)^{-1} \hOf' (\hWf)^2$
one obtains
$$
\hKp' - \Kp' = \hOf^\dagger \tilde \beta_{r} - \Of^\dagger b_{r} =
(\hOf^\dagger- \Of^\dagger) \tilde \beta_{r} + \Of^\dagger( \tilde \beta_{r} - b_{r}).
$$

Then $\| \hOf - \Of\|  = o(1)$ and $\| \hWf -  \Wf \| =o(T^{-\epsilon})$ together with the
bounds on the norms
$\| (\Of' (\Wf)^2 \Of)^{-1}\|, \| \Wf \|$ and $\| \Of \|_{Fr}$
imply $\| \Of^\dagger - \hOf^\dagger\|_{Fr} = o(1)$.
Using the fact that $(\hKp'-G')S_p=0$
and $b_{r}(I-S_pG')=0$
one obtains
$$
\hKp' - \Kp' = \Of^\dagger (\tilde \beta_{r} - b_{r})(I-S_pG') +
(\hOf^\dagger - \Of^\dagger) (\tilde \beta_{r} - b_{r})(I-S_pG')
$$
which proves the lemma.
\end{proof}
\begin{lemma} \label{lem:block}
Let $A_T=A_T' = A_0 + \delta^A, A_0 = A_0'$ and $B_T = B_0 + \delta^B$ be two sequences of matrices
$A_T \in \mR^{a \times a}$ and $B_T \in \mR^{a \times b}$.
$A_0$ and $B_0$ are possibly random matrices.
Assume that all matrices are partitioned as
\begin{eqnarray*}
A_T & = & \left[ \begin{array}{cc} A_{T,11} & A_{T,12} \\ A_{T,21} & A_{T,22}  \end{array} \right]=
\left[ \begin{array}{cc} A_{0,11}+ \delta_{11}^A & \delta_{12}^A \\ \delta_{21}^A & A_{0,22} + \delta_{22}^A  \end{array} \right]=
\left[ \begin{array}{cc} A_{0,11}+ O_P(a_T^2) & O_P(a_T) \\ O_P(a_T) & A_{0,22} + O_P(a_T)  \end{array} \right],\\
B_T & = & \left[ \begin{array}{cc} B_{T,11} & B_{T,12} \\ B_{T,21} & B_{T,22}  \end{array} \right]=
\left[ \begin{array}{cc} B_{0,11}+ \delta_{11}^B & \delta_{12}^B \\ \delta_{21}^B & B_{0,22} + \delta_{22}^B  \end{array} \right]=
\left[ \begin{array}{cc} B_{0,11}+O_P(b_T^2) & O_P(b_T) \\ O_P(b_T) & B_{0,22}+ O_P(b_T)  \end{array} \right],
\end{eqnarray*}

such that $A_{T,11} \in \mR^{c \times c}, B_{T,11} \in \mR^{c \times c}$ and all
other matrices have the corresponding dimensions. 
The subscripts for all matrices indicate
the corresponding blocks.
Assume that $A_{0,11}^{-1} = O_P(1), A_{0,22}^{-1} = O_P(1)$.
Finally let $J_T := B_T'(A_T)^{-1} B_T - B_0' (A_0)^{-1} B_0$. 

Then if $a_T$ and $b_T$ are such that $a_T \to 0, b_T \to 0$ we have
\begin{eqnarray*}
J_{T,11} & = & (\delta_{11}^B)'A_{0,11}^{-1} B_{0,11} + B_{0,11}'A_{0,11}^{-1} \delta_{11}^B - B_{0,11}'A_{0,11}^{-1}\delta_{11}^A A_{0,11}^{-1}B_{0,11}\\
 & & + \left[ (\delta_{21}^B)' - B_{0,11}'A_{0,11}^{-1}\delta_{12}^A \right]
 A_{0,22}^{-1} \left[ \delta_{21}^B - \delta_{21}^A A_{0,11}^{-1}B_{0,11}\right] +  o_P(a_T^2+b_T^2) \\
J_{T,21} & = & (\delta_{12}^B)'A_{0,11}^{-1} B_{0,11}
+ \left[ B_{0,22} + \delta_{22}^B \right]' A_{0,22}^{-1}\left[ \delta_{21}^B - \delta_{21}^A A_{0,11}^{-1}B_{0,11}\right] \\
& & - B_{0,22}'A_{0,22}^{-1}\delta_{22}^A A_{0,22}^{-1} \left[ \delta_{21}^B - \delta_{21}^A A_{0,11}^{-1}B_{0,11}\right]+ o_P(a_T^2+ b_T^2), \\
J_{T,22} & = & (\delta_{22}^B - \delta_{21}^A A_{0,11}^{-1} \delta_{12}^B)'\left( A_{0,22}^{-1} - A_{0,22}^{-1}(\delta_{22}^A - \delta_{21}^A A_{0,11}^{-1}\delta_{12}^A - \delta_{22}^A A_{0,22}^{-1} \delta_{22}^A)A_{0,22}^{-1}\right)B_{0,22} \\
& & + B_{0,22}'\left( A_{0,22}^{-1} - A_{0,22}^{-1}(\delta_{22}^A - \delta_{21}^A A_{0,11}^{-1}\delta_{12}^A - \delta_{22}^A A_{0,22}^{-1} \delta_{22}^A)A_{0,22}^{-1}\right)(\delta_{22}^B -  \delta_{21}^A A_{0,11}^{-1} \delta_{12}^B)\\
& & - B_{0,22}'\left( A_{0,22}^{-1}(\delta_{22}^A - \delta_{21}^A A_{0,11}^{-1}\delta_{12}^A - \delta_{22}^A A_{0,22}^{-1} \delta_{22}^A)A_{0,22}^{-1}\right) B_{0,22} \\
& &  + (\delta_{22}^B - \delta_{21}^A A_{0,11}^{-1} \delta_{12}^B)'A_{0,22}^{-1} (\delta_{22}^B -  \delta_{21}^A A_{0,11}^{-1} \delta_{12}^B)
+ (\delta_{12}^B)'A_{0,11}^{-1} \delta_{12}^B
+ o_P(a_T^2 + b_T^2).
\end{eqnarray*}
Therefore $J_{T,11} = O_P(a_T^2+b_T^2)$ and $J_{T,12}=O_P(a_T+b_T), J_{T,22} = O_P(a_T+b_T)$. All evaluations
hold if all in probability statements are exchanged by almost sure convergence.
\end{lemma}
\begin{proof}
The proof follows from straightforward algebraic manipulations using the block matrix inversion
\begin{equation}
A_T^{-1} = \left[ \begin{array}{cc} A_{T,11}^{-1} & 0 \\ 0 & 0 \end{array} \right]
+ \left[ \begin{array}{c} - A_{T,11}^{-1} A_{T,12} \\ I \end{array} \right]
\left( A_{T,22} - A_{T,21}A_{T,11}^{-1}A_{T,12} \right)^{-1}
\left[ \begin{array}{cc} -A_{T,21}A_{T,11}^{-1}, I \end{array} \right]
\end{equation}
noting that
$$
A_{T,11}^{-1} = (A_{0,11} + \delta_{11}^A)^{-1} = A_{0,11}^{-1} - A_{0,11}^{-1}\delta_{11}^A A_{0,11}^{-1} + o_P(a_T^2)
$$
since $a_T \to 0$ and $A_{0,11}^{-1} =O_P(1)$ by assumption.
Similarly
$$
\left( A_{T,22} - A_{T,21}A_{T,11}^{-1}A_{T,12} \right)^{-1}
= A_{0,22}^{-1} - A_{0,22}^{-1}\left( \delta_{22}^A  - \delta_{21}^A A_{0,11}^{-1} \delta_{12}^A - \delta_{22}^A A_{0,22}^{-1} \delta_{22}^A \right) A_{0,22}^{-1} + o_P(a_T^2)
$$
follows. The remaining calculations are tedious but straightforward and hence omitted.
\end{proof}

\begin{lemma} \label{lem:linsvd}
Define the two generalized eigenvalue problems: 
$$
(a) \quad \bar Q \bar G = \bar M \bar G \bar R^2, \quad 
(b) \quad \bar \Phi \bar \Gamma = \bar \Psi \bar \Gamma \bar \Theta^2. 
$$
where $\bar G \in \mR^{m_r \times n}, \bar \Gamma \in \mR^{m_r \times n}, \bar R^2 \in \mR^{n \times n}, \bar \Theta \in \mR^{n \times n}$.
Further $\bar \Psi$ and $\bar \Theta$ are assumed to be nonsingular a.s. and $\bar \Theta$ is diagonal.

(I) If $J:= \bar Q - \bar \Phi = O(a_T)$ and $\delta_{zz} := \bar M - \bar \Psi = O(b_T)$ (where $a_T \to 0, b_T \to 0$ for $T 
\to \infty$)   
then there exists matrices $\bar G$ and $\bar R$ solving the eigenvalue problem (a)
and matrices $\bar \Gamma$ and $\bar \Theta$ solving (b)
such that
$\bar \Gamma' S_p = I_n$ (where $S_p$ denotes a selector matrix, i.e. a matrix consisting of columns of the identity matrix), 
$\bar G - \bar \Gamma = O(a_T+b_T), \bar R- \bar \Theta =O(a_T+b_T)$.

(II) Further let $\delta G := \bar G - \bar \Gamma$.
Then the following two equations hold ($\bar \Gamma^\dagger := (\bar \Gamma' \bar \Psi \bar \Gamma)^{-1} \bar \Gamma'$):
\begin{eqnarray}
\bar \Phi\delta G - \bar \Psi \delta G \bar R^2 & = & \delta_{zz} \bar{G}\bar R^2 +
\bar \Psi\bar \Gamma(\bar R^2 - \bar{\Theta}^2)-J\bar{G} \label{equ:dg}, \\
(I_m - \bar \Psi\bar \Gamma\bar \Gamma^\dagger)\bar \Psi \delta G \bar R^2  & = &
(I_m - \bar \Psi\bar \Gamma\bar \Gamma^\dagger)\left[ J\bar{G} - \delta_{zz}\bar{G}\bar R^2 + \bar \Phi \delta G \right]
\label{eq:imgg}
\end{eqnarray}

(III) By transforming $\check G = \bar G(S_p' \bar G)^{-1}$ it follows that $\check G$ solves the generalized eigenvalue problem (a) with
matrix $\check R^2 = (S_p' \bar G) \bar R^2 (S_p' \bar G)^{-1}$. Then $\check G - \bar \Gamma = O(a_T+b_T), \check R- \bar \Theta =O(a_T+b_T)$. 
Here $\check R$ is not necessarily block diagonal.  
\end{lemma}

\begin{proof}
Solutions to the generalized eigenvalue problem are not identified. If all eigenvalues are
distinct then fixing the sign of one nonzero entry in each column of $\bar \Gamma$
results in a unique solution \citeaffixed[p. 1246, for a discussion]{BauDeiSch96}{see e.g.}.
If there are repeated eigenvalues then
more restrictions need to be introduced in order to achieve identification.
It follows from operator theory \citeaffixed{Chatelin}{cf. e.g.}
that there exist normalizations such that
the solution to the eigenvalue problem depends analytically on the
matrix which is decomposed, i.e. such that $G - \Gamma =o(1)$ a.s.
In these normalizations $R^2$ is not
necessarily diagonal while still being block diagonal
where the blocks correspond to the identical
eigenvalues in $\Theta$.

In particular let
the sequence of  matrices ${\cal M}_T \to {\cal M}_0$.
Let $\tilde \varphi_i$ denote the
matrix whose columns span the eigenspaces of ${\cal M}_T$ corresponding to the
eigenvalues $\tilde \lambda_{j} \to \lambda_{0,i}, j=1,\dots,m_i$
where $m_i$ denotes the multiplicity
of the eigenvalue $\lambda_{0,i}$ of ${\cal M}_0$ with corresponding
eigenspace spanned by the columns of the matrix $\varphi_{0,i}$.
Here it is assumed that the normalization $\tilde \varphi_i ' \varphi_{0,i} = I_{m_i} = \varphi_{0,i}'\varphi_{0,i}$
is chosen. Then
it holds that
\begin{equation} \label{equ:linsvd}
\tilde \varphi_{i} - \varphi_{0,i} = (\lambda_{0,i}I - {\cal M}_0)^\dagger ({\cal M}_T - {\cal M}_0)\varphi_{0,i}
+ O(\| {\cal M}_T- {\cal M}_0 \|^2 ).
\end{equation}

Here $X^\dagger$ denotes the Moore-Penrose pseudo-inverse. In particular let ${\cal M}_T = \bar M^{-1} \bar Q$ and ${\cal M}_0 = \bar \Psi^{-1}\bar \Phi$ and the columns of $G$ and $\Gamma$ equal $\tilde \varphi_{i}$ and $\varphi_{0,i}$ 
respectively. The condition of nonsingularity for $\bar \Theta$ ensures 
separation from the kernel of ${\cal M}_0$ and hence correct specification of the size of $\Gamma$.  
Then let $\bar \Gamma := \Gamma(S_p'\Gamma)^{-1}, \bar G := G(S_p'\Gamma)^{-1}$. 
Clearly $\bar \Gamma$ is a solution to the problem (b) fulfilling the assumption of the Lemma.\\
The assumptions imply that ${\cal M}_T - {\cal M}_0 = O(a_T+b_T)$ showing $G - \Gamma = O(a_T+b_T)$ and consequently $\bar G - \bar \Gamma = O(a_T+b_T)$. 
The order of convergence
for $\bar R - \bar \Theta$ then follows from the fact that all other terms in (a) and (b) differ only by this order.

(II) 
Equation \eqref{equ:dg} follows from simple algebraic manipulations using the
definitional equations (a) and (b) and $\bar Q = \bar \Phi + J, \bar M = \bar \Psi + \delta_{zz}, \bar G = \bar \Gamma + \delta G$.
Premultiplying \eqref{equ:dg} with $\bar \Gamma^\dagger$ 
a rearranging of terms leads to
\begin{equation}\label{eq:singval}
\bar R^2 - \bar{\Theta}^2 =   \bar \Gamma^\dagger[\bar \Phi\delta G - \bar \Psi \delta G \bar R^2-\delta_{zz} \bar{G}\bar R^2 +J\bar{G}].
\end{equation}

Inserting this into \eqref{equ:dg} shows that
$$
(I_m - \bar \Psi\bar \Gamma\bar \Gamma^\dagger)\left[ \bar \Phi\delta G - \bar \Psi \delta G \bar R^2 -  \delta_{zz} \bar{G}\bar R^2 + J\bar{G} \right] = 0.
$$
This shows equation~\eqref{eq:imgg}.
\\
(III) Follows immediately from (I).
\end{proof}

\subsection{Proof of Theorem~\ref{thm:repr}}
(I) Note that $\mbox{diag}(\Delta(L) I_{c_r},I_{m_r-c_r})H_{r}' z_{t}^r = c_v(L)\ve_t$.
Consequently the dimension of the cointegrating space of $(z_t^r)_{t \in \mN}$ 
is equal to $m_r - c_r$.
The claim on the dimension of the cointegrating space for
$(y_t - b_u z_{t}^u)_{t \in \mZ}$ also follows immediately from this representation. 

(II)
Note that $\tilde y_t$ denotes a transformation of
$y_t - b_u z_t^u = b_r z_t^r  + \Lambda\ve_t = b_r H_r H_r' z_t^r  + \Lambda\ve_t$ which
equals the estimation equation with the effects of $z_{t}^u$ removed. Here we use that $H_r$ was defined to be orthogonal. 
Next it is proved that matrices ${\mathcal T}_y$ and ${\mathcal T}_{z,r}$ transforming the equation into the
required form exist.

Let  ${\mathcal T}_y \in \mR^{s \times s}$ and nonsingular $C \in \mR^{c_r \times c_r}$ be chosen such that
$$
{\mathcal T}_y b_r H_{r,\parallel} C  =\left[ \begin{array}{cc} I_{c_y} & 0 \\ 0 & 0  \end{array} \right] \in \mR^{s \times c_r}.
$$
It is easy to see that such choices always exist since $c_y$ denotes the rank of $b_r H_{r,\parallel}$.
Then in $\tilde y_t := {\mathcal T}_y (y_t - b_u z_t^u)= {\mathcal T}_y (b_r H_r H_r' z_t^r + \Lambda \ve_t)$ the first $c_y$ coordinates are integrated,
the remaining being stationary.
Choosing $\bar{\mathcal T}_{z,r} = \mbox{diag}(C^{-1},I_{m_r-c_r})H_r'$ 
we obtain that the first $c_r$ components of $\bar{\mathcal T}_{z,r}z_t^r$
are integrated, the remaining
ones being stationary. Using the above equation we obtain
$$
{\mathcal T}_y b_r \bar{{\mathcal T}}_{z,r}^{-1} = \left[ \begin{array}{ccc} I_{c_y} & 0 & \tilde b_{r,13} \\ 0 & 0 & \tilde b_{r,23} \end{array} \right].
$$
Then the choice
$$
{\mathcal T}_{z,r} =  \left[ \begin{array}{ccc} I & 0 & \tilde b_{r,13} \\ 0 & I & 0 \\ 0 & 0 & I \end{array} \right] \bar{{\mathcal T}}_{z,r}
$$
leads to the required representation. The remaining claims are straightforward to derive. Details are omitted.

\subsection{Proof of Theorem~\ref{thm:RRR}}
\subsubsection{Consistency} 

Note that
all estimators can be obtained in a two step procedure by first
concentrating out $z_{t}^u$ and afterwards maximizing the quasi likelihood
with respect to $\beta_r$ (Frisch-Waugh-Lovell equations).
For fixed estimate $\hat \beta_r$ the least squares estimator for
$b_u$ is given by $\hat \beta_u = \la y_t - \hat \beta_r z_{t}^r, z_t^u \ra \la z_t^u , z_t^u \ra^{-1}$.
Therefore
\begin{equation} \label{eq:defbetau}
\hat \beta_u - b_u = \la b_r z_{t}^r + b_u z_t^u + \Lambda \ve_t - \hat \beta_r z_t^r - b_u z_t^u, z_t^u \ra 
\la z_t^u , z_t^u \ra^{-1}
= \la \Lambda\ve_t, z_t^u \ra \la z_t^u, z_t^u\ra^{-1} + (b_r - \hat \beta_r) \la z_t^r , z_t^u \ra \la z_t^u, z_t^u \ra^{-1}.
\end{equation}

This formula applies for the restricted and the unrestricted estimator.
For the first term note that 
$$
\la \ve_t, z_t^u \ra \la z_t^u, z_t^u\ra^{-1} = \la \ve_t, {\mathcal T}_{z,u}z_t^u \ra \la {\mathcal T}_{z,u} z_t^u, {\mathcal T}_{z,u}z_t^u\ra^{-1} {\mathcal T}_{z,u}
= \la \ve_t, \tilde z_t^u \ra \la \tilde z_t^u, \tilde z_t^u\ra^{-1} {\mathcal T}_{z,u}.
$$

Now \cite{baunote} implies that 
$$
\la \ve_t, \tilde z_t^u \ra \la \tilde z_t^u, \tilde z_t^u\ra^{-1} = [O(P_T),O(Q_T)]
$$
where $Q_T = \sqrt{\log \log T/T}$ and $P_T = \sqrt{\log T \log \log T/T^2}$.

In order to simplify the notation we use the symbols $\tilde y_t^\pi := {\mathcal T}_y(y_t - \la y_t, z_t^u \ra \la z_t^u, z_t^u \ra^{-1}z_t^u)$
and $\tilde z_t^\pi := {\mathcal T}_{z,r}(z_t^r - \la z_t^r, z_t^u \ra \la z_t^u, z_t^u \ra^{-1}z_t^u)$
throughout the proof.
Here the superscript $^r$ corresponding to $z_{t}^r$ will be omitted for
notational simplicity. The corresponding symbols $\tilde y_{t,i}^\Pi$ and $\tilde z_{t,i}^\Pi$
denote the corresponding
limit (a.s.) for $T \to \infty$ (where the symbols are only used if the limit exists).
In general the residuals of the regression of any variable
onto $z_t^u,t =1,\dots,T$ will be denoted using the superscript $^\pi$ and $^\Pi$ will denote
the corresponding limit (where it exists).

Using the same result and the Frisch-Waugh-Lovell equations for the first term and  
the orders of convergence stated in Lemma~\ref{lem:bound}
\begin{eqnarray*}
\la \tilde \ve_t, \tilde z_t^\pi \ra \la \tilde z_t^\pi, \tilde z_t^\pi\ra^{-1} & = & [O(P_T),O(Q_T)], \\
\la  \tilde z_t , \tilde z_t^u \ra \la \tilde z_t^u, \tilde z_t^u \ra^{-1} & = & 
\la  \tilde z_t , \tilde z_t^u \ra \mbox{diag}(T^{-1}I,I)(\la \tilde z_t^u, \tilde z_t^u \ra \mbox{diag}(T^{-1}I,I))^{-1} 
\\ &  = & \left[ \begin{array}{cc} O((\log\log T)^2) & O(\log T(\log \log T)^2) \\
O(1) & O(1) \end{array} \right].
\end{eqnarray*}

Therefore it is sufficient to show that $(b_r - \hat \beta_r)$ converges to zero. To this end a transformed problem such that all transformed matrices converge to 
nonrandom matrices is analyzed first. In this setting it will be possible to provide a.s. bounds for convergence rates. Afterwards the solution to the transformed problem 
is related to the solutions of the original problem. 

Thus consider the transformed problem using the transformation matrices 
$\check D_y = \mbox{diag}(\la \tilde z_{t,1}, \tilde z_{t,1} \ra^{-1/2},I)$  and 
$\check D_r = \mbox{diag}(\la \tilde z_{t,1}, \tilde z_{t,1} \ra^{-1/2},I)$ respectively to transform the input and output 
of the estimated regression according to $\check y_t = \check D_y \tilde y_t,  \check z_t = \check D_r \tilde z_t$. 
The transformed estimator $\check \beta_r := \check D_y  \tilde \beta_{OLS,r} \check D_z^{-1}$ converges to $\tilde b_{r} = {\mathcal T}_y b_r {\mathcal T}_z^{-1}=OG'$ where the last equation
defines $O$ and $G$.
Adapting the weighting $\check \Xi_+ := \hWf \check D_y^{-1}$ 
we obtain $\Wf = \mbox{diag}(I_{c_y}, (\mE \tilde y_{t,2}\tilde y_{t,2}')^{-1/2})$ and $\check \Xi_+ - \Wf = O((\lT) \llT /\sqrt{T})$ 
as needed in Lemma~\ref{lem:dKp}. This is obtained using the Cholesky factor as the square root of a matrix which is a differentiable operation.
With this new normalization we obtain\footnote{Here and below we will not always use the tightest possible bounds but 
use powers of $\log(T)$ instead for readability. Improvements are possible but their practical merits must be doubted.} 
\begin{eqnarray*}
\delta \check \beta_r  & := &  \check \beta_{OLS,r} - \tilde b_r = \check D_y ( \tilde \beta_{OLS,r} - \tilde b_r) \check D_z^{-1} = \check D_y [O(P_T), O(Q_T) ] \check D_z^{-1}  \\
 & = & 
\left[ \begin{array}{cc}   O(\lT (\llT)^3 /T) & O((\llT)^{3/2}/T) \\ O((\lT)^{3/2}/\sqrt{T}) & O(Q_T) \end{array} \right], \\
\check \beta_{RRR,r} - \tilde b_r & = & \check O \check G' - OG' = (\check O-O)\check G' + O(\check G'-G'), \\  
\check O - O & = & (\check \beta_{RRR,r} - \tilde b_{r})S_p, \\ 
\check G' - G' & = & \Of^\dagger (\check \beta_{OLS,r} - \tilde b_{r})(I-S_pG') +
(\bOf^\dagger - \Of^\dagger) (\check \beta_{OLS,r} - \tilde b_{r})(I-S_pG')
\end{eqnarray*}
where
$$
S_p = \left[ \begin{array}{cc} I & 0 \\ 0 & 0 \\ 0 & I \\ 0 & 0 \end{array}\right], 
G = \left[ \begin{array}{cc} I & 0 \\ 0 & 0 \\ 0 & \Gamma_{3,2} \end{array} \right],
I-S_pG' = I - \left[ \begin{array}{ccc} I & 0 & 0  \\ 0 & 0 & 0 \\ 0 & 0 & \Gamma_{32}' \\ 0 & 0 & 0 \end{array} \right].
$$

Here the first two block rows of $S_p$ correspond to $\tilde z_{t,1}$ and $\tilde z_{t,2}$. The remaining two blocks 
correspond to $\tilde z_{t,3}$. Since $\check \beta_{RRR,r}$ is a best rank $n$ approximation to $\check \beta_r$ (see the proof of Lemma~\ref{lem:dKp}) the 
rate of convergence of $\check \beta_r$ implies that $\check \beta_{RRR,r} - \tilde b_r = O((\lT)^{3/2}/\sqrt{T})$.
Consequently also   
$\check O-O=O((\lT)^{3/2}/\sqrt{T})$ and hence also
$ \bOf^\dagger - \Of^\dagger  = O((\lT)^{3/2}/\sqrt{T})$ where $\bOf^\dagger := (\check O' \check \Xi_+^2 \check O)^{-1} \check O' \check \Xi_+^2$.
This shows the convergence rates for the solutions to the transformed problem. It thus remains to connect the solution to the original problem to the solution of the 
transformed problem. 

It is straightforward to see that the untransformed estimate $\tilde G' = (\bOf^\dagger \check D_y \tilde \beta_{OLS,r} S_p)^{-1} \bOf^\dagger \check D_y \tilde \beta_{OLS,r}$
such that $\tilde G'S_p = I_n$.  
According to the limits above 
\begin{eqnarray*}
\bOf^\dagger \check D_y \tilde \beta_{OLS,r} \check D_z^{-1}  & = & \bOf^\dagger \tilde b_r + \bOf^\dagger (\check D_y \tilde \beta_{OLS,r} \check D_z^{-1} - \tilde b_r)    \\
& = & 
\bOf^\dagger OG'  + \bOf^\dagger \delta \check  \beta_{r} 
\end{eqnarray*}

Now let $\check D_n = \mbox{diag}(\la \tilde z_{t,1}, \tilde z_{t,1} \ra^{-1/2},I_{n-c_y})$ such that $\check D_z S_p = S_p \check D_n$. 
Then 
\begin{eqnarray*}
[\tilde G' - G'] & = & 
\check D_n^{-1} (\bOf^\dagger O + \bOf^\dagger \delta \check  \beta_{r}S_p)^{-1} (\bOf^\dagger OG'  + \bOf^\dagger \delta \check  \beta_{r})\check D_z - G' \\ 
& = &
\check D_n^{-1} (\bOf^\dagger O + \bOf^\dagger \delta \check  \beta_{r}S_p)^{-1} ((\bOf^\dagger O + \bOf^\dagger \delta \check  \beta_{r}S_p)G' - \bOf^\dagger \delta \check  \beta_{r}S_pG' + \bOf^\dagger \delta \check  \beta_{r})\check D_z - G' \\ 
& = & 
\check D_n^{-1} (\bOf^\dagger O + \bOf^\dagger \delta \check  \beta_{r}S_p)^{-1} \bOf^\dagger \delta \check  \beta_{r} (I-S_pG') \check D_z \\
& = & 
\check D_n^{-1} (\bOf^\dagger O + \bOf^\dagger \delta \check  \beta_{r}S_p)^{-1} \bOf^\dagger \delta \check  \beta_{r} \check D_z (I-S_pG')  \\
& = & \check D_n^{-1} \Of^\dagger \delta \check  \beta_{r} \check D_z (I-S_pG') + \check D_n^{-1} O((\lT)^3/T) \check D_z \\
& = &  \Of^\dagger \check D_y^{-1} \delta \check  \beta_{r} \check D_z (I-S_pG') + \check D_n^{-1} O((\lT)^3/T) \check D_z \\
& = & [O((\lT)^4/T,(\lT)^4/\sqrt{T})]
\end{eqnarray*}
where we have used that all terms have been shown above to be of order $O((\lT)^{3/2}/\sqrt{T})$. 
Further (due to the usage of the SVD and the corresponding orthogonality relations)
\begin{eqnarray*}
\tilde O - O & = & \tilde O \tilde G' \la \tilde z_t^{\pi} , \tilde z_t^{\pi} \ra   \tilde G^\dagger - O = \tilde \beta_{OLS,r} \la \tilde z_t^{\pi} , \tilde z_t^{\pi} \ra \tilde G^\dagger - O  \\
& = & \la \tilde y_t , \tilde z_t^\pi \ra \tilde G^\dagger - O  
= \la \tilde b_r \tilde z_{t} + \tilde \ve_t , \tilde z_t^\pi \ra  \tilde G^\dagger - O \\
& = & O G' \la  \tilde z_t^\pi , \tilde z_t^\pi \ra  \tilde G^\dagger - O  + 
\la  \tilde \ve_t , \tilde z_t^\pi \ra  \tilde G^\dagger \\ 
& = & 
O \tilde G'  \la \tilde z_t^{\pi} , \tilde z_t^{\pi} \ra \tilde G^\dagger 
- O  + 
O (G'-\tilde G') \la \tilde z_t^{\pi} , \tilde z_t^{\pi} \ra \tilde G^\dagger +
\la  \tilde \ve_t , \tilde z_t^\pi \ra  \tilde G^\dagger \\
& = & O (G'-\tilde G') \la \tilde z_t^{\pi} , \tilde z_t^{\pi} \ra \tilde G^\dagger +
\la  \tilde \ve_t , \tilde z_t^\pi \ra  \tilde G^\dagger
\end{eqnarray*}
where 
$$
\tilde G^\dagger :=  \tilde G (\tilde G' \la \tilde z_t^{\pi} , \tilde z_t^{\pi} \ra \tilde G)^{-1}
=  \tilde G \check D_n (\check D_n\tilde G' \la \tilde z_t^{\pi} , \tilde z_t^{\pi} \ra \tilde G\check D_n)^{-1}\check D_n.
$$ 

It then follows from
the orders of convergence provided in Lemma~\ref{lem:bound} 
that $\tilde D_z \la \tilde z_t^{\pi} , \tilde z_t^{\pi} \ra \tilde G^\dagger \tilde D_n^{-1} =  O(\lT)$. 
Consequently
\begin{eqnarray*}
[\tilde G'-G']\la \tilde z_t^{\pi} , \tilde z_t^{\pi} \ra \tilde G^\dagger  & = &  [O((\lT)^4/T), O( (\lT)^4 /\sqrt{T}] \tilde D_z^{-1} \tilde D_z \la \tilde z_t^{\pi} , \tilde z_t^{\pi} \ra \tilde G^\dagger  \\
& = & [O((\lT)^6/T),O((\lT)^6/\sqrt{T})].
\end{eqnarray*}

Furthermore
\begin{eqnarray*}
 \la \tilde \ve_t^{\pi} , \tilde z_t^{\pi} \ra \tilde G^\dagger  & = & 
 \la \tilde \ve_t^{\pi} , \tilde z_t^{\pi} \ra \la \tilde z_t^{\pi} , \tilde z_t^{\pi} \ra^{-1}
\tilde D_z^{-1} \tilde D_z \la \tilde z_t^{\pi} , \tilde z_t^{\pi} \ra \tilde G^\dagger \\
& = &  [O(P_T),O(Q_T)]\tilde D_z^{-1} O(\lT) \tilde D_n
\\
& = &  [O((\lT)^3/T), O((\lT)^3)/\sqrt{T})].
\end{eqnarray*}

Together these orders imply $\tilde O -O = [O((\lT)^6/T),O((\lT)^6/\sqrt{T})]$ 
and thus
$$
\tilde \beta_{RRR,r} - \tilde b_r = (\tilde O-O)G' + O(\tilde G'-G') + (\tilde O-O)(\tilde G'-G') = [O((\lT)^6/T),O((\lT)^6/\sqrt{T})].  
$$

Consequently we obtain from transforming~\eqref{eq:defbetau}
$$
\tilde \beta_{RRR,u} - \tilde b_u = O((\lT)^6/\sqrt{T}).
$$

This shows the convergence rates.

\subsubsection{Asymptotic Normality}

In order to derive the asymptotic distribution of the estimator $\hat \beta_{RRR}$
the proof extends the theory
contained in \citeasnoun{Anderson02}. Since the proof is rather lengthy, the main
steps are documented using lemmas summing up the main intermediate results. 

Note that the RRR estimator
is obtained from the singular value decomposition (using the symmetric matrix square roots)
$$
\la \tilde y_t^\pi, \tilde y_t^\pi \ra^{-1/2}\la \tilde y_t^\pi , \tilde z_t^\pi \ra
\la \tilde z_t^\pi , \tilde z_t^\pi \ra^{-1/2} = \hat U \hat R \hat V'.
$$

Then
as in \citeasnoun{Anderson02} (1.10), p. 205, the reduced rank estimator
can be obtained as
$$
{\mathcal T}_y\hat \beta_{RRR,r}{\cal T}_z^{-1} = \la \tilde y_t^\pi, \tilde z_t^\pi \ra \tilde G (\tilde G'\la \tilde z_t^\pi , \tilde z_t^\pi \ra \tilde G)^{-1} \tilde G'
$$
where 
$\tilde G = \la \tilde z_t^\pi , \tilde z_t^\pi \ra^{-1/2} \hat V_n {\mathcal T}_G \in \mR^{m \times n}$
satisfies the following equations
$$
\la \tilde z_t^\pi , \tilde y_t^\pi \ra \la \tilde y_t^\pi , \tilde y_t^\pi \ra^{-1}\la \tilde y_t^\pi , \tilde z_t^\pi \ra \tilde G =
\la \tilde z_t^\pi , \tilde z_t^\pi \ra \tilde G ({\mathcal T}_G^{-1}\hat R^2{\mathcal T}_G)
$$
where $\hat R^2=\mbox{diag}(\hat r_1^2,\hat r_2^2,\dots,\hat r_n^2)$ denotes the matrix
containing the squares of the $n$ largest estimated singular values as its diagonal entries.
The function of the transformation matrix ${\mathcal T}_G$ will become clear from the following.

Introduce the following notation (where in $D_z$ the subscript $_r$ is omitted for notational simplicity):
\begin{eqnarray*}
\tilde D_z & := & D_z T^{1/2} = \mbox{diag}(T^{-1/2}I,I), \quad \tilde D_y =  D_y T^{1/2} = \mbox{diag}(T^{-1/2}I,I),\\
\bar{G} & := & \tilde D_z^{-1} \tilde G, \\
\bar Q & := & \la \tilde D_z \tilde z_t^\pi , \tilde D_y \tilde y_t^\pi \ra
\la \tilde D_y \tilde y_t^\pi , \tilde D_y \tilde y_t^\pi \ra^{-1}
\la \tilde D_y \tilde y_t^\pi , \tilde D_z \tilde z_t^\pi \ra, \\
\bar M & := & \la \tilde D_z \tilde z_t^\pi , \tilde D_z \tilde z_t^\pi \ra, \\
\bar \Phi & := &  \left[ \begin{array}{ccc} T^{-1}\la \tilde z_{t,1}^\pi, \tilde z_{t,1}^\pi\ra & T^{-1}\la \tilde z_{t,1}^\pi, \tilde z_{t,2}^\pi\ra & 0 \\
T^{-1}\la \tilde z_{t,2}^\pi, \tilde z_{t,1}^\pi\ra & T^{-1}\la \tilde z_{t,2}^\pi, \tilde z_{t,1}^\pi\ra \la \tilde z_{t,1}^\pi, \tilde z_{t,1}^\pi\ra^{-1}\la \tilde z_{t,1}^\pi, \tilde z_{t,2}^\pi\ra & 0 \\
0 & 0 & \la \tilde z_{t,3}^\pi, \tilde y_{t,2}^\pi\ra \la \tilde y_{t,2}^\pi, \tilde y_{t,2}^\pi\ra^{-1}\la \tilde y_{t,2}^\pi, \tilde z_{t,3}^\pi\ra
\end{array} \right], \\
\bar \Psi & := & \left[ \begin{array}{ccc} T^{-1}\la \tilde z_{t,1}^\pi, \tilde z_{t,1}^\pi\ra & T^{-1}\la \tilde z_{t,1}^\pi, \tilde z_{t,2}^\pi\ra & 0 \\
T^{-1}\la \tilde z_{t,2}^\pi, \tilde z_{t,1}^\pi\ra & T^{-1}\la \tilde z_{t,2}^\pi, \tilde z_{t,2}^\pi\ra & 0 \\
0 & 0 & \la \tilde z_{t,3}^\pi, \tilde z_{t,3}^\pi\ra
\end{array} \right].
\end{eqnarray*}

A summary of the (unfortunately heavy) notation used can be found in Appendix~\ref{sec:notation}.
The main guideline of the notation is to use
Latin letters for matrices in which the stationary and the nonstationary subproblems
are not separated (i.e. the off-diagonal blocks potentially are nonzero) and Greek letters for matrices for the decoupled problems. A bar
indicates estimates (appropriately normalized so that convergence holds).
This leads to two generalized eigenvalue problems related to SVDs:
$$
(a) \quad \bar Q \bar G = \bar M \bar G \bar R^2, \quad 
(b) \quad \bar \Phi \bar \Gamma = \bar \Psi \bar \Gamma \bar \Theta^2. 
$$

Hence $\bar G$ denotes the solution to the original problem (a),
$\bar \Gamma$ the solution to
problem (b) where stationary and nonstationary components are
separated. 
Consequently the solutions to (b) have the form:
\begin{equation} \label{equ:defG}
\bar \Gamma = \left[ \begin{array}{cc} I 
& 0 \\
 0 & 0 \\ 0 & \bar \Gamma_{3,2} \end{array} \right]
{\to} \Gamma = \left[ \begin{array}{cc} I 
& 0 \\
 0 & 0 \\ 0 & \Gamma_{3,2} \end{array} \right]
\end{equation}
where the corresponding SVD for the stationary subproblem of (b) and its limit can be written as
$$
\begin{array}{l}
\la \tilde z_{t,3}^\pi, \tilde z_{t,3}^\pi\ra \bar \Upsilon \bar \Theta_2  = 
\la \tilde z_{t,3}^\pi, \tilde y_{t,2}^\pi\ra \la \tilde y_{t,2}^\pi, \tilde y_{t,2}^\pi\ra^{-1}
\la \tilde y_{t,2}^\pi, \tilde z_{t,3}^\pi\ra \bar \Upsilon,\\
\mE \tilde z_{t,3}^\Pi (\tilde z_{t,3}^\Pi)' \Upsilon  \Theta_2  = 
\mE \tilde z_{t,3}^\Pi (\tilde y_{t,2}^\Pi)' (\mE \tilde y_{t,2}^\Pi (\tilde y_{t,2}^\Pi)')^{-1}
\mE \tilde y_{t,2}^\Pi (\tilde z_{t,3}^\Pi)' \Upsilon.
\end{array}
$$

Solutions to these equation are not unique. In light of Lemma~\ref{lem:dKp} the restrictions $\Gamma_{3,2}'S_{p,22}=I = \bar \Gamma_{3,2}'S_{p,22}$ will be imposed. Here $S_{p,22}$ is a suitable selector 
matrix, i.e. a matrix whose columns are
columns of an identity matrix. 
W.r.o.g. it can be assumed that $S_{p,22}' = [I,0]$ by using an appropriate transformation ${\mathcal T}_z$. 
Note that this implies that $\Theta_2$ and $\bar \Theta_2$ are not necessarily diagonal. Let $S_p = [S_{p,1},S_{p,2}]$ where $S_{p,1}' = [I,0]$ and $S_{p,2}' = [0,S_{p,22}']$. 
Then $\Gamma' S_p = I = \bar \Gamma' S_p$ are sufficient restrictions to identify the solutions $\Gamma$ and $\bar \Gamma$.
Analogously $\bar G_{3,2}'S_{p,22} = I, \bar G_{1,1} = I, \bar R=\mbox{diag}(\bar{R}_{1},\bar{R}_{2})$ identify a solution (asymptotically, see Lemma~\ref{lem:dKp} and Lemma~\ref{lem:linsvd}). 
These solutions will be used in the following.  
Here $\bar \Theta_2$ and $\Theta_2$ resp. denote the $(2,2)$ blocks of
$\bar \Theta = \mbox{diag}(I,\bar \Theta_2)$ and $\Theta = \mbox{diag}(I,\Theta_2)$ respectively.

The relations between the various solutions to the generalized eigenvalue problem 
are collected in section~\ref{sec:notation}.
Throughout the rest of the proof we will use the
following notation for blocks of matrices: For a matrix $X$ partitioned into blocks we let
$X_{i,j}$ denote the blocks of the matrix. If multiple blocks are included also the
notation $'i:j'$ will be used indicating the matrix built of blocks with indices $i$ up to (and including)
$j$.
In order to denote block rows or columns
we use a semicolon for selecting the whole row or column. Hence e.g. $\bar G_{3,2}$ denotes the (3,2) block,
$\bar G_{1,:}$ the first block row and $\bar G_{1:2,1}$ the first two blocks rows in the first block column of the
matrix $\bar G$.

The next lemma establishes orders of convergence of the solutions to the generalized eigenvalue problems.
\begin{lemma} \label{lem:SVD}
Let the assumptions of Theorem~\ref{thm:RRR} hold. 

(I) Partition the matrices
$\bar Q, \bar M,\bar{\Phi}, \bar{\Psi}$ according to the partitioning of $\tilde z_{t}$
denoting the various blocks using subscripts. Then
\begin{eqnarray*}
\delta_{zz} & := & \bar{M}-\bar{\Psi} =
\left[\begin{array}{ccc} 0 & 0 & O_P(T^{-1/2}) \\ 0 & 0 & O_P(T^{-1/2}) \\ O_P(T^{-1/2}) & O_P(T^{-1/2}) & 0 \end{array} \right], \\
\delta_{yz} & := & \left[\begin{array}{ccc} T^{-1}(\la \tilde y_{t,1}^\pi, \tilde z_{t,1}^\pi \ra - \la \tilde z_{t,1}^\pi, \tilde z_{t,1}^\pi \ra)  &
T^{-1}(\la \tilde y_{t,1}^\pi, \tilde z_{t,2}^\pi \ra - \la \tilde z_{t,1}^\pi, \tilde z_{t,2}^\pi \ra ) &
T^{-1/2}\la \tilde y_{t,1}^\pi, \tilde z_{t,3}^\pi \ra \\
T^{-1/2}\la \tilde y_{t,2}^\pi, \tilde z_{t,1}^\pi \ra  &
T^{-1/2}\la \tilde y_{t,2}^\pi, \tilde z_{t,2}^\pi \ra &
0  \end{array} \right] \\
& = & \left[\begin{array}{ccc} O_P(T^{-1}) & O_P(T^{-1}) & O_P(T^{-1/2}) \\ O_P(T^{-1/2}) & O_P(T^{-1/2}) & 0 \end{array} \right], \\
\delta_{yy} & := & \left[\begin{array}{cc} T^{-1}(\la \tilde y_{t,1}^\pi, \tilde y_{t,1}^\pi \ra - \la \tilde z_{t,1}^\pi, \tilde z_{t,1}^\pi \ra)  &
T^{-1/2}\la \tilde y_{t,1}^\pi, \tilde y_{t,2}^\pi \ra \\
T^{-1/2}\la \tilde y_{t,2}^\pi, \tilde y_{t,1}^\pi \ra  &
0  \end{array} \right]
 =  \left[\begin{array}{cc} O_P(T^{-1}) & O_P(T^{-1/2})
 \\ O_P(T^{-1/2}) &  0 \end{array} \right].
\end{eqnarray*}
The terms $O_P(T^{-1})$ are  $O((\log T)(\log \log T)^{2}/T)$ and the
$O_P(T^{-1/2})$ terms are $O((\log T)(\log \log T)^{2}/T^{-1/2})$.

(II) Let $J := \bar{Q} - \bar \Phi$. To simplify notation define
$Z_{ij} := T^{-1}\la \tilde z_{t,i}^\pi, \tilde z_{t,j}^\pi \ra, i,j=1,2$. Then
\begin{equation} \label{equ:defJ}
\begin{array}{lll}
J_{i,j} & =&  [\delta_{zy}^{i1}-Z_{i1}Z_{11}^{-1}\delta_{yy}^{11}]Z_{11}^{-1}Z_{1j} +
Z_{i1}Z_{11}^{-1}\delta_{yz}^{1j} \\
& & + [\delta_{zy}^{i2}- Z_{i1}Z_{11}^{-1}\delta_{yy}^{12}] (\mE \tilde y_{t,2}^\Pi (\tilde y_{t,2}^\Pi)')^{-1}
[ \delta_{yz}^{2j}- \delta_{yy}^{21}Z_{11}^{-1}Z_{1j}] + o_P(T^{-1}),\\
J_{3,i} & = & \delta_{zy}^{31}Z_{11}^{-1}Z_{1i}  +
\la \tilde z_{t,3}^\pi, \tilde y_{t,2}^\pi\ra \la \tilde y_{t,2}^\pi, \tilde y_{t,2}^\pi \ra^{-1}
[\delta_{yz}^{2i}-\delta_{yy}^{21}Z_{11}^{-1}Z_{1i}] + o_P(T^{-1}),\\
J_{3,3} & = &
\left[ \la \tilde z_{t,3}^\pi, \tilde y_{t,2}^\pi\ra
(\la \tilde y_{t,2}^\pi, \tilde y_{t,2}^\pi\ra)^{-1}
\delta_{yy}^{21} - \delta_{zy}^{31} \right]
Z_{11}^{-1}
\left[ \delta_{yy}^{12}(\la \tilde y_{t,2}^\pi, \tilde y_{t,2}^\pi\ra)^{-1}\la \tilde y_{t,2}^\pi,
\tilde z_{t,3}^\pi\ra -\delta_{yz}^{13} \right]
 + o_P(T^{-1})
\end{array}
\end{equation}
for $i=1,2, j=1,2$ where expressions for the remaining blocks of $J$ follow from symmetry.
Hence $J_{i,j} = O_P(T^{-1})$ and indeed $J_{i,j} = O((\log T)^3 /T)$ for $i,j=1,2$.
Further $J_{3,i} = O_P(T^{-1/2})$ and indeed $J_{3,i} = O((\log T)^3/T^{-1/2})$ for $i=1,2$.
$J_{3,3} = O_P(T^{-1})$ and $J_{3,3} = O((\log T)^3 /T)$ respectively.

(III) $\delta G :=  \bar G - \Gamma = O_P(T^{-1/2})$ and moreover $\delta G = O((\log T)^3/T^{1/2})$. 
\end{lemma}
\begin{proof}
(I) The orders of convergence for the various entries of $\delta_{zz}$ and $\delta_{yz}$
follow from Lemma~\ref{lem:bound}. Details are omitted.

(II) Set $B_T:=\tilde D_y \la \tilde y_t^\pi , \tilde z_t^\pi \ra \tilde D_z$ and $A_T:= \tilde D_y \la \tilde y_t^\pi,
\tilde y_t^\pi \ra \tilde D_y$
in Lemma~\ref{lem:block} where the partitioning refers to nonstationary and stationary
components in the various matrices. Note that for the in probability part
$a_T = b_T = T^{-1/2}$ and for the almost sure convergence
$a_T=b_T = \log T (\log \log T)^{2}/T^{-1/2}$ fulfill the assumptions of Lemma~\ref{lem:block}.
Also note that $\delta_{22}^A = 0, \delta_{22}^B = 0$ simplifying the expression for $J_{3,3}$.
Then Lemma~\ref{lem:block} proves this part of the lemma. The orders of convergence for the various entries of $J$ follow from the
equations given and the orders of convergence provided in Lemma~\ref{lem:bound}.
Here also the uniform bound on the two and infinity norm of $\la \tilde z_{t,2}^u, \tilde z_{t,2}^u \ra^{-1},
\la \tilde y_{t,2}^\pi, \tilde y_{t,2}^\pi \ra^{-1}$ which are implied by Assumption~P are used.

(III) follows from (I) and (II) in combination with Lemma~\ref{lem:linsvd}.
\end{proof}
The next lemma (proven in section~\ref{proof:asy}) gathers more detailed results on the asymptotic properties of the entries of $\delta G$:
\begin{lemma} \label{lem:asydeltaG}
Let the assumptions of Theorem~\ref{thm:repr} hold. Then 
we obtain 
\begin{eqnarray*}
\delta G_{1,1} & = & 0, \\
\delta G_{1,2} & = &  -Z_{11}^{-1}Z_{12} \delta G_{2,2}
+ Z_{11}^{-1}[\delta_{zz}^{13}\bar \Gamma_{3,2}\bar \Theta_2^2 - J_{1,3}\bar \Gamma_{3,2}](I-\bar \Theta_2^2)^{-1} + o(T^{-1/2}), \\
\delta G_{2,1}& = & (T^{-1} \la \tilde z_{t,2.1}^\pi ,\tilde z_{t,2.1}^\pi\ra)^{-1}
\left[ J_{2.1,1} + 
(J_{2.1,3} - \delta_{zz}^{2.1,3} ) \delta G_{3,1}\right] + o(T^{-1}) = O((\log T)^7/T),\\
\delta G_{2,2}  & = & (T^{-1} \la \tilde z_{t,2.1}^\pi ,\tilde z_{t,2.1}^\pi\ra)^{-1} \left[
J_{2.1,3}\bar \Gamma_{3,2}\bar \Theta_2^{-2} - \delta_{zz}^{2.1,3}\bar \Gamma_{3,2} 
\right]  +o(T^{-1/2}),\\
\delta G_{3,1} & = & \bar S  [J_{3,1} - \delta_{zz}^{31}] + o(T^{-1}),\\
\bar P_{3,3}\delta G_{3,1}  & = &
(I - \la\tilde z_{t,3}^\pi, \tilde z_{t,3}^\pi\ra \bar \Gamma_{3,2}\bar \Gamma_{3,2}^\dagger)
[J_{3,1} -\delta_{zz}^{31}]  + o(T^{-1}) = O((\lT)^6/T), \\
\delta G_{3,2} & = & o(T^{-1/2})
\end{eqnarray*}
where $\tilde z_{t,2.1}^\pi = \tilde z_{t,2}^\pi - Z_{21}Z_{11}^{-1}\tilde z_{t,1}^\pi,
\delta_{zz}^{2.1,i} = \delta_{zz}^{2i} - Z_{21}Z_{11}^{-1} \delta_{zz}^{1i}, J_{2.1,1} = J_{1,1} - Z_{21}Z_{11}^{-1} J_{2,1}$,
$$
\bar S = (Z_{33} - \la \tilde z_{t,3}^\pi, \tilde y_{t,2}^\pi \ra
\la \tilde y_{t,2}^\pi, \tilde y_{t,2}^\pi \ra^{-1}\la \tilde y_{t,2}^\pi, \tilde z_{t,3}^\pi \ra)^{-1}
$$
using $Z_{33} = \la \tilde z_{t,3}^\pi , \tilde z_{t,3}^\pi \ra$ and
\begin{equation} \label{equ:defP33}
\bar P_{33} =
Z_{33} - Z_{33} \bar \Gamma_{3,2}^\dagger \bar \Gamma_{3,2}' 
Z_{33}.
\end{equation}
\end{lemma}
Next these approximations are linked to the 
estimate $\hat \beta_{RRR,r}$.
\begin{lemma} \label{lem:defhb}
Let the assumptions of Theorem~\ref{thm:RRR} hold.

(I)
Then
\begin{eqnarray*}
{\mathcal T}_y (\hat \beta_{RRR,r}-b_r) {\cal T}_z^{-1}D_z^{-1}
 & = & \left[ \sqrt{T} \la \tilde \ve_t, \tilde z_t^\pi \ra \tilde D_z \right] \bar \Gamma \bar \Gamma^\dagger
+ \left[ \begin{array}{cc} TI & 0 \\ 0 & \sqrt{T} \tO_2 \end{array} \right] \delta G'
( I- \bar M \bar G\bar G^\dagger)\\
& &
+ \left[ \begin{array}{ccc} 0 & 0 & 0 \\ 0 & 0 & \sqrt{T}[\tilde \beta_{2,3}-\bar \beta_{2,3}]
 [I- \mE \tilde z_{t,3}^\Pi(\tilde z_{t,3}^\Pi)'
\Gamma_{3,2}\Gamma_{3,2}^\dagger] \end{array} \right]+o_P(1)
\end{eqnarray*}
where
$$
\bar \beta_{2,3} = \la \tilde y_{t,2}^\pi, \tilde z_{t,3}^\pi \ra
\bar \Gamma_{3,2} \bar \Gamma_{3,2}^\dagger = \bar \O_2 \bar \Gamma_{3,2}' \to \tO_2\Gamma_{3,2}' = \tilde b_{2,3}
$$
denotes the solution to the subproblem of the problem (b) corresponding to the stationary components.
$\Gamma_{3,2}^\dagger := (\Gamma_{3,2}' \mE \tilde z_{t,3}^\pi (\tilde z_{t,3}^\pi)' \Gamma_{3,2})^{-1}\Gamma_{3,2}'$. 

(II)
Letting $\delta G_{:,1}$ and $\delta G_{:,2}$ denote the first and second block column of
$\delta G$ it holds that
\begin{eqnarray}
T\delta G_{:,1}'(I-\bar M\bar G \bar G^\dagger) & = & \left[ -T\delta H' Z_{21} Z_{11}^{-1} \quad T\delta H' \quad
T\delta G_{3,1}'\bar P \right] + o(1), \label{equ:tdeltag}\\
\sqrt{T}\delta G_{:,2}'(I-\bar M \bar G\bar G^\dagger) & = & \sqrt{T}\delta G_{2,2}'
\left[ - Z_{21} Z_{11}^{-1} \quad I \quad 0 \right] + o(1) \label{equ:sqrttdeltag}
\end{eqnarray}
where $\delta H = \delta G_{2,1} - \delta G_{2,2}(\Gamma_{3,2}^\dagger)'\mE \tilde z_{t,3}^\Pi (\tilde z_{t,3}^\Pi)'\delta G_{3,1}$
The term $o_P(1)$ is also $O((\log T)^a )$.
\end{lemma}

Part (I) of the lemma splits the estimation error $\hat \beta_{RRR,r}-b_r$ (up to errors of higher order) into three terms: The first term asymptotically equals
the estimation error in the situation that the row space of $b_r$ is known and used in the estimation
to obtain unrestricted least squares estimators.
The second term
accounts for the effects of the rank restriction in the nonstationary components
of both the output (i.e. $\tilde y_{t,1}^\pi$) as well as the regressors (i.e. $\tilde z_{t,i}^\pi, i=1,2$).
Part (II) of the lemma provides a more detailed expression for this term. 
The last term corrects for the rank restriction in the stationary directions.

Up to now only for showing
$
\sqrt{T}\delta G_{3,1}'
(I-  \la \tilde z_{t,3}^\pi, \tilde z_{t,3}^\pi \ra \bar \Gamma_{3,2}\bar \Gamma_{32}^\dagger) \to 0$
the exact form 
of $\delta_{zz}$ and $J$ are used. All other results up to now rely only on
the order of convergence of these terms.
The proof of Theorem~\ref{thm:RRR} is completed with the last lemma of this section
giving explicit expressions for the asymptotic distributions of the various parts
of the expressions for ${\mathcal T}_y(\hat \beta_{RRR,r} - b_r){\mathcal T}_{z,r}^{-1} D_{z,r}^{-1}$ given in
Lemma~\ref{lem:defhb}. The result then follows directly from~\eqref{eq:defbetau}.
\begin{lemma} \label{lem:asydist}
With $\tO_2^\dagger =
(\tO_2' (\mE \tilde y_{t,2}^\Pi (\tilde y_{t,2}^\Pi)')^{-1}\tO_2)^{-1} \tO_2'(\mE \tilde y_{t,2}^\Pi(\tilde y_{t,2}^\Pi)')^{-1}$ we have
\begin{eqnarray*}
\la \tilde \ve_{t}, \tilde z_{t,1}^\pi \ra Z_{11}^{-1} & \stackrel{d}{\to} & f({\mathcal T}_y \Lambda W,
W_{z,1}^\Pi), \\
\sqrt{T}\la \tilde \ve_t , \tilde z_{t,3}^\pi \ra \la \tilde z_{t,3}^\pi , \tilde z_{t,3}^\pi \ra^{-1} & \stackrel{d}{\to} & Z_r \sim  {\cal N}(0,V), \\
\sqrt{T}\delta G_{2,2} & = & (T^{-1}\la \tilde z_{t,2.1}^\pi, \tilde z_{t,2.1}^\pi \ra)^{-1}
\la \tilde z_{t,2.1}^\pi ,\tilde \ve_{t,2} \ra (\tO_2^\dagger)' + o_P(1) \stackrel{d}{\to} M_{r,2}'(\tO_2^\dagger)',\\
T\bar P'\delta G_{3,1} & = & Z_{33}^{-1}P \sqrt{T}\la \tilde z_{t,3}^\pi, \tilde \ve_{t,1}\ra  
+ Z_{33}^{-1}\sqrt{T}\left(\bar P\la \tilde z_{t,3}^\pi ,\tilde y_{t,2} \ra  -P\mE \tilde z_{t,3}^\Pi (\tilde y_{t,2}^\Pi)' \right)\Xi' 
+ o_P(1) \to \tilde R',\\
T\delta H  & = & (T^{-1}\la \tilde z_{t,2.1}^\pi, \tilde z_{t,2.1}^\pi \ra)^{-1}
\left[ \la \tilde z_{t,2.1}^\pi , \tilde \ve_{t,1} \ra +
\la \tilde z_{t,2.1}^\pi , \tilde \ve_{t,2}^\pi \ra (I-\tO_2 \tO_2^\dagger)'\Xi' \right] + o_P(1)\stackrel{d}{\to} N',\\
\sqrt{T}[ \bar \beta_{2,3} - \tilde \beta_{2,3} ]P & = &
\tO_2\tO_2^\dagger  \sqrt{T}\la \tilde \ve_{t,2}, \tilde z_{t,3}^\Pi \ra
(\mE \tilde z_{t,3}^\Pi (\tilde z_{t,3}^\Pi)')^{-1}P+ o_P(1),
\end{eqnarray*}
where $P = (I - \mE \tilde z_{t,3}^\Pi (\tilde z_{t,3}^\Pi)' \Gamma_{3,2}\Gamma_{3,2}^\dagger)$.
Further $Y_{i1}= \int W_{z,i}^\Pi(W_{z,1}^\Pi)'$. Here $M_{r,2} = [0,I] M_r$ with $M_r := f({\mathcal T}_y\Lambda W,W_{z,2.1}^\Pi)$ where $W_{z,2.1} := W_{z,1}^\Pi - Y_{21}Y_{11}^{-1}W_{z,1}^\Pi$.
$N = [[I,0]+\Xi(I-\tO_2\tO_2^\dagger)[0,I]] M_r$. 
$T\bar P'\delta G_{3,1}$ and $\sqrt{T}[ \tilde \beta_{2,3} - \beta_2 ]P$ converge in distribution
to Gaussian random variables with mean zero.
\end{lemma}

Combining Lemma~\ref{lem:defhb} and ~\ref{lem:asydist}
we obtain that ${\mathcal T}_y(\hat \beta_{RRR,r}-b_r){\mathcal T}_{z,r}^{-1}D_z^{-1} \stackrel{d}{\to}$
$$
\left[ f({\mathcal T}_y\Lambda W,W_{z,1}^\Pi),0, Z_r Z_{33}\Gamma_{3,2}\Gamma_{3,2}^\dagger \right] +
\left[ \begin{array}{ccc} -N Y_{21}Y_{11}^{-1} & N & \tilde R \\
-\tO_2\tO_2^\dagger M_{r,2} Y_{21}Y_{11}^{-1} & \tO_2\tO_2^\dagger M_{r,2} & \tO_2\tO_2^\dagger Z_{r,2}P \end{array} \right].
$$

From this the claim follows using the block matrix inversion since ${\mathcal T}_y (\hat \beta_{OLS,r}-b_r){\mathcal T}_z^{-1}D_z^{-1}
\stackrel{d}{\to} [ f({\mathcal T}_y \Lambda W,W_z^\Pi), Z_r ]$.

Using standard asymptotics for the term $\la \ve_t, z_t^u \ra\la z_t^u, z_t^u\ra^{-1}$ we obtain the asymptotic distribution 
of $\tilde \beta_{RRR,u} - \tilde \beta_{OLS,u}$.
stated in Theorem~\ref{thm:RRR} from~\eqref{eq:defbetau}.
Note in particular that
$R :=  Z_{r,1}P - \tilde R$
where $ Z_{r,1}$ denotes the limit of
$\sqrt{T}\la \tilde \ve_{t,1},\tilde z_{t,3}^\Pi\ra Z_{33}^{-1}$.
This concludes the proof of (I) of Theorem~\ref{thm:RRR}. 

Following the arguments of the proof and using Lemma~\ref{lem:bound} (IV) it follows that the
usual changes occur if a constant (and a deterministic trend respectively) is included in the
regression: The asymptotics in the stationary directions are unchanged. For the
nonstationary directions the Brownian motions are replaced by their corresponding
demeaned (and detrended respectively) versions. We omit details in this respect.

\subsection{Proof of Lemma~\ref{lem:asydeltaG}} \label{proof:asy}
In the proof the following results are used:
With $\Xi := -\mE \tilde \ve_{t,1} \tilde y_{t,2}' (\mE \tilde y_{t,2}^\Pi (\tilde y_{t,2}^\Pi)')^{-1},
\tilde z_{t,2.1} = \tilde z_{t,2} - Z_{21}Z_{11}^{-1} \tilde z_{t,1}, J_{2.1,1} = J_{2,1}-Z_{21}Z_{11}^{-1}J_{1,1},
J_{2.1,3} = J_{2,3}-Z_{21}Z_{11}^{-1}J_{1,3}$
and $\tilde \ve_{t,1.2} = \tilde \ve_{t,1} + \Xi \tilde y_{t,2}^\Pi$ it holds that
\begin{eqnarray*}
\sqrt{T} [\delta_{yz}^{21} - \delta_{yy}^{21}]'\la \tilde y_{t,2}^\pi, \tilde y_{t,2}^\pi \ra^{-1} & = &
\Xi + o(1), \\
\sqrt{T} [J_{3,1} - \delta_{zz}^{3,1} ] & = & \mE \tilde z_{t,3}^\Pi (\tilde y_{t,2}^\Pi)'\Xi' + o_P(1), \\
TJ_{2.1,1} & = & \la \tilde z_{t,2.1}^\pi, \tilde \ve_{t,1.2}\ra + o_P(1), \\
\sqrt{T}J_{2.1,3}
& = &
\la \tilde z_{t,2.1}^\pi,\tilde y_{t,2}^\pi \ra (\mE \tilde y_{t,2}^\Pi (\tilde y_{t,2}^\Pi)')^{-1}\tO_{2}\Gamma_{3,2}'Z_{33} + o_P(1).
\end{eqnarray*}
where $Z_{33} := \mE \tilde z_{t,3}^\Pi (\tilde z_{t,3}^\Pi)'$.

The first claim follows from
$$
\sqrt{T}[\delta_{yz}^{21} - \delta_{yy}^{21}]= \la \tilde y_{t,2}^\pi , \tilde z_{t,1}^\pi - \tilde y_{t,1}^\pi \ra
= -\la \tilde y_{t,2}^\pi , \tilde \ve_{t,1} \ra \to -\mE \tilde \ve_{t,2}  \tilde \ve_{t,1}'
$$
where $\tilde y_{t,2}^\pi = \tilde \ve_{t,2} +  \tilde b_{2,3} \tilde z_{t,3} - \la \tilde \ve_{t,2} + \tilde b_{2,3} \tilde z_{t,3}, z_t^u \ra
\la z_t^u, z_t^u \ra^{-1} z_t^u$.
Then the result follows straightforwardly since
$\tilde z_{t,3}$ and $\tilde \ve_{t,1}$ are stationary and uncorrelated by assumption.
From this also the second and third claim follow immediately using the expressions for $J_{i,1}$ derived in Lemma~\ref{lem:asydeltaG}.
The fourth claim also follows from these expressions noting that $\mE \tilde y_{t,2}^\Pi (\tilde z_{t,3}^\Pi)' =
\tO_2\Gamma_{3,2}'\mE \tilde z_{t,3}^\Pi (\tilde z_{t,3}^\Pi)'$.\\
Now with respect to the blocks of $\delta G$ note that due to the chosen normalizations $\bar \Gamma_{1,1}= \bar G_{1,1} = I$ implying $\delta G_{1,1}=0$.

Using the order of convergence for $J, \delta_{zz}$ (Lemma~\ref{lem:SVD} (I) and (II)) 
and $\delta G$ (Lemma~\ref{lem:SVD} (III)) the (1,2) block of~\eqref{equ:dg} implies
$$
\delta G_{1,2}  + Z_{11}^{-1}Z_{12} \delta G_{2,2}
= Z_{11}^{-1}
[\delta_{zz}^{13}\bar G_{3,2}\bar R_2^2 - J_{1,3}\bar G_{3,2}](I-\bar R_2^2)^{-1} + o(T^{-1}).
$$
The expression given 
then follows from noting that
$\delta_{zz}^{13} =O((\log T)^a / T^{1/2}),\bar G_{3,2} = \bar \Gamma_{3,2} + O((\log T)^3 / T^{1/2}),
J_{1,3}=O((\log T)^3 / T^{1/2})$ and
$\bar R_2^2 = \bar \Theta_2^2 + O((\log T)^3 / T^{1/2})$. \\
The expressions for $\delta G_{2,1}$ and $\delta G_{2,2}$ follow from the second block row of
equation~\eqref{eq:imgg} noting that
\begin{eqnarray*}
(I_m - \bar \Psi\bar \Gamma^\dagger\bar \Gamma') & =& 
\left[ \begin{array}{ccc} 0 & 0 & 0 \\ -Z_{21}Z_{11}^{-1} & I & 0 \\ 0 & 0 & I - \bar \Psi_{3,3}\bar \Gamma_{3,2}
(\bar \Gamma_{3,2}'\bar \Psi_{3,3}\bar \Gamma_{3,2})^{-1}\bar \Gamma_{3,2}'
\end{array} \right], \\
(I_m - \bar \Psi\bar \Gamma\bar \Gamma^\dagger)\bar \Psi & = & \left[ \begin{array}{ccc} 0 & 0 & 0 \\ 0 &
T^{-1}\la \tilde z_{t,2.1}^\pi, \tilde z_{t,2.1}^\pi\ra & 0 \\
0 & 0 & \bar P_{33} \end{array}\right].
\end{eqnarray*}
Also $\bar R_{1}^2 - I = O((\lT)^7/T)$ follows from the $(1,1)$ entry of~\eqref{equ:dg}. 
Then the $(3,1)$ entry of equation~\eqref{equ:dg} implies that
$$
\delta G_{3,1} = \bar S [J_{3,1} - \delta_{zz}^{31}] + o(T^{-1})
$$
where $\bar S^{-1}\to S^{-1}$ as 
$$
\la \tilde z_{t,3}^\pi, \tilde z_{t,3}^\pi \ra - \la \tilde z_{t,3}^\pi, \tilde y_{t,2}^\pi \ra
\la \tilde y_{t,2}^\pi, \tilde y_{t,2}^\pi \ra^{-1}\la \tilde y_{t,2}^\pi, \tilde z_{t,3}^\pi \ra
\to \mE \tilde z_{t,3}^\Pi (\tilde z_{t,3}^\Pi)' - \mE \tilde z_{t,3}^\Pi (\tilde y_{t,2}^\Pi)'
(\mE \tilde y_{t,2}^\Pi (\tilde y_{t,2}^\Pi)')^{-1}\mE \tilde y_{t,2}^\Pi(\tilde z_{t,3}^\Pi)'>0
$$
%
which is ensured by the
assumed nonsingularity of the
covariance of $[(\tilde y_{t,2}^\Pi)',(\tilde z_{t,3}^\Pi)']'$.
Further the $(3,1)$ entry in equation~\eqref{eq:imgg} directly implies the expression for
$\bar P_{3,3}\delta G_{3,1}$ using the orders of convergence derived above.
Next 
$$
J_{3,1} - \delta_{zz}^{31} = \delta_{zy}^{31} - \delta_{zz}^{31} + \la \tilde z_{t,3}^\pi, \tilde y_{t,2}^\pi \ra
\la \tilde y_{t,2}^\pi, \tilde y_{t,2}^\pi \ra^{-1}( \delta_{yz}^{21} - \delta_{yy}^{21}) 
= T^{-1/2} [ \la \tilde z_{t,3}^\pi, \tilde \varepsilon_{t,1} \ra + \la \tilde z_{t,3}^\pi, \tilde y_{t,2}^\pi \ra ( \Xi' + o(1))].
$$ 

Note that $\la \tilde z_{t,3}^\pi, \tilde \varepsilon_{t,1} \ra = O(Q_T), Z_{33}^{-1}\la \tilde z_{t,3}^\pi, \tilde y_{t,2}^\pi\ra = \Gamma_{3,2}\tO_2' + O(Q_T), \bar P - P = O((\lT)^3/T^{1/2}))$
and thus $\sqrt{T} \bar P_{3,3}\delta G_{3,1} =$
$$
= \sqrt{T}\bar P(J_{31}-\delta_{zz}^{31}) + o(T^{-1/2}) =  \bar P \la \tilde z_{t,3}^\pi, \tilde \varepsilon_{t,1} \ra + \bar P \la \tilde z_{t,3}^\pi, \tilde y_{t,2}^\pi \ra ( \Xi' + o(1))+ o(T^{-1/2})  = O((\lT)^6/T^{1/2}).
$$

Here the last order follows from $P \mE \tilde z_{t,3}^\Pi (\tilde z_{t,3}^\Pi)' \Gamma_{3,2} = 0$ as is easy to verify.

Since $\delta G=O((\log T)^3)/T^{1/2})$ also $\delta G_{1,2}=O((\log T)^3/T^{1/2})$
and $\delta G_{2,2} = O((\log T)^3/T^{1/2})$.
For $\bar \Gamma_\bot$ such that $\bar \Gamma_\bot' \bar \Gamma_{3,2}= 0$
we have that $\bar \Gamma_\bot' Z_{33}^{-1}\bar P_{33} =  \bar \Gamma_\bot'$. 
Then the (3,2) block of~\eqref{eq:imgg} shows that
$$
[0,\bar \Gamma_\bot']  \delta G_{:,2} = \bar \Gamma_\bot' \delta G_{3,2} =
\bar \Gamma_\bot' Z_{33}^{-1}\bar P_{33} \delta G_{3,2} =
[0,\bar \Gamma_\bot'Z_{33}^{-1}][J\bar G_{:,2}\bar R_2^{-2} - \delta_{zz}\bar G_{:,2}] =  o(T^{-1/2})
$$
since the (3,2) block entry of $J\bar{G}$ and $\delta_{zz}\bar G$ both are of order $o(T^{-1/2})$
as follows from the norm bounds given for the blocks of $J$ and $\delta_{zz}$.

Due to the chosen normalization $\bar \Gamma_{3,2}'S_{p,22} = \bar G_{3,2}'S_{p,22} = I$ 
and thus $S_{p,22}' \delta G_{3,2}=0$. 
Since $[\bar \Gamma_\bot,S_{p,22}]$ is nonsingular (as is straightforward to see from $\Gamma_{3,2}'S_{p,22}=I$) 
we obtain $\delta G_{3,2} = o(T^{-1/2})$. 

The order of convergence for $\delta G_{2,1}$ follows from the orders of convergence of $J, \delta_{zz}$ and $\delta G$ as derived in 
Lemma~\ref{lem:SVD}.

\subsection{Proof of Lemma~\ref{lem:defhb}}
(I) Let $\tilde \beta_{RRR,r} = {\mathcal T}_y \hat \beta_{RRR,r} {\mathcal T}_{z,r}^{-1}$ and $\tilde b_r =
{\mathcal T}_y b_r {\mathcal T}_{z,r}^{-1}$.
Then
$$
\begin{array}{ccc}
\tilde \beta_{RRR,r}  & = &
\la \tilde y_t^\pi , \tilde z_t^\pi \ra \hat G\hat G^\dagger =\tilde b_r \la \tilde z_t^\pi , \tilde z_t^\pi \ra \tilde D_z
\bar G\bar G^\dagger \tilde D_z' +
\la \tilde \ve_t , \tilde z_t^\pi \ra \tilde D_z \bar G\bar G^\dagger \tilde D_z
\\ & \doteq & \tilde b_r \la \tilde z_t^\pi , \tilde z_t^\pi \ra \tilde D_z \bar G\bar G^\dagger \tilde D_z' +
\left[ \sqrt{T}\la \tilde \ve_t , \tilde z_t^\pi \ra  \tilde D_z  \right] \bar \Gamma \bar \Gamma^\dagger D_z
\end{array}
$$
where $\doteq$ stands for equality up to terms of order $o(T^{-1})$ in the first $c_z$ columns
and of order $o(T^{-1/2})$ in the remaining columns.
This follows from $\la \tilde \ve_t, \tilde z_t^\pi \ra \tilde D_z = O((\log T)^3)T^{-1/2}),
\delta G = O((\log T)^3 T^{-1/2})$ and the definition of $D_z=\mbox{diag}(T^{-1}I,T^{-1/2}I)$
showing that $\bar G$ can be replaced by $\bar \Gamma$ in the second term in the last equation
with introduction of an error of the stated order.
Now since $\tilde b_r$ is block diagonal
$$
\tilde b_r \la \tilde z_t^\pi , \tilde z_t^\pi \ra \tilde D_z \bar G\bar G^\dagger \tilde D_z
= \tilde D_y^{-1} \tilde b_r \tilde D_z \la \tilde z_t^\pi , \tilde z_t^\pi \ra \tilde D_z
\bar G\bar G^\dagger \tilde D_z.
$$

Next note that (using the index '$:$' to denote block columns or rows resp.)
$\tilde b_{r,1,:} = [I,0]= \bar \Gamma_{:,1}'$.
Therefore we have 
recalling that $(\bar{G}^\dagger)' = \bar G(\bar G'\bar M \bar G)^{-1}$
\begin{eqnarray*}
[I,0] & = &
\bar{G}_{:,1}' \bar M 
(\bar{G}^\dagger)'
= (\bar \Gamma_{:,1})' \bar M 
(\bar{G}^\dagger)' +
\delta G_{:,1}' \bar M 
(\bar{G}^\dagger)' \\
& = & \tilde b_{r,1,:} \bar M 
(\bar{G}^\dagger)' +
\delta G_{:,1}' \bar M  (\bar{G}^\dagger)'.
\end{eqnarray*}

This leads to
$\tilde b_{r,1,:} \bar M 
(\bar{G}^\dagger)' = [ [I,0] -
\delta G_{:,1}' \bar M  (\bar{G}^\dagger)']$.
Hence we obtain
$$
\begin{array}{ccc}
\tilde b_{r,1,:} \la \tilde z_t^\pi , \tilde z_t^\pi \ra \tilde D_z \bar G\bar G^\dagger \tilde D_z
 & = & T^{1/2} \left[ [I,0] - \delta G_{:,1}' \bar M (\bar G^\dagger)'\right]\bar{G}'
\tilde D_z \\
 & = & \tilde b_{r,1,:} +  T\left[ \delta G_{:,1}' \right] \left[ I - \bar M\bar G\bar G^\dagger \right] D_z.
\end{array}
$$

Here
$\bar{G}_{:,1}' =[I,0]+\delta G_{:,1}'$ is used in the last line. 

With respect to the second
block row it follows from $\tilde b_{r,2,:} = [0,\tilde b_{2,3}]$ that
$\tilde b_{r,2,:}= \tilde b_{r,2,:}\tilde D_{z}$. Also we have
from $\bar \beta_{2,3}=\bar \O_2\bar \Gamma_{3,2}'$ using $\bar G_{:,2}'\bar M \bar G\bar G^\dagger = \bar G_{:,2}'$
by the definition of $\bar G^\dagger$ that
\begin{eqnarray*}
\bar \beta_{2,3}   \la \tilde z_{t,3}^\pi , \tilde z_t^\pi \ra \tilde D_z \bar G\bar G^\dagger \tilde D_z
& = & \bar \O_2\bar \Gamma_{3,2}'
\la \tilde z_{t,3}^\pi , \tilde z_t^\pi \ra \tilde D_z \bar G\bar G^\dagger \tilde D_z
 =  \bar \O_2 \bar \Gamma_{:,2}'\bar M \bar G\bar G^\dagger \tilde D_z \\
& = & \bar \O_2[\bar G_{:,2}' \tilde D_z - \delta G_{:,2}'\bar M \bar G\bar G^\dagger \tilde D_z]
 =  \bar \O_2[\bar \Gamma_{:,2}' + \delta G_{:,2}'(I-\bar M \bar G\bar G^\dagger) \tilde D_z] \\
& = & [0,\bar \beta_{2,3}] + \bar \O_2 \delta G_{:,2}'(I-\bar M \bar G\bar G^\dagger) \tilde D_z.
\end{eqnarray*}

This implies
\begin{eqnarray*}
\tilde b_{r,2,:}  \la \tilde z_t^\pi , \tilde z_t^\pi \ra \tilde D_z \bar G\bar G^\dagger \tilde D_z
& = & \tilde b_{r,2,:} \tilde D_z  \la \tilde z_t^\pi , \tilde z_t^\pi \ra \tilde D_z \bar G\bar G^\dagger \tilde D_z \\
&  = & [\tilde b_{2,3} - \bar \beta_{2,3}]  \la \tilde z_{t,3}^\pi , \tilde z_t^\pi \ra \tilde D_z \bar G\bar G^\dagger \tilde D_z +
\bar \beta_{2,3}  \la \tilde z_{t,3}^\pi , \tilde z_t^\pi \ra \tilde D_z \bar G\bar G^\dagger \tilde D_z \\
&  \doteq  &
[0,[\tilde b_{2,3} - \bar \beta_{2,3}]\mE \tilde z_{t,3}^\Pi(\tilde z_{t,3}^\Pi)'\Gamma_{3,2}\Gamma_{3,2}^\dagger] +
\bar \beta_{2,3}   \la \tilde z_{t,3}^\pi , \tilde z_t^\pi \ra \tilde D_z \bar G\bar G^\dagger \tilde D_z \\
& = & [0,[\tilde b_{2,3} - \bar \beta_{2,3}]   \mE \tilde z_{t,3}^\Pi (\tilde z_{t,3}^\Pi)'
\Gamma_{3,2}\Gamma_{3,2}^\dagger] +
[0,\bar \beta_{2,3}] +
\bar \O_2 \delta G_{:,2}' \left[ I - \bar M \bar G\bar G^\dagger \right]   \tilde D_z \\
& = & \tilde b_{r,2,:} + [0,[\tilde b_{2,3}- \bar \beta_{2,3}] \left[
\mE \tilde z_{t,3}^\Pi (\tilde z_{t,3}^\Pi)'\Gamma_{3,2}\Gamma_{3,2}^\dagger - I \right] ] +
\bar \O_2 \sqrt{T}\delta G_{:,2}' \left[ I - \bar M \bar G\bar G^\dagger \right]   D_z.
\end{eqnarray*}

Here the third line follows from the orders of convergence in $\delta G = \bar{G} - \bar \Gamma$
established above and $\tilde b_{2,3} - \bar \beta_{2,3} = o(T^{-\epsilon}), \epsilon>0$
as follows from standard theory in the stationary case.
Then the representation given in (I) is proved by replacing $\bar \O_2$ by its limit $\tilde O_2$
which introduces an error of the required form since $\bar \O_2 - \tilde O_2 = o(T^{-\epsilon})$ for some $\epsilon>0$
as follows from $\bar\beta_{2,3} -\tilde b_{2,3} = o(T^{-\epsilon})$ (see the proof of Lemma~\ref{lem:dKp}).

(II) For ~\eqref{equ:sqrttdeltag} note that $\delta G$ and
$\bar M- \bar \Psi$ both are of order $O((\log T)^3 T^{-1/2})$
and the two norm of $\bar \Psi$ and $\bar \Gamma$ is of order $O(\log T)$. 
Therefore replacing
$\bar M \bar G\bar G^\dagger$ by $\bar \Psi \bar \Gamma \bar \Gamma^\dagger$ introduces
an error of order
$O((\log T)^{4} T^{-1/2})=o(1)$ proving ~\eqref{equ:sqrttdeltag} since
$\delta G_{3,2}=o(T^{-1/2})$ (see Lemma~\ref{lem:asydeltaG}) and
$$
\sqrt{T}\delta G_{:,2}'(I-\bar \Psi \bar \Gamma \bar \Gamma^\dagger) = 
\sqrt{T}\delta G_{2,2}[ -Z_{21}Z_{11}^{-1} \quad I \quad 0] + \sqrt{T}\delta G_{3,2}' [ 0 \quad 0 \quad I - Z_{33}\bar \Gamma_{3,2}\bar \Gamma_{3,2}^\dagger].
$$

With respect to~\eqref{equ:tdeltag} note that
$$
T\delta G_{:,1}'\left[ I - \bar M \bar G\bar G^\dagger \right]  =
T\delta G_{:,1}'\left[ I - \bar \Psi \bar \Gamma \bar \Gamma^\dagger \right] +
T\delta G_{:,1}'\left[ \bar \Psi \bar \Gamma\bar \Gamma^\dagger  - \bar M \bar G\bar G^\dagger \right]
$$
where
\begin{equation}\label{eq:dG11}
T\delta G_{:,1}'\left[ I - \bar \Psi \bar \Gamma \bar \Gamma^\dagger \right]
=   \left[ -T\delta G_{2,1}'Z_{21}Z_{11}^{-1} \quad
T\delta G_{2,1}'  
\quad
T\delta G_{3,1}'\bar P \right]
\end{equation}
is obvious from the form of $I-\bar \Psi \bar \Gamma \bar \Gamma^\dagger$ (see the proof of Lemma~\ref{lem:asydeltaG}).

Noting that $\bar M \bar G\bar G^\dagger = \bar M \bar G(\bar G'\bar M \bar G)^{-1}\bar G'$
it follows that
\begin{eqnarray*}
\delta G_{:,1}' [ \bar \Psi \bar \Gamma \bar \Gamma^\dagger - \bar M \bar G \bar G^\dagger] & = &
\delta G_{:,1}' [ (I-\bar \Psi \bar \Gamma \bar \Gamma^\dagger)(\bar \Psi - \bar M)\bar \Gamma \bar \Gamma^\dagger
- \bar \Psi \bar \Gamma (\bar \Gamma' \bar \Psi \bar \Gamma)^{-1} \delta G' (I-\bar \Psi \bar \Gamma \bar \Gamma^\dagger) \\
& & -(I-\bar \Psi \bar \Gamma \bar \Gamma^\dagger) \bar \Psi \delta G \bar \Gamma^\dagger ]+o(T^{-1}) \\
& = &
-\delta G_{3,1}' \la \tilde z_{t,3}^\pi, \tilde z_{t,3}^\pi \ra \bar \Gamma_{3,2}(\bar \Gamma_{3,2}' Z_{3,3} 
\bar \Gamma_{3,2})^{-1} \delta G_{:,2}'(I-\bar \Psi \bar \Gamma \bar \Gamma^\dagger) + o(T^{-1})
\end{eqnarray*}
since $\delta G_{1,1}=o(T^{-1}), \delta G_{2,1} = O((\log T)^7 /T)$
and (see the proof of Lemma~\ref{lem:asydeltaG} and~\eqref{equ:defJ})
\begin{eqnarray*}
\sqrt{T}\delta G_{3,1}'
(I-  Z_{33} \bar \Gamma_{3,2}\bar \Gamma_{32}^\dagger) & = &
\sqrt{T}[J_{3,1}-\delta_{z,z}^{31}]'\bar S
(I-  Z_{33} \bar \Gamma_{3,2}\bar \Gamma_{32}^\dagger)  + o(T^{-\epsilon})\\
& = &
[\la \tilde z_{t,3}^\pi, \tilde \ve_{t,1} \ra + \la \tilde z_{t,3}^\pi, \tilde y_{t,2} \ra \Xi']'\bar S
(I-  Z_{33} \bar \Gamma_{3,2}\bar \Gamma_{32}^\dagger) + o(T^{-\epsilon}) \\
& = & o(T^{-\epsilon})
\end{eqnarray*}
for some $\epsilon > 0$
due to $\la \tilde z_{t,3}^\pi, \tilde \ve_{t,1} \ra = o(T^{-\epsilon})$
for $0<\epsilon < 1/2$ and $\la \tilde z_{t,3}^\pi, \tilde y_{t,2} \ra \to \mE \tilde z_{t,3}^\Pi (\tilde y_{t,2}^\Pi)'
= \mE \tilde z_{t,3}^\Pi(\tilde z_{t,3}^\Pi)' \Gamma_{3,2}\tO_2'$.
Now
$\la \tilde y_{t,2}^\pi, \tilde z_{t,3} \ra \bar S(I-  Z_{33} \bar \Gamma_{3,2}\bar \Gamma_{3,2}^\dagger)\to 0$
according to
\begin{equation} \label{equ:convS}
\bar S Z_{33}\bar \Gamma_{3,2} \to S \mE \tilde z_{t,3}^\Pi(\tilde z_{t,3}^\Pi)' \Gamma_{3,2}.
\end{equation}

Recall that by definition 
$$
\mE \tilde z_{t,3}^\Pi(\tilde z_{t,3}^\Pi)'\Gamma_{3,2}\Theta_2^2 = 
\mE \tilde z_{t,3}^\Pi(\tilde y_{t,2}^\Pi)'(\mE \tilde y_{t,2}^\Pi(\tilde y_{t,2}^\Pi)')^{-1}\mE \tilde y_{t,2}^\Pi(\tilde z_{t,3}^\Pi)'\Gamma_{3,2}.
$$

Together with the definition of $S = (\mE \tilde z_{t,3}^\Pi(\tilde z_{t,3}^\Pi)' - \mE \tilde z_{t,3}^\Pi(\tilde y_{t,2}^\Pi)'(\mE \tilde y_{t,2}^\Pi(\tilde y_{t,2}^\Pi)')^{-1}\mE \tilde y_{t,2}^\Pi(\tilde z_{t,3}^\Pi)')^{-1}$ this implies 
$$
S \mE \tilde z_{t,3}^\Pi(\tilde z_{t,3}^\Pi)' \Gamma_{3,2} = \Gamma_{3,2} (I-\Theta_2^2)^{-1}.
$$

Finally $\sqrt{T}[\delta_{yz}^{21}-\delta_{yy}^{21}]\la \tilde y_{t,2}^\pi, \tilde y_{t,2}^\pi \ra^{-1}
= \Xi+ o(T^{-\epsilon})$ is used (as derived above).

Since 
$\delta G_{3,2}=o(T^{-1/2-\epsilon})$
(see Lemma~\ref{lem:asydeltaG})
it follows from the form of $(I-\bar \Psi \bar\Gamma \bar \Gamma^\dagger)$ (see the proof of Lemma~\ref{lem:asydeltaG})
that
\begin{equation} \label{eq:dG12}
T\delta G_{:,1}' [ \bar \Psi \bar \Gamma \bar \Gamma^\dagger - \bar M \bar G \bar G^\dagger]
=
-(\sqrt{T} \delta G_{3,1})' \mE \tilde z_{t,3}^\Pi (\tilde z_{t,3}^\Pi)' (\Gamma_{3,2}^\dagger)'
(\sqrt{T} \delta G_{2,2}')[-Z_{21}Z_{11}^{-1},I,0] + o(1).
\end{equation}
Combining \eqref{eq:dG11} and \eqref{eq:dG12} we obtain
$$
T\delta G_{:,1}'\left[ I - \bar M \bar G\bar G^\dagger \right]
= T\left( \delta G_{2,1}' - \delta G_{3,1}' \mE \tilde z_{t,3}^\Pi (\tilde z_{t,3}^\Pi)' (\Gamma_{3,2}^\dagger)'
 \delta G_{2,2}' \right)'\left[ -Z_{21}Z_{11}^{-1},
\quad  I , \quad 0 \right]
+
\left[0 \quad 0 \quad T\delta G_{3,1}'\bar P \right] + o(1).
$$

This completes the proof by the definition of $\delta H$.

\subsection{Proof of Lemma~\ref{lem:asydist}}
The first claim is standard and follows from Lemma~\ref{lem:bound}
using $n_t := [\tilde z_{t,1}',(\tilde z_{t,1}^u)']', v_t := {\mathcal T}_y\Lambda \ve_t$
noting that then $v_t$ is a martingale difference.
The second claim is a standard central limit result.
Further from Lemma~\ref{lem:asydeltaG} we have
$$
\sqrt{T} \delta G_{2,2} = (T^{-1}\la \tilde z_{t,2.1}^\pi, \tilde z_{t,2.1}^\pi \ra)^{-1}
\sqrt{T}\left[ J_{2.1,3}\Gamma_{3,2}\Theta_2^{-2} - \delta_{zz}^{2.1,3}\Gamma_{3,2} \right] + o_P(1).
$$

Now from the proof of Lemma~\ref{lem:asydeltaG}
\begin{eqnarray*} \sqrt{T} J_{2.1,3}\Gamma_{3,2}  & = &
\la \tilde z_{t,2.1}^\pi,\tilde y_{t,2}^\pi \ra (\mE \tilde y_{t,2}^\Pi(\tilde y_{t,2}^\Pi)')^{-1}\tO_2\Gamma_{3,2}'
Z_{33}\Gamma_{3,2} + o_P(1) \\
 & = & \la \tilde z_{t,2.1}^\pi,\tilde \ve_{t,2} \ra (\mE \tilde y_{t,2}^\Pi(\tilde y_{t,2}^\Pi)')^{-1}\tO_2(\tO_2'(\mE \tilde y_{t,2}^\Pi(\tilde y_{t,2}^\Pi)')^{-1}\tO_2)^{-1}\Theta_2^2 \\
& & + \la \tilde z_{t,2.1}^\pi,\tilde z_{t,3}^\pi \ra \Gamma_{3,2} \Theta_2^2 + o_P(1)
\end{eqnarray*}
since $\Theta_2^2 = \tO_2'(\mE \tilde y_{t,2}^\Pi(\tilde y_{t,2}^\Pi)')^{-1}\tO_2(\Gamma_{3,2}' Z_{33}\Gamma_{3,2})$ 
as is straightforward to show. 
This shows the expression for $\sqrt{T}\delta G_{2,2}$
since $\sqrt{T}\delta_{zz}^{2.1,3} = \la \tilde z_{t,2.1}^\pi,\tilde z_{t,3}^\pi \ra$.\\
In order to obtain the expression for $T\bar P'\delta G_{3,1}$ note that from the
definition of $\bar P, \bar P_{33}$ and
the arguments in the proof of Lemma~\ref{lem:asydeltaG}
$$
\la \tilde z_{t,3}^\pi, \tilde z_{t,3}^\pi \ra \bar P' \delta G_{3,1} =
\bar P_{33}\delta G_{3,1} = \bar P[J_{3,1}- \delta_{zz}^{31} ] + o(T^{-1}).
$$

Now $\sqrt{T}[J_{3,1} - \delta_{zz}^{31}] \to \mE \tilde z_{t,3}^\Pi(\tilde z_{t,3}^\Pi)' \Gamma_{3,2}\tO_2'\Xi'$ according to the proof of Lemma~\ref{lem:asydeltaG} 
and
$\bar P \to I - \mE \tilde z_{t,3}^\Pi(\tilde z_{t,3}^\Pi)' \Gamma_{3,2}\Gamma_{3,2}^\dagger$ where the difference is
of order $O_P(T^{-1/2})$ since only stationary components are involved.
The result then follows from the expression for $J_{3,1}$ according to Lemma~\ref{lem:SVD} (II) since
$$
(I - \mE \tilde z_{t,3}^\Pi(\tilde z_{t,3}^\Pi)' \Gamma_{3,2}\Gamma_{3,2}^\dagger)\mE \tilde z_{t,3}^\Pi(\tilde z_{t,3}^\Pi)' \Gamma_{3,2}O_2'\Xi' =0.
$$

The evaluations for $\delta H$ are more involved:
Using the expressions given in Lemma~\ref{lem:asydeltaG}
and defining $Z_{22.1} = \la \tilde z_{t,2.1}, \tilde z_{t,2.1}\ra$ we have
$\delta H = \delta G_{2,1} - \delta G_{2,2}(\Gamma_{3,2}^\dagger)'\mE \tilde z_{t,3}^\Pi(\tilde z_{t,3}^\Pi)' \delta G_{3,1}=$
%
$$
\begin{array}{l}
=
Z_{22.1}^{-1}\left[ J_{2.1,1} + \left\{ J_{2.1,3}-\delta_{zz}^{2.1,3}
- (J_{2.1,3}\Gamma_{3,2}\Theta_2^{-2} - \delta_{zz}^{2.1,3} \Gamma_{3,2})
\Gamma_+  \right\}\delta G_{3,1} \right] +o(T^{-1})\\
= Z_{22.1}^{-1}\left[ J_{2.1,1} + \left\{
J_{2.1,3}-\delta_{zz}^{2.1,3}
- \left(J_{2.1,3}\Gamma_{3,2}(\Theta_2^{-2}-I)+(J_{2.1,3}-\delta_{zz}^{2.1,3})\Gamma_{3,2}\right)\Gamma_{+} \right\} \delta G_{3,1} \right] +o(T^{-1})\\
= Z_{22.1}^{-1}\left[ J_{2.1,1} +
(J_{2.1,3}-\delta_{zz}^{2.1,3})\left(I-\Gamma_{3,2}\Gamma_{+} \right)\delta G_{3,1}
- J_{2.1,3}\Gamma_{3,2}(\Theta_2^{-2}-I)(\Gamma_{3,2}^\dagger)'\mE \tilde z_{t,3}^\Pi(\tilde z_{t,3}^\Pi)' \delta G_{3,1} \right] + o(T^{-1})\\
= Z_{22.1}^{-1}\left( J_{2.1,1} -
J_{2.1,3}/\sqrt{T} \Gamma_{3,2}\Theta_2^{-2}\tO_2'\Xi'
\right)+ o_P(T^{-1})
\end{array}
$$
%
(using $\Gamma_+ := (\Gamma_{3,2}^\dagger)'\mE \tilde z_{t,3}^\Pi(\tilde z_{t,3}^\Pi)'$)
where the last line follows from
$T(J_{2.1,3}-\delta_{zz}^{2.1,3})\left(I-\Gamma_{3,2}(\Gamma_{3,2}^\dagger)' Z_{33}\right)\delta G_{3,1} \to 0$
in probability since
$\sqrt{T} [J_{2.1,3}-\delta_{zz}^{2.1,3}]$ converges in distribution,
$\sqrt{T}\delta G_{3,1} \to S Z_{33}\Gamma_{3,2}\tO_2'\Xi'$
and
$
S Z_{33}\Gamma_{3,2} 
= \Gamma_{3,2} (I-\Theta_2^2)^{-1}
$
(see ~\eqref{equ:convS}). The result then follows from some algebraic operations.

The final statement 
applies Lemma~\ref{lem:dKp}:
\begin{equation} \label{eq:dKp}
\bar \Gamma_{3,2}' - \Gamma_{3,2}' = \tilde O_2^\dagger (\la \tilde y_{t,2}^\pi, \tilde z_{t,3}^\pi \ra
(\la \tilde z_{t,3}^\pi, \tilde z_{t,3}^\pi \ra)^{-1} - \tilde \beta_{2,3} )
(I - S_{p,22}\Gamma_{3,2}') + o( T^{-1/2}).
\end{equation}

This shows the result since 
$\bar \beta_{2,3} - \tilde b_{2,3} =
(\tilde {\mathcal O}_2 -\tilde O_2)\bar \Gamma_{3,2}' + \tilde O_2(\bar \Gamma_{3,2}' - \Gamma_{3,2}')$ and
$\bar \Gamma_{3,2}'(I-Z_{33}\Gamma_{3,2}\Gamma_{3,2}^\dagger)\to 0$.
The fact that 
$(I-S_{p,22}\Gamma_{3,2}')(I-Z_{33}\Gamma_{3,2}\Gamma_{3,2}^\dagger)=
(I-Z_{33}\Gamma_{3,2}\Gamma_{3,2}^\dagger)$ simplifies the expressions.
This concludes the proof.

\subsection{Proof of Theorem ~\ref{thm:FM}}

The proof 
follows closely the proof in the OLS case. The changes in comparison to the OLS case are that $\la y_t, z_t \ra$ is replaced with 
$$
\hYZ : = \la y_t , z_t^\pi \ra - \hD_{\hat u,\Delta z^\pi} - \hO (\la \Delta z_t, z_t^\pi \ra - \hD_{\Delta z, \Delta z^\pi}) = \la y_t , z_t^\pi \ra + \hat B^+ 
$$
where hence the additional term is called $\hat B^+$ and in the SVD the weighting is not based on $\la y_t^\pi ,y_t^\pi\ra$ but on $W_+ = \left(  
\la y_t^\pi , y_t^\pi \ra + \hat C^+
\right)^{-1}$ where
$$
\hat C^+ := - \hO (\la \D z_t, y_t^\pi \ra - \hD_{\D z,\D y^\pi}) - (\la \D z_t, y_t^\pi \ra - \hD_{\D z,\D y^\pi})'(\hO)'.
$$

The asymptotics for the additional terms are detailed in the lemma below:
\begin{lemma} \label{lem:addterms}
Under the assumptions of Theorem~\ref{thm:FM} the following holds:
\begin{eqnarray*}
\hD_{\hat u,\Delta \tilde z^\pi} + \hO (\la \Delta z_t, \tilde z_t^\pi \ra - \hD_{\Delta z, \Delta \tilde z^\pi}) & = & 
\left[ \begin{array}{cc} \Omega_{\hat u\Delta z}^{:,n}(\Omega_{\Delta z\Delta z}^{n,n})^{-1}\int dB_{n} B_{n}' +o_P(1) & o_P(T^{-1/2}) \end{array} \right], \\
\hO (\la \D z_t, \tilde y_t^\pi \ra - \hD_{\D z,\D \tilde y^\pi}) & = & \left[ \begin{array}{cc} \Omega_{\hat u\Delta z}^{:,n}(\Omega_{\Delta z\Delta z}^{n,n})^{-1}\int dB_{n} B_{n}' 
\left( \begin{array}{c} I \\ 0 
\end{array} \right)  
+o_P(1) & o_P(T^{-1/2}) \end{array} \right]
\end{eqnarray*} 
\end{lemma} 
\begin{proof}
The proof of the first statement uses the fact that according to Lemma~\ref{lem:omegadelta}
$\hD_{\hat u,\Delta z} = [o_P(1),o_P(T^{-1/2})]$. 
Here the restrictions on the increase of $K$ as a function of $T$ is used such that 
we obtain $\sqrt{K/T} = o_P(1), 1/\sqrt{KT} = o_P(T^{-1/2}), K^{-2} = o_P(T^{-1/2})$.  
Hence it is of lower order compared to the 
leading terms. Further $\hO (\la \Delta z_t, \tilde z_t \ra - \hD_{\Delta z, \Delta \tilde z}) = [O_P(1),o_P(T^{-1/2})]$
and hence these terms are of the same order as the leading terms.
The proof of the second statement is an easy consequence of the results listed in Lemma~\ref{lem:omegadelta} using $\tilde y_t^\pi = \tilde b_r \tilde z_t^\pi + \tilde \varepsilon_t^\pi$.
Note that
$$
\la \D z_t, \tilde y_t^\pi \ra - \hD_{\D z,\D \tilde y^\pi} = (\la \D z_t, \tilde z_t^\pi \ra - \hD_{\D z,\D \tilde z^\pi})\tilde b_r'  + 
(\la \D z_t, \tilde u_t \ra - \hD_{\D z,\D \tilde u^\pi})
$$

Convergence for the first summand is contained as the first statement in the lemma, while 
again following Lemma~\ref{lem:omegadelta} we obtain convergence for the second term.  
\end{proof}

For both $\hat \Sigma_{y,z}$ and $W_+$ after transformation using the 
matrices ${\mathcal T}_y, {\mathcal T}_z$ the additional terms in the diagonal blocks are of lower order than the original terms. 
For the off-diagonal blocks the additional terms are of the same order 
in probability. 
This follows from the results in Lemma~\ref{lem:omegadelta}. 
However, for the off-diagonal terms in the consistency proof 
only the order of convergence is used.
Consequently the consistency result and the order of convergence (in probability) also hold in the FM case.  

In the following we will use the following definitions using the same notation as in the OLS case in order to avoid the 
introduction of new symbols.  
\begin{eqnarray*}
\bar Q & := & \hYZ' W_+^{-1} \hYZ, \\
\bar M & := & \la \tilde D_z \tilde z_t^\pi , \tilde D_z \tilde z_t^\pi \ra, \\
\bar \Phi & := &  \left[ \begin{array}{ccc} T^{-1}\la \tilde z_{t,1}^\pi, \tilde z_{t,1}^\pi\ra & T^{-1}\la \tilde z_{t,1}^\pi, \tilde z_{t,2}^\pi\ra & 0 \\
T^{-1}\la \tilde z_{t,2}^\pi, \tilde z_{t,1}^\pi\ra & T^{-1}\la \tilde z_{t,2}^\pi, \tilde z_{t,1}^\pi\ra \la \tilde z_{t,1}^\pi, \tilde z_{t,1}^\pi\ra^{-1}\la \tilde z_{t,1}^\pi, \tilde z_{t,2}^\pi\ra & 0 \\
0 & 0 & \la \tilde z_{t,3}^\pi, \tilde y_{t,2}^\pi\ra \hYY^{-1} 
\la \tilde y_{t,2}^\pi, \tilde z_{t,3}^\pi\ra
\end{array} \right], \\ 
\bar \Psi & := & \left[ \begin{array}{ccc} T^{-1}\la \tilde z_{t,1}^\pi, \tilde z_{t,1}^\pi\ra & T^{-1}\la \tilde z_{t,1}^\pi, \tilde z_{t,2}^\pi\ra & 0 \\
T^{-1}\la \tilde z_{t,2}^\pi, \tilde z_{t,1}^\pi\ra & T^{-1}\la \tilde z_{t,2}^\pi, \tilde z_{t,2}^\pi\ra & 0 \\
0 & 0 & \la \tilde z_{t,3}^\pi, \tilde z_{t,3}^\pi\ra
\end{array} \right].
\end{eqnarray*}

Here 
$$
\hYY := \la \tilde y_{t,2}^\pi, \tilde y_{t,2}^\pi \ra - [0,I]' \left[ \hO (\la \D z_t, \tilde y_t^\pi \ra - \hD_{\D z,\D \tilde y}) - (\la \D z_t, \tilde y_t^\pi \ra - 
\hD_{\D z,\D \tilde y})'(\hO)' \right] [0,I]'$$ 

such that $\hYY \to \YY := \mE \tilde y_{t,2}^\Pi (\tilde y_{t,2}^\Pi)'$.

Hence the definition of $\bar Q$ is adapted to the SVD occurring in the FM estimation. Also the $(3,3)$ entry of $\bar \Phi$ is 
changed slightly. The reason for this is visible in the next Lemma~\ref{lem:SVDfm} 
which is the analogon to Lemma~\ref{lem:SVD} for the OLS case. 
\begin{lemma}\label{lem:SVDfm} 
Let the assumptions of Theorem~\ref{thm:FM} hold. 

(I) Partition the matrices
$\bar Q, \bar M,\bar{\Phi}, \bar{\Psi}$ according to the partitioning of $\tilde z_{t}$
denoting the various blocks using subscripts. Then:
\begin{eqnarray*}
\delta_{zz} & := & \bar{M}-\bar{\Psi} =
\left[\begin{array}{ccc} 0 & 0 & O_P(T^{-1/2}) \\ 0 & 0 & O_P(T^{-1/2}) \\ O_P(T^{-1/2}) & O_P(T^{-1/2}) & 0 \end{array} \right], \\
\delta_{yz} & := & \left[\begin{array}{ccc} T^{-1}(\la \tilde y_{t,1}^\pi, \tilde z_{t,1}^\pi \ra - \la \tilde z_{t,1}^\pi, \tilde z_{t,1}^\pi \ra)  &
T^{-1}(\la \tilde y_{t,1}^\pi, \tilde z_{t,2}^\pi \ra - \la \tilde z_{t,1}^\pi, \tilde z_{t,2}^\pi \ra ) &
T^{-1/2}\la \tilde y_{t,1}^\pi, \tilde z_{t,3}^\pi \ra \\
T^{-1/2}\la \tilde y_{t,2}^\pi, \tilde z_{t,1}^\pi \ra  &
T^{-1/2}\la \tilde y_{t,2}^\pi, \tilde z_{t,2}^\pi \ra &
0  \end{array} \right] + \tilde D_y \tilde B^+ \tilde D_z \\
& = & \left[\begin{array}{ccc} O_P(T^{-1}) & O_P(T^{-1}) & O_P(T^{-1/2}) \\ O_P(T^{-1/2}) & O_P(T^{-1/2}) & o_P(1) \end{array} \right], \\
\delta_{yy} & := & \tilde D_y \left( 
\la \tilde y_t^\pi, \tilde y_t^\pi \ra +\tilde C^+ 
- \left[ \begin{array}{cc} \la z_{t,1}^\pi , z_{t,1}^\pi \ra & 0 \\ 0 & 
\hYY  \end{array} \right] 
\right) \tilde D_y \\
&  = &  \left[\begin{array}{cc} O_P(T^{-1}) & O_P(T^{-1/2})
 \\ O_P(T^{-1/2}) &  0 \end{array} \right].
\end{eqnarray*}

(II) Let $J := \bar{Q} - \bar \Phi$. To simplify notation define
$Z_{ij} := T^{-1}\la \tilde z_{t,i}^\pi, \tilde z_{t,j}^\pi \ra, i,j=1,2$. Then
\begin{equation} \label{equ:defJ+}
\begin{array}{lll}
J_{i,j} & =&  [\delta_{zy}^{i1}-Z_{i1}Z_{11}^{-1}\delta_{yy}^{11}]Z_{11}^{-1}Z_{1j} +
Z_{i1}Z_{11}^{-1}\delta_{yz}^{1j} \\
& & + [\delta_{zy}^{i2}- Z_{i1}Z_{11}^{-1}\delta_{yy}^{12}] \YY^{-1}
[ \delta_{yz}^{2j}- \delta_{yy}^{21}Z_{11}^{-1}Z_{1j}] + o_P(T^{-1}),\\
J_{3,i} & = & \delta_{zy}^{31}Z_{11}^{-1}Z_{1i}  +
\la \tilde z_{t,3}^\pi, \tilde y_{t,2}^\pi\ra \hYY^{-1} 
[\delta_{yz}^{2i}-\delta_{yy}^{21}Z_{11}^{-1}Z_{1i}] + o_P(T^{-1}),\\
J_{3,3} & = &
\left[ \la \tilde z_{t,3}^\pi, \tilde y_{t,2}^\pi\ra
\hYY^{-1} 
\delta_{yy}^{21} - \delta_{zy}^{31} \right]
Z_{11}^{-1}
\left[ \delta_{yy}^{12} \hYY^{-1} 
\la \tilde y_{t,2}^\pi,
\tilde z_{t,3}^\pi\ra -\delta_{yz}^{13} \right]
 + o_P(T^{-1})
\end{array}
\end{equation}
for $i=1,2, j=1,2$ where expressions for the remaining blocks of $J$ follow from symmetry.
Hence $J_{i,j} = O_P(T^{-1})$ for $i,j=1,2$.
Further $J_{3,i} = O_P(T^{-1/2})$  for $i=1,2,3$.
$J_{3,3} = O_P(T^{-1})$ and $J_{3,3} = O((\log T)^3 /T)$ respectively.

(III) $\delta G = O_P(T^{-1/2})$. 
\end{lemma} 
\begin{proof}
(I) follows from Lemma~\ref{lem:addterms}. Note that compared to the OLS case in the $(3,3)$ entry of $\bar{\Phi}$
the matrix $\la \tilde y_{t,2}^\pi, \tilde y_{t,2}^\pi \ra$ is replaced with $\hYY = \la \tilde y_{t,2}^\pi, \tilde 
y_{t,2}^\pi \ra + o_P(T^{-1/2})$ in order to obtain $\delta_{yy}^{22} =0$ rather than $o_P(T^{-1/2})$.   
Only the in probability statements are used. 

The proof of (II) then is unchanged except that $\delta_{yz}^{2,3} = o_P(T^{-1/2})$ needs to be taken into account. 
(III) is then immediate. 
\end{proof}

Next the proof of Lemma~\ref{lem:asydeltaG} uses only the results of Lemma~\ref{lem:SVD} and equation~\eqref{eq:imgg} in 
combination with the following limit results:
\begin{eqnarray*}
\sqrt{T} [\delta_{yz}^{21} - \delta_{yy}^{21}]'\la \tilde y_{t,2}^\pi, \tilde y_{t,2}^\pi \ra^{-1} & = &
\Xi + o(1), \\
\sqrt{T} [J_{3,1} - \delta_{zz}^{3,1} ] & = & \mE \tilde z_{t,3}^\Pi (\tilde y_{t,2}^\Pi)'\Xi' + o_P(1), \\
T\la \tilde z_{t,2.1}^\pi, \tilde z_{t,2.1}^\pi \ra^{-1} J_{2.1,1} & = & 
(T^{-1} \la \tilde z_{t,2.1}^\pi, \tilde z_{t,2.1}^\pi \ra)^{-1} \left( \la \tilde z_{t,2.1}^\pi, \tilde \ve_{t,1.2}\ra) + \tilde B_{1,2.1}' + \tilde B_{2,2.1}'\Xi' \right)' 
+ o_P(1), \\
\sqrt{T}\la \tilde z_{t,2.1}^\pi, \tilde z_{t,2.1}^\pi \ra^{-1} J_{2.1,3}
& = &
\la \tilde z_{t,2.1}^\pi, \tilde z_{t,2.1}^\pi \ra^{-1} \left( \la \tilde z_{t,2.1}^\pi, \tilde y_{t,2}^\pi +  \ra +  \tilde B_{2.1,2} \right)' \YY^{-1} \tO_2\Gamma_{3,2}'Z_{33} + o_P(1).
\end{eqnarray*}
where again $\Xi := -\mE \tilde \ve_{t,1} \tilde y_{t,2}' \YY^{-1}$ 
and $\tilde \ve_{t,1.2} = \tilde \ve_{t,1} + \Xi \tilde y_{t,2}^\Pi$.
Further ($\tilde B^+$ denoting the transformed quantity $\hat B^+$)
$$
\tilde B_{1,2.1} = [I,0] \tilde B^+ [-Z_{21}Z_{11}^{-1},I,0]', \quad
\tilde B_{2,2.1} = [0,I] \tilde B^+ [-Z_{21}Z_{11}^{-1},I,0]'.
$$ 

Here the first statement follows from Lemma~\ref{lem:addterms}. The second from the fact that the (1,3) block of $\tilde \Sigma_{y,z}$ is of order $o_P(T^{-1/2})$ relating to 
a stationary component of the regressors. The remaining statements follow straightforwardly from the definition of $J$.

Lemma~\ref{lem:defhb} needs to be changed slightly by replacing a.s. statements by the corresponding in probability version.
\begin{lemma} \label{lem:defhbfm}
Let the assumptions of Theorem~\ref{thm:FM} hold.

(I)
Then
\begin{eqnarray*}
{\mathcal T}_y (\hat \beta_{RRR,r}^+ -b_r) {\mathcal T}_z^{-1}D_z^{-1}
 & = & \sqrt{T}\left[ \la \tilde \varepsilon_t, \tilde z_t^\pi \ra + \tilde B^+  \right] \tilde D_z \bar \Gamma \bar \Gamma^\dagger
+ \left[ \begin{array}{cc} TI & 0 \\ 0 & \sqrt{T} O \end{array} \right] \delta G'
( I- \bar M \bar G\bar G^\dagger)\\
& &
+ \left[ \begin{array}{ccc} 0 & 0 & 0 \\ 0 & 0 & \sqrt{T}[\tilde \beta_{2,3}-\tilde b_{2,3}]
 [I- \mE \tilde z_{t,3}^\Pi(\tilde z_{t,3}^\Pi)'
\Gamma_{3,2}\Gamma_{3,2}^\dagger] \end{array} \right]+o_P(1)
\end{eqnarray*}
where
$$
\bar \beta_{2,3} = \la \tilde y_{t,2}^\pi, \tilde z_{t,3}^\pi \ra
\bar \Gamma_{3,2} \bar \Gamma_{3,2}^\dagger = \bar O_2 \bar \Gamma_{3,2}' \to \tO_2\Gamma_{3,2}' = \tilde b_{2,3}
$$
denotes the solution to the subproblem of the problem (b) corresponding to the stationary components.
$\bar\Gamma_{3,2}^\dagger := (\bar\Gamma_{3,2}' \la \tilde z_{t,3}^\pi, \tilde z_{t,3}^\pi \ra \bar\Gamma_{3,2})^{-1}\bar\Gamma_{3,2}'$. 

(II)
Letting $\delta G_{:,1}$ and $\delta G_{:,2}$ denote the first and second block column of
$\delta G$ it holds that
\begin{eqnarray}
T\delta G_{:,1}'(I-\bar M\bar G \bar G^\dagger) & = & \left[ -T\delta H' Z_{21} Z_{11}^{-1} \quad T\delta H' \quad
T\delta G_{3,1}'\bar P \right] + o_P(1), \label{equ:tdeltag+}\\
\sqrt{T}\delta G_{:,2}'(I-\bar M \bar G\bar G^\dagger) & = & \sqrt{T}\delta G_{2,2}'
\left[ - Z_{21} Z_{11}^{-1} \quad I \quad 0 \right] + o_P(1) \label{equ:sqrttdeltag+}
\end{eqnarray}
where $\delta H = \delta G_{2,1} - \delta G_{2,2}(\Gamma_{3,2}^\dagger)'\mE \tilde z_{t,3}^\Pi (\tilde z_{t,3}^\Pi)'\delta G_{3,1}$
and $\bar P = I - Z_{33} \bar \Gamma_{3,2}\bar \Gamma_{3,2}^\dagger$.
\end{lemma}

The only change in the proof consists in exchanging the estimation error for the OLS estimator by the estimation error for the 
FM estimator in (I). The rest of the proof is analogously to the OLS case and hence omitted. Primarily the orders of 
convergence derived above are used.

It remains to analyze the asymptotic distribution of the various terms. This is done in the analogon to Lemma~\ref{lem:asydist}:

\begin{lemma}\label{lem:asydistfm}
With $\tO_2^\dagger =
(\tO_2' (\mE \tilde y_{t,2}^\Pi (\tilde y_{t,2}^\Pi)')^{-1}\tO_2)^{-1} \tO_2'(\mE \tilde y_{t,2}^\Pi(\tilde y_{t,2}^\Pi)')^{-1}$ we have
\begin{eqnarray*}
(\la \tilde \ve_{t}, \tilde z_{t,1}^\pi \ra + \tilde B_{:,1}^+) Z_{11}^{-1} & \stackrel{d}{\to} & f({\mathcal T}_y\Lambda B, W_{z,1}^\Pi,0), \\
\sqrt{T}\la \tilde \ve_t , \tilde z_{t,3}^\pi \ra & \stackrel{d}{\to} & {\cal N}(0,V), \\
\sqrt{T}\delta G_{2,2} & = & Z_{22.1}^{-1}( \la \tilde \ve_{t,2}^\pi, \tilde z_{t,2.1}^\pi \ra + \tilde B_{2,2.1} )' (\tO_2^\dagger)' + o_P(1) \stackrel{d}{\to} M_{2,+}'(\tO_2^\dagger)',\\
T\bar P'\delta G_{3,1} & = & Z_{33}^{-1}P' \sqrt{T}\la \tilde z_{t,3}^\pi, \tilde \ve_{t,1.2}\ra  
+ Z_{33}^{-1}\sqrt{T}(\bar P-P)' \mE \tilde z_{t,3}^\Pi (\tilde y_{t,2}^\Pi)'\Xi' 
+ o_P(1) \to \tilde R',\\
T\delta H  & = & Z_{22.1}^{-1} 
\left[ \la \tilde z_{t,2.1}^\pi , \tilde \ve_{t,1} \ra + \tilde B_{1,2.1}' +
\left(  \la \tilde z_{t,2.1}^\pi , \tilde \ve_{t,2}^\pi \ra + \tilde B_{2,2.1}' \right)  (I-\tO_2 \tO_2^\dagger )'\Xi' \right] + o_P(1) \\
& \stackrel{d}{\to} & N_+',\\
\sqrt{T}[ \bar \beta_{2,3} - \tilde \beta_{2,3} ]P & = &
\tO_2 \tO_2^\dagger  \sqrt{T}\la \tilde \ve_{t,2}, \tilde z_{t,3}^\Pi \ra
(\mE \tilde z_{t,3}^\Pi (\tilde z_{t,3}^\Pi)')^{-1}P+ o_P(1),
\end{eqnarray*}
where $P = (I - \mE \tilde z_{t,3}^\Pi (\tilde z_{t,3}^\Pi)' \Gamma_{3,2}\Gamma_{3,2}^\dagger)$.
Here $M_{2,+} = f([0,I]{\mathcal T}_yB,W_{z,2}^\Pi - Y_{21}Y_{11}^{-1}W_{z,1}^\Pi,0)$
and $N_+ = f([[I,0]+\Xi(I-O_2O_2^\dagger)[0,I]]{\mathcal T}_y B,W_{z,2}^\Pi - Y_{21}Y_{11}^{-1}W_{z,1}^\Pi,0)$.
$T\bar P\delta G_{3,1}$ and $\sqrt{T}[ \tilde \beta_{2,3} - \beta_2 ]P$ converge in distribution
to Gaussian random variables with mean zero.
\end{lemma}

The proof follows analogously to the OLS case. The remaining steps of the proof are analogous to the proof for Theorem~\ref
{thm:RRR} and hence omitted.

\newpage 

\section{Collection of notation} \label{sec:notation}
In this section the notation is presented in order to make reference easier. The general concept is to use lower case letters 
for processes (where $y$ is reserved for the dependent variable, $z$ denotes regressors and $v,w,u$ is reserved for stationary processes, 
$\ve, \eta$ denote white noise). Processes built using a number of coordinates of other processes are indicated using sub- or 
superscripts. Regression residuals are indicated using a superscript $\pi$ (where the regressors are clear from the context) 
and their corresponding limits with a superscript $\Pi$ (this notation is only used, if limits exist).

Upper case letters are used for matrices. Matrices that transform the basis of processes are indicated using ${\mathcal T}$ 
where the transformed process is indicated as a subscript. Scaling matrices that are introduced in order to ensure the 
convergence of matrices are denoted using $D$ with subscripts denoting the processes to which they are applied. 

Estimates in the original basis are denoted using a $\hat \bullet$, in the transformed basis (see Theorem~\ref{thm:repr}) with a
$\tilde \bullet$ and in the transformed basis with appropriate scaling ensuring convergence with a $\bar \bullet$ or $\check \bullet$ respectively. 

\subsection{Processes}
Below $a_t, b_t$ are used to denote arbitrary processes, where the notation applies to a number of different processes. 
\begin{eqnarray*}
y_t & = & b_r z_t^r + b_u z_t^u + L\ve_t, \\
z_t & = & \left[ \begin{array}{c} z_t^r \\ z_t^u \end{array}\right], 
\mbox{diag}(\Delta,I)H_r' z_t^r = v_t, 
\mbox{diag}(\Delta,I)H_u' z_t^u = w_t, \\
\nu_t & = & \left[ \begin{array}{c} v_t \\ w_t \end{array}\right] = c(z)\ve_t, c(0)=0, \det c(1) \ne 0, \\
\la a_t, b_t \ra & = & T^{-1}\sum_{t=1}^T a_t b_t', \\
a_t^\pi & = & a_t - \la a_t, z_t^u \ra \la z_t^u, z_t^u \ra^{-1} z_t^u \left( \to a_t^\Pi \mbox{(if convergent)}\right), \\
\tilde y_t & = & {\mathcal T}_y (y_t - b_u z_t^u) = \tilde b_r \tilde z_t +\tilde \ve_t = \left[\begin{array}{ccc} I & 0 & 0 \\ 0 & 0 & \tilde b_{2,3} \end{array} \right] \left[ \begin{array}{c} \tilde z_{t,1} \\ \tilde z_{t,2} \\ \tilde z_{t,3} \end{array}\right] + \left[ \begin{array}{c} \tilde \ve_{t,1} \\ \tilde \ve_{t,2}  \end{array}\right], \\
\tilde z_t & = & {\mathcal T}_{z,r} z_t^r=  \left[ \begin{array}{c} \tilde z_{t,1} \\ \tilde z_{t,2} \\ \tilde z_{t,3} \end{array}\right],  \\
\Delta \tilde z_{t,1} & = & \tilde c_{z,1}(z)\ve_t, 
\Delta \tilde z_{t,2} = \tilde c_{z,2}(z)\ve_t, 
\tilde z_{t,3} = \tilde c_{z,3}(z)\ve_t \quad \mbox{stationary} \\
\tilde z_t^u & = & {\mathcal T}_{z,u}z_t^u =  \left[ \begin{array}{c} \tilde z_{t,1}^u \\ \tilde z_{t,2}^u  \end{array}\right], \\
\Delta \tilde z_{t,1}^u & = & \tilde c_{u,1}(z)\ve_t, 
\tilde z_{t,2}^u = \tilde c_{u,2}(z)\ve_t \quad \mbox{stationary}, \\
\left[\begin{array}{c}  \tilde c_{z,1}(1) \\ \tilde c_{z,2}(1) \\ \tilde c_{u,1}(1) \end{array} \right] & & \mbox{is of full row rank}. 
\end{eqnarray*}

\subsection{Matrices} 

\begin{itemize}
	\item[$b_r$] $\in \mR^{s \times m_r}$ coefficient matrix corresponding to $z_t^r$
	\item[$b_u$] $\in \mR^{s \times m_u}$ coefficient matrix corresponding to $z_t^u$
	\item[$b$] $=[b_r,b_u]$ coefficient matrix corresponding to $z_t$.
	\item[${\mathcal T}_y$] $\in \mR^{s \times s}$ used to transform $y_t$ into $\tilde y_t$ separating stationary from nonstationary terms. 
	\item[${\mathcal T}_{z,r}$] $\in \mR^{m_r \times m_r}$ used to transform $z_t^r$ into $\tilde z_t$ separating stationary from nonstationary terms. 
	\item[${\mathcal T}_{z,u}$] $\in \mR^{m_u \times m_u}$ used to transform $z_t^u$ into $\tilde z_t^u$ separating stationary from nonstationary terms. 
	\item[${\mathcal T}_z$] $= \mbox{diag}({\mathcal T}_{z,r},{\mathcal T}_{z,u})$.
	\item[$\tilde b_r$] $= {\mathcal T}_y b_r {\mathcal T}_{z,r}^{-1}$. 
	\item[$\tilde b_u$] $= {\mathcal T}_y b_u {\mathcal T}_{z,u}^{-1}$.
	\item[$\hat \beta_{OLS}$] OLS estimator of $\beta$.
	\item[$\hat \beta_{OLS,r}$] OLS estimator of $b_r$
	\item[$\hat \beta_{OLS,u}$] OLS estimator of $b_u$
	\item[$\tilde \beta_{OLS}$] OLS estimator of $\tilde b$.
	\item[$\tilde \beta_{OLS,r}$] OLS estimator of $\tilde b_r$
	\item[$\tilde \beta_{OLS,u}$] OLS estimator of $\tilde b_u$
	\item[$\hat \beta_{RRR}$] RRR estimator of $b$.
	\item[$\hat \beta_{RRR,r}$] RRR estimator of $b_r$
	\item[$\hat \beta_{RRR,u}$] RRR estimator of $b_u$
	\item[$\tilde \beta_{RRR}$] RRR estimator of $\tilde b$.
	\item[$\tilde \beta_{RRR,r}$] RRR estimator of $\tilde b_r$
	\item[$\tilde \beta_{RRR,u}$] RRR estimator of $\tilde b_u$	
	\item[$\hat \Xi_+$] $\in \mR^{s \times s}$ weighting matrix. 
	\item[$\Xi$] $= - \mE \tilde \ve_{t,1}\tilde y_{t,2}'(\mE \tilde y_{t,2}^{\Pi}\tilde y_{t,2}')^{-1}$
	\item[$M_r$] $= f(W,W_z^{\Pi})$ where $W$ denotes the Brownian motion corresponding to $(\ve_t)_{t \in \mN}$, $W_z = \tilde c_{1:2}(1)W, W_u = \tilde c_{u,1}(1)W, W_z^\Pi = W_z - \int W_zW_u'(\int W_uW_u')^{-1}W_u$. Further $f(W_1,W_2) = \int dW_1W_2(\int W_2W_2')^{-1}$. 
	\item[$M_u$] $=f(W,W_u)$. 
	\item[$N_r$] $=\int W_zW_u' (\int W_uW_u')^{-1}$. 
	\item[$M_{r,2}$] $=f([0,I]{\mathcal T}_yW,W_{z,2}^\Pi-Y_{21}Y_{11}^{-1}W_{z,1}^\Pi)$.
	\item[$Y_{i1}$] $=\int W_{z,i}^\Pi (W_{z,1}^\Pi)'$.
	\item[$P$] $= I - \mE \tilde z_{t,3}^\Pi (\tilde z_{t,3}^\Pi)'\Gamma_{3,2}\Gamma_{3,2}^\dagger, \Gamma_{3,2}^\dagger = (\Gamma_{3,2}' \mE \tilde z_{t,3}^\Pi (\tilde z_{t,3}^\Pi)' \Gamma_{3,2})^{-1}\Gamma_{3,2}$.
	\item[$\bar P_{3,3}$] $=  \la \tilde z_{t,3}^\pi , \tilde z_{t,3}^\pi \ra - \la \tilde z_{t,3}^\pi , \tilde z_{t,3}^\pi \ra
\bar \Gamma_{3,2} \bar \Gamma_{3,2}^\dagger \la \tilde z_{t,3}^\pi , \tilde z_{t,3}^\pi \ra$.
	\item[$D_{y}$] $= \mbox{diag}(T^{-1}I_{c_y},T^{-1/2})$ proper scaling for $\tilde y_t^\pi$.
	\item[$D_{z,r}$] $= \mbox{diag}(T^{-1}I_{c_r},T^{-1/2})$ proper scaling for $\tilde z_t^\pi$.
	\item[$D_{z,u}$] $= \mbox{diag}(T^{-1}I_{c_u},T^{-1/2})$ proper scaling for $\tilde z_t^u$.
	\item[$D_z$] $=\mbox{diag}(D_{z,r},D_{z,u})$
	\item[$\tilde D_z$] $= D_z T^{1/2} = \mbox{diag}(T^{-1/2}I,I), \quad \tilde D_y =  D_y T^{1/2} = \mbox{diag}(T^{-1/2}I,I)$
\item[$\bar{G}$] $= \tilde D_z^{-1} \hat G$
\item[$\bar Q $] $= \la \tilde D_z \tilde z_t^\pi , \tilde D_y \tilde y_t^\pi \ra
\la \tilde D_y \tilde y_t^\pi , \tilde D_y \tilde y_t^\pi \ra^{-1}
\la \tilde D_y \tilde y_t^\pi , \tilde D_z \tilde z_t^\pi \ra$
\item[$\bar M$] $= \la \tilde D_z \tilde z_t^\pi , \tilde D_z \tilde z_t^\pi \ra$
\item[$\bar \Phi$] $= \left[ \begin{array}{ccc} T^{-1}\la \tilde z_{t,1}^\pi, \tilde z_{t,1}^\pi\ra & T^{-1}\la \tilde z_{t,1}^\pi, \tilde z_{t,2}^\pi\ra & 0 \\
T^{-1}\la \tilde z_{t,2}^\pi, \tilde z_{t,1}^\pi\ra & T^{-1}\la \tilde z_{t,2}^\pi, \tilde z_{t,1}^\pi\ra \la \tilde z_{t,1}^\pi, \tilde z_{t,1}^\pi\ra^{-1}\la \tilde z_{t,1}^\pi, \tilde z_{t,2}^\pi\ra & 0 \\
0 & 0 & \la \tilde z_{t,3}^\pi, \tilde y_{t,2}^\pi\ra \la \tilde y_{t,2}^\pi, \tilde y_{t,2}^\pi\ra^{-1}\la \tilde y_{t,2}^\pi, \tilde z_{t,3}^\pi\ra
\end{array} \right]$
\item[$\bar \Psi$] $=\left[ \begin{array}{ccc} T^{-1}\la \tilde z_{t,1}^\pi, \tilde z_{t,1}^\pi\ra & T^{-1}\la \tilde z_{t,1}^\pi, \tilde z_{t,2}^\pi\ra & 0 \\
T^{-1}\la \tilde z_{t,2}^\pi, \tilde z_{t,1}^\pi\ra & T^{-1}\la \tilde z_{t,2}^\pi, \tilde z_{t,2}^\pi\ra & 0 \\
0 & 0 & \la \tilde z_{t,3}^\pi, \tilde z_{t,3}^\pi\ra
\end{array} \right]$
\item[$J$] $=\bar Q - \bar \Phi$.
\item[$Z_{ij}$] $= T^{-1}\la \tilde z_{t,i}^\pi,  \tilde z_{t,j}^\pi \ra, j=1,2$. 
\item[$\delta_{zz}$] $= \bar{M}-\bar{\Psi}$ 
\item[$\delta_{yz}$] $= \left[\begin{array}{ccc} T^{-1}(\la \tilde y_{t,1}^\pi, \tilde z_{t,1}^\pi \ra - \la \tilde z_{t,1}^\pi, \tilde z_{t,1}^\pi \ra)  &
T^{-1}(\la \tilde y_{t,1}^\pi, \tilde z_{t,2}^\pi \ra - \la \tilde z_{t,1}^\pi, \tilde z_{t,2}^\pi \ra ) &
T^{-1/2}\la \tilde y_{t,1}^\pi, \tilde z_{t,3}^\pi \ra \\
T^{-1/2}\la \tilde y_{t,2}^\pi, \tilde z_{t,1}^\pi \ra  &
T^{-1/2}\la \tilde y_{t,2}^\pi, \tilde z_{t,2}^\pi \ra &
0  \end{array} \right] $
\item[$\delta_{yy}$] $= \left[\begin{array}{cc} T^{-1}(\la \tilde y_{t,1}^\pi, \tilde y_{t,1}^\pi \ra - \la \tilde z_{t,1}^\pi, \tilde z_{t,1}^\pi \ra)  &
T^{-1/2}\la \tilde y_{t,1}^\pi, \tilde y_{t,2}^\pi \ra \\
T^{-1/2}\la \tilde y_{t,2}^\pi, \tilde y_{t,1}^\pi \ra  &
0  \end{array} \right]$
\item[$\bar S$] $= (\la \tilde z_{t,3}^\pi, \tilde z_{t,3}^\pi \ra - \la \tilde z_{t,3}^\pi, \tilde y_{t,2}^\pi \ra
\la \tilde y_{t,2}^\pi, \tilde y_{t,2}^\pi \ra^{-1}\la \tilde y_{t,2}^\pi, \tilde z_{t,3}^\pi \ra)^{-1}$.
\item[$\bar \Gamma^\dagger$] $= (\bar \Gamma' \bar \Psi \bar \Gamma)^{-1} \bar \Gamma$
\item[$\bar G^\dagger$] $= (\bar G'  \bar M \bar G)^{-1} \bar G$
\item[$\delta G$] $ = \bar G - \bar \Gamma$
\item[$\tO_2^\dagger$]  $=  (\tO_2'(\mE \tilde y_{t,2}^\Pi\tilde y_{t,2}')^{-1}O_2)^{-1}\tO_2'(\mE \tilde y_{t,2}^\Pi\tilde y_{t,2}')^{-1}$.

\end{itemize}

\subsection{Singular value decompositions}

\begin{table}[t]
\begin{tabular}{|cc|c|}
\hline
 Original 
& Formula & 
$\la z_{t}^{r,\pi} , y_{t}^{\pi} \ra \la y_{t}^{\pi} , y_{t}^{\pi} \ra^{-1}  \la y_{t}^{\pi} , z_{t}^{r,\pi} \ra \hat G = \la z_{t}^{r,\pi} , z_{t}^{r,\pi}\ra \hat G \hat R^2$  \\
\hline
 Transformed 
& Formula & $\la \tilde z_{t}^{\pi} , \tilde y_{t}^{\pi} \ra \la \tilde y_{t}^{\pi} , \tilde y_{t}^{\pi} \ra^{-1}\la \tilde y_{t}^{\pi} , \tilde z_{t}^{\pi} \ra \tilde  G = \la \tilde z_{t}^{\pi} , \tilde z_{t}^{\pi}\ra \tilde  G \tilde R^2$ \\
& Relations & $\tilde G = {\mathcal T}_{z,r}^{-1}\hat G$ \\
\hline
 Scaled 
& Formula & $\bar Q \bar G = \bar M \bar G \bar R^2$ \\ 
& Restrictions & $[I,0] \bar G_{:,1} = I, S_{p,2}' \bar G_{:,2}= I, \bar R^2 = \mbox{diag}(\bar R_1^2,\bar R_2^2)$  \\
& Relations & $\bar G = \tilde D_z^{-1} \tilde G$ \\
\hline
 Decoupled 
& Formula & $\bar \Phi \bar \Gamma = \bar \Psi \bar \Gamma \bar \Theta^2$\\ 
& Restrictions & $\bar \Gamma' S_p = I, \bar \Theta^2 = \mbox{diag}(I,\bar \Theta_2^2)$, \\
& Relations & $\bar \Gamma = \left[ \begin{array}{cc} 
I & 0 \\ 0 & 0 \\ 0 & \bar \Gamma_{3,2} \end{array}\right]$ \\
\hline
 Stationary subproblem 
& Formula & 
$\la \tilde z_{t,3}^{\pi} , \tilde y_{t,2}^{\pi} \ra \la \tilde y_{t,2}^{\pi} , \tilde y_{t,2}^{\pi} \ra^{-1}\la \tilde y_{t,2}^{\pi} , \tilde z_{t,3}^{\pi} \ra \bar \Gamma_{3,2} = \la \tilde z_{t,3}^{\pi} , \tilde z_{t,3}^{\pi}\ra \bar \Gamma_{3,2} \bar \Theta_2^2$ \\
& Restrictions & 
$\bar \Gamma_{3,2}'S_{p,22} = I$\\ 
& Relations &  converges to $\tilde b_{2,3} = \tO_2\Gamma_{3,2}'$. 
\\
\hline

\end{tabular}
\caption{Singular value decompositions used in the article.}
\end{table}

\end{document}